\documentclass[a4paper,11pt
]{amsart}

\usepackage{amssymb,amsthm,bbm,epic,eepic,graphics,amsbsy,url}
\usepackage[hyperindex,breaklinks,pdftex]{hyperref}
\usepackage{a4wide}
\usepackage{amssymb,amsthm}
\usepackage{graphics}
\usepackage{verbatim}
\usepackage{multirow}
\usepackage[all,  curve, frame, arc, color,
]{xy}
    \xyoption{tile}
\usepackage[applemac]{inputenc}
\usepackage{longtable}
\usepackage{array}
\newcommand{\la}{\lambda}
\def\md{\mathrm{md}}

\DeclareMathOperator{\GL}{\mathrm{GL}}
\DeclareMathOperator{\Hom}{\mathrm{Hom}}
\def\Ker{\mathrm{Ker}\,}

\def\C{\ensuremath{\mathbbm{C}}}
\def\Q{\mathbbm{Q}}
\def\Z{\mathbbm{Z}}
\def\N{\mathbbm{N}}
\def\R{\mathbbm{R}}
\def\F{\mathbbm{F}}
\def\bar{\overline}

\def\eps{\epsilon}
\def\onto{\twoheadrightarrow}
\def\into{\hookrightarrow}
\def\ii{\mathrm{i}}

\def\lcm{\mathrm{lcm}}

\newtheorem{theo}{Theorem}[section]

\newtheorem{prop}[theo]{Proposition}
\newtheorem{defi}[theo]{Definition}
\newtheorem{lemma}[theo]{Lemma}
\newtheorem{remark}[theo]{Remark}
\newtheorem{cor}[theo]{Corollary}

\newcommand{\Ax}[0]{\mathbf{A}}
\newcommand{\Bx}[0]{\mathbf{B}}

\newcommand{\Art}[0]{\mathrm{A}}
\newcommand{\B}[0]{\mathrm{B}}

\newcommand{\DA}[0]{\partial}
\newcommand{\DB}[0]{\overline{\partial}}

\providecommand{\Br}{{\mathrm{Br}}}

\providecommand{\B}{{\mathrm{B}}}

\newcommand{\pmu}[0]{{\pm 1}}
\newcommand{\Cal}[1]{\mathcal{#1}}

\newcommand{\ph}[0]{\varphi}
\newcommand{\too}[0]{\longrightarrow}
\newcommand{\sst}[0]{\subset}

\newcommand{\qbin}[2]{ \left[ \! \! \! \begin{array}{c} #1 \\ #2 \end{array} \! \! \! \right] }

\newenvironment{dummyenvironment}[0]{}{}

\makeatletter
\title{{\bf Homology computations for complex braid groups}}
\author{Filippo Callegaro  \& Ivan Marin}
\date{November 18, 2010}
\begin{document}

\maketitle

\bigskip
\begin{center}
Scuola Normale Superiore \\
Piazza dei Cavalieri, 7 \\
56126 Pisa \\
Italy \\
--- \\
Institut de Math\'ematiques de Jussieu \\
Universit\'e Paris 7 \\
175 rue du Chevaleret \\
75013 Paris \\
France
\end{center}
\bigskip

\bigskip

%
%
%
%
%
%
%
%

\bigskip

\noindent {\bf Abstract.} 
Complex braid groups are the natural generalizations of braid
groups associated to arbitrary (finite) complex reflection groups. We
investigate several methods for computing the homology of
these groups. In particular, we get the Poincar\'e polynomial
with coefficients in a finite field for one large series of such groups,
and compute the second integral cohomology group for all of them. 
As a consequence we get non-isomorphism results for these groups.

\medskip

 

\def\RR{\mathcal{R}}
\def\SS{\mathcal{S}}
\def\HH{\mathcal{H}}
\def\BB{\mathcal{B}}
\def\XX{\mathrm{X}}


\tableofcontents


\section{Introduction}

\subsection{Presentation}
The aim of this paper is to provide homological
results and tools for the generalized braid groups associated
to complex (pseudo-)reflection groups. Recall that
a complex reflection group $W$ is a finite subgroup of some
$\GL_r(\C)$ generated by (pseudo-)reflections, namely
finite-order endomorphisms of $\GL_r(\C)$ which leave
invariant some hyperplane in $\C^r$. The collection $\mathcal{A}$
of the hyperplanes associated to the reflections of $W$ is a
central hyperplane arrangement in $\C^r$. We let $X = \C^r \setminus
\bigcup \mathcal{A}$ denote the corresponding hyperplane complement.
The generalised braid group $B = \pi_1(X/W)$ is an extension
of $W$  by $P = \pi_1(X)$. When $W$ is a finite Coxeter
group, $B$ is an Artin group of finite Coxeter type.

Every $W$ can be decomposed as a direct product of
so-called irreducible groups (meaning that their natural
linear action is irreducible), and $B$ decomposes
accordingly. For homological purposes, by K\"unneth formula we can
thus assume that $W$ is irreducible.

The irreducible complex reflection groups have
been classified in 1954 by Shephard and Todd (\cite{shep-todd}): there is an infinite series
$G(de,e,r)$ with three integer parameters, and 34 exceptions,
labelled $G_4,\dots,G_{37}$. Their braided counterparts
however are far less understood. It is for instance an open problem
to decide the lack of injectivity of $W \mapsto B$. Indeed,
two reflection groups $W$ can provide the same $B$ (up to isomorphism),
to the extent that all possible $B$ arise from the 2-reflection
groups, that is complex reflection groups $W$ with reflections of
order 2. 

Recall that $X$ and $X/W$ are $K(\pi,1)$-spaces by work of \cite{fa_ne, briesk, deligne, BANNAI, naka, OT, BESSISKPI1}. From this,
by general arguments, one can however prove that both
the rank $r$ of $W$ and the number $|\mathcal{A}/W|$ of $W$-orbits in $\mathcal{A}$
is detected by $B$ :
\begin{prop} 
The homological dimension of $B$ is equal to the rank of $W$. \\
$H_1(B,\Z)$ is a free module of dimension $|\mathcal{A}/W|$.
\end{prop}
\begin{proof}
It is known that $X/W$ is an affine variety of (complex) dimension $r$, it is homotopically
equivalent to a finite CW-complex of dimension $r$. Moreover,
the $r$-th cohomology with trivial coefficients
of $P= \Ker(B \onto W)$ is nonzero. Indeed,
the Poincar\'e polynomial of $X$ is $(1 + c_1 t)\dots (1+ c_r)t$
(see \cite{OT} cor. 6.62) 
where the $c_i$ are positive numbers,
called the co-exponents of $W$. In particular, we have
$H^r(P,\Q) = H^r(X,\Q) \neq 0$, and since $P < B$ which implies that
$B$ has homological dimension at least $r$, hence exactly $r$,
which proves the first part. The second part is proved in \cite{BMR}.
\end{proof}

As opposed to the case of Artin groups of finite Coxeter type,
for which there are uniform `simplicial' theories and homological
methods, it seems that different methods have to be used in
order to deal with these complex braid groups in general. Due
to some of the coincidences mentioned above, the groups $B$
provided by the 3-parameters series $G(de,e,r)$ actually arise
from two a priori disjoint series with 2 parameters
$G(2e,e,r)$ and $G(e,e,r)$ of 2-reflection groups. The
corresponding braid groups $\B(2e,e,r) = \B(de,e,r)$ for $d > 1$ and
$\B(e,e,r)$ seem to belong to distinct worlds. The first ones
can be better understood as subgroups of the usual braid groups,
or semidirect products of $\Z$ with an Artin group of affine
type, whereas the second ones might be better understood
as the group of fractions of suitable monoids with similar
(Garside) properties than the usual braid group; it should be
noted for instance that the groups $G(e,e,r)$ are generated by $r$
reflections, hence belong to the class of well-generated groups,
for which there is a uniform generalization of
the Garside approach (see \cite{BESSISKPI1}). 
Moreover, using a
specific Garside monoid recently introduced
by Corran and Picantin in \cite{corpic} 
for dealing with the groups
$\B(e,e,r)$, our work on parabolic subgroups suggests that the filtrations classically
used in the homology computations for usual braid groups
might well be extended to this more general setting.

Before proceeding to the exposition of our main results, we recall
the results obtained earlier by G. Lehrer on the \emph{rational} homology of $B$
for the general series.

\begin{theo}[\cite{lehr04}] \label{t:lehr04} 
The Poincar\'e polynomial for the cohomology $H^*(\B(e,e,r);\Q)$ is
$$
P(\B(e,e,r),t)= \left\{
\begin{array}{cl}
1+t & \mbox{ if either } e \mbox{ or } r \mbox{ is odd,} \\
1 + t+ t^{r-1} + t^r & \mbox{ otherwise.}
\end{array}
\right. 
$$

The Poincar\'e polynomial for the cohomology $H^*(\B(2e,e,r);\Q)$ is
$$
P(\B(2e,e,r),t)= \left\{
\begin{array}{cl}
(1+t)(1+t+t^2+ \cdots + t^{r-1}) & \mbox{ if either } e \mbox{ or } r \mbox{ is odd,} \\
(1+t)(1+t+t^2+ \cdots + t^{r-1}) + (t^{r-1} +t^r) & \mbox{ otherwise.}
\end{array}
\right. 
$$\qed
\end{theo}

\subsection{Main results}

By combining several methods, we are able
to compute the low-dimensional integral homology of these
groups. We use the notation $\Z_n = \Z/n\Z$.

First consider the case of the $\B(e,e,r)$. The case $r=2$ is when $G(e,e,2)$ is a dihedral group,
and this case is known by \cite{SALVETTI} : we have $H_2(B,\Z) = 0$
if $e$ is odd, $H_2(B,\Z) = \Z$ if $e$ is even.

In section \ref{s:beer} we prove the following result, by
using a complex defined by Dehornoy and Lafont for Garside monoids in \cite{dehpar}
and a convenient monoid defined by Corran and Picantin in \cite{corpic} for the groups
$G(e,e,r)$ (of which we prove some additional properties) :

\begin{theo}[Theorem \ref{H2BEER}] Let $B = \B(e,e,r)$ with $r \geq 3$.
\begin{itemize}
\item When $r = 3$, $H_2(B,\Z) \simeq \Z_e$
\item When $r = 4$ and $e$ is odd, $H_2(B,\Z) \simeq \Z_e \times \Z_2 \simeq \Z_{2e}$
\item When $r = 4$ and $e$ is even, $H_2(B,\Z) \simeq \Z_e \times \Z_2^2$
\item When $r \geq 5$, $H_2(B,\Z) \simeq \Z_e \times \Z_2$.
\end{itemize}
\end{theo}

In section \ref{s:b2eer}, Theorem \ref{prop:b2eerF2} and 
Theorem \ref{prop:b2eerFp}, we compute the homology of complex braid groups of type
$\B(2e,e,r)$ with coefficients in a finite field, using filtrations of the
Salvetti complex for the Artin group of type $\Bx_r $. With a little additional
computation (see section \ref{ss:higher}), we prove as a corollary:

\begin{theo} \label{H2B2EER} Let $B = \B(2e,e,r)$ with $r \geq 2$.
\begin{itemize}
\item When $r = 2$ and $e$ is odd, $H_2(B,\Z) \simeq \Z$
\item When $r = 2$ and $e$ is even, $H_2(B,\Z) \simeq \Z^2$
\item When $r = 3$, $H_2(B,\Z) \simeq \Z^2$
\item When $r = 4$ and $e$ is odd, $H_2(B,\Z) \simeq \Z^2 \times \Z_2$
\item When $r = 4$ and $e$ is even, $H_2(B,\Z) \simeq \Z^2 \times \Z_2^2$
\item When $r \geq 5$, $H_2(B,\Z) \simeq \Z^2 \times \Z_2$
\end{itemize}
\end{theo}

We also get a stabilization property for the groups $\B(2e,e,r)$
similar to the classical one for the usual braid groups (see Corollaries
\ref{c:stabBeerF2}, \ref{c:stabBeerFp}); it turns out that the stable homology
does not depend on $e$, and is thus the same as the stable homology for the Artin group
of type $\Bx$.
Unfortunately, these computations
do not suffice in general to get the full \emph{integral} homology
groups. Indeed, we show in section \ref{ss:higher} that, contrary to what
happens for Artin groups, the integral homology groups may contain $p^2$
torsion. This phenomenon appears for the exceptional groups
as well.

The reader will notice that the cell complex that we use for the $G(e,e,r)$,
obtained by combining the Dehornoy-Lafont complex and the Corran-Picantin
monoid, share similarities with the Salvetti complex, and actually specializes
to it, for the usual braid group, in the case $e=1$. It is then
likely that this complex can be filtered by a chain of parabolic subcomplexes,
paving the way to the methods we use here for the groups $\B(2e,e,r)$
in order to get the higher homology groups. 
The differential of the complex is inherited from the work of Kobayashi in \cite{koba}.
The problem
is that the behaviour of this differential under the simplest operations, like taking
the direct product of two monoids or restricting to a parabolic submonoid,
is not yet understood.
As a consequence, plausible analogues of formulas of the form
`$\overline{\partial 0} A = (\partial A) B + (-1)^{|A|} A0 (\partial B)$'
(see   section \ref{ss:b2eerintro}) are hard to prove.

In section \ref{s:low} we compute the integral homology for
all exceptional groups, except for $G_{34}$, for which we are able to compute only
$H_2(B,\Z)$ (see Table \ref{tableexc}). As a consequence, we get a complete determination of the
groups $H_2(B,\Z)$ for all complex braid groups. Notice that,
since $H_1(B,\Z)$ is a finitely generated free $\Z$-module,
$H_2(B,\Z)$ determines the cohomology group
$H^2(B,\C^{\times}) \simeq \Hom(H_2(B,\Z),\C^{\times})$,
which contains the relevant obstruction classes to the linearization
of the projective representations of $B$ -- and thus deserves
the name `Schur multiplier' usually restricted to the theory of finite groups.
We show in section \ref{ss:H2} that the Schur multiplier of $B$ always contains
the Schur multiplier of $W$, and that this latter group can most of the time be identified
to the 2-torsion subgroup of $H_2(B,\Z)$.

Finally, at least when $W$ has one conjugacy class of hyperplanes, there
is a well-defined sign morphism $\eps : W \to \{ \pm 1 \}$
and a corresponding sign representation $\Z_{\eps}$. We determine
in general the group $H_1(B,\Z_{\eps})$, which is closely
related to the abelianization of the group $\Ker \eps$ of `even braids',
whose structure remains largely unexplored in general.

\begin{remark} 
It should be noted that even the \emph{rational} homology
is not yet known for $W = G_{34}$, due to the large size of $W$ and of its large rank. For instance, formulas involving
the lattice (like
\cite{OT} cor. 6.17) seem to fail because of the size of the hyperplane
arrangement. The methods
of \cite{lehr04} could lead to the (possibly computer-aided) counting of points in some $\F_p^6$,
but only if we can get a nice form of the discriminant equation,
for which we are able to decide which primes $p$ do satisfy the
arithmetic-geometric requirements of \cite{lehr04}.
As far as we know, this problem has not been settled yet.
Another method would be to use \cite{lehr95}, which provides
information on $H^*(P,\Q)$ as a $G_{34}$-module. Finally, the methods of
\cite{lehr95} enables to compute the trace of the reflections and
of regular elements on this module, but it is so huge (the
Poincar\'e polynomial of $P$ is $1+126t+6195t^2+148820 t^3 + 1763559 t^4 + 8703534 t^5 + 7082725 t^6$)
that this does not allow to determine the dimensions of the invariant
subspaces leading to $H^*(B,\Q)$.
\end{remark}

\subsection{Distinction of complex braid groups}

As we noticed before, we can assume that $W$ is a 2-reflection group.
We recall that, under the Shephard-Todd parametrization,
we have the duplication $G(1,1,4) = G(2,2,3)$. Also notice
that the groups $B$ originating from irreducible groups $W$
should be distinguishable from the groups originating from
non-irreducible ones by the property $Z(B) = \Z$ -- this assertion
for $W = G_{31}$ still being conjectural.

It has been noted by Bannai that $G_{13}$ and $G(6,6,2)$ have the
same braid group, and that the $\B(2e,e,2)$ depend only on
the parity of $e$. In \cite{BANNAI} it is stated without proof
(see remark 6 there) that these are the only coincidences in rank 2.
We provide a proof that uses our computations.

\begin{prop} On irreducible 2-reflection groups of rank 2, the Bannai isomorphisms
are the only coincidences under $W \mapsto B$.
\end{prop}
\begin{proof}
According to our results, $H_2(B,\Z)$ is a free $\Z$-module
of rank 0,1 or 2. The case $H_2(B,\Z) = \Z^2$ holds only
for the $\B(2e,e,2)$ with $e$ even. If $H^2(B,\Z) = \Z$,
then either it is $\B(2e,e,r)$, or it is a group $\B(e,e,2)$
with $e$ even. The fact that the groups $\B(e,e,2)$, that is the Artin
groups of type $I_2(e)$, are distinct groups
is proved in \cite{PARISISO}, and $\B(2e,e,2)$ is the only group
of rank 2 with $H_1(B,\Z) = \Z^3$. 
If $H^2(B,\Z) = 0$, then $W$ is either $G_{12}$, $G_{22}$
or $G(e,e,2)$ with $e$ odd. In these cases, there
is only one non-trivial morphism $\eps : B \onto \Z_2$, so we
can compare the groups $H_1(B,\Z_{\eps})$ determined in section	 \ref{ss:sign}.
It is $\Z_3$ for $G_{12}$, $0$ for $G_{22}$, and $\Z$ for the
$\B(e,e,2)$. Once again, the groups $\B(e,e,2)$ can be distinguished
following \cite{PARISISO}, and this concludes the proof.
\end{proof}

In order to distinguish the exceptional groups, we need to
prove a couple of independent results by ad-hoc methods.
We let $B_{23},B_{24},\dots$ denote the complex braid groups associated
to $G_{23},G_{24},\dots$.

\begin{lemma}\label{lemadhoc}\mbox{ }
\begin{enumerate}
\item There is no surjective morphism from $B_{24}$ to the
alternating group $\mathfrak{A}_5$.
\item $B_{24}$ is not isomorphic to $B_{23}$.
\item There is no surjective morphism from $\B(3,3,4)$ to the
symmetric group $\mathfrak{S}_6$.
\item $B_{31}$ is not isomorphic to $\B(3,3,4)$.
\end{enumerate}
\end{lemma}
\begin{proof}
Recall that $B_{24}$ has a presentation with generators $s,t,u$
and relations $stst = tsts$, $tutu = utut$, $sus = usu$, $tstustu = 
stustus$. We check by computer that none of the $60^3$ tuples
$(s,t,u) \in \mathfrak{A}_5^3$ can generate $\mathfrak{A}_5$ and
satisfy these relations at the same time, which proves (1).
This implies $(2)$, as $G_{23}/Z(G_{23}) \simeq \mathfrak{A}_5$ (see \cite{BMR}).
We proceed in the same way for (3), using the presentation
in \cite{BMR} for $\B(3,3,4)$,
namely with generators $s,t,u,v$ and presentation
$sts=tst, stustu = ustust,sus=usu, tut = utu,
vuv=uvu, vs=sv, vt=tv$. By computer we find that there exists 9360
4-tuples in $\mathfrak{S}_6$ satisfying these relations, none of them
generating $\mathfrak{S}_6$, which proves (3). Then (4)
is a consequence of (3), because $G_{31}/Z(G_{31})$ is a
semidirect product $2^4 \rtimes \mathfrak{S}_6$ (again, from \cite{BMR}).
\end{proof}

In rank at least 3, using $H_2(B,\Z)$ and $H_1(B,\Z)$, we can
separate the groups $\B(2e,e,r)$ from the rest, as they
are the only groups with $H_1(B,\Z) = \Z^2$ and infinite
$H_2(B,\Z)$. All exceptional groups of rank at least $3$
have $|\mathcal{A}/W| = 1$, that is $H_1(B,\Z) = \Z$, except $G_{28}=F_4$.

\begin{theo} 
The correspondence $W \mapsto B$
is injective on the 2-reflection groups with $|\mathcal{A}/W| = 1$.
\end{theo}
\begin{proof}
Note that the assumption $|\mathcal{A}/W| = 1$, which is equivalent
to $H_1(B,\Z) = \Z$, implies that $W$ is irreducible. It also
implies that there exists a unique surjective morphism
$\eps : B \onto \Z_2$, so that $H_1(B,\Z_{\eps})$ is well-defined.
In rank 2, the statement to prove is a consequence of above, so we can
assume that the rank $r$ is at least 3. Then only cases with
infinite $H_2(B,\Z)$ are the exceptional rank 3 groups
$G_{23}$, $G_{24}$, $G_{27}$. The $H_2$ being in these cases $\Z,\Z,\Z_3 \times \Z$,
only $G_{23}$ and $G_{24}$ need to be distinguished, and this done in Lemma \ref{lemadhoc}.
We can now assume that $H_2(B,\Z)$ is finite. Since all exceptional
groups have been taken care of in rank 3, and $H_2(\B(e,e,3),\Z) = \Z_e$,
$H_2(\B(1,1,4),\Z) = \Z_2$ with $G(1,1,4) \simeq G(2,2,3)$,
$W \mapsto B$ is injective in rank 3 and we can assume that the rank is
at least $4$.
In rank $4$ and $W = G(e,e,4)$, $e$ is odd exactly when $H_2(B,\Z)$
is cyclic, so all such $\B(e,e,4)$ are distinguished by $H_2(B,\Z)$.
Moreover, since $H_2(B_{29},\Z) = \Z_2 \times \Z_4$ is neither cyclic nor
isomorphic to a group of the form $\Z_e \times \Z_2^2$, it
does not appear as the $H_2$ of a $\B(e,e,4)$. We have $H_2(B_{30},\Z) = \Z_2 \simeq \Z_e \times
\Z_2$ if $e = 1$, but $G(1,1,4) \simeq \mathfrak{S}_4$ has rank $3$.
We thus only need to distinguish $B_{31}$ from $\B(3,3,4)$. If the presentation
of \cite{BMR} for $B_{31}$, is correct, they are distinguished
by $H_1(B,\Z_{\eps})$. Otherwise, we can use the argument in the Lemma \ref{lemadhoc}
above. 
When $r \geq 5$, we have $H_2(\B(e,e,r),\Z) = \Z_e \times \Z_2$,
and $H_1(B,\Z_{\eps}) = \Z_3$ when $e \geq 2$. Now $H_1(B_{33},\Z_{\eps}) =
H_1(B_{34},\Z_{\eps}) = 0$, $H_2(B_{33},\Z) =
H_1(B_{34},\Z) = \Z_6$ so this distinguishes $B_{33}$ and $B_{34}$.
There only remains to distinguish the Artin groups $B_{35},B_{36},B_{37}$
of types $E_6,E_7,E_8$ from the usual braid groups $\B(1,1,r)$,
and this is done in \cite{PARISISO}.

\end{proof}

In the family of groups $\B(2e,e,r)$, there are many isomorphisms,
and we only get partial results in section \ref{s:iso}.

{\bf Acknowledgements.} The first computations for the exceptional groups
were made with the help of Jean Michel. The second author benefited
of the ANR grant ANR-09-JCJC-0102-01.



\def\arxiv#1#2{\href{http://arxiv.org/abs/#1}{\texttt{arXiv:#1[math:#2]}}}
\def\MR#1{}

%
%

\section{Homology of the classical braid group} \label{s:braids}


Let $\Br(n)$ be the classical Artin braid group in $n$ strands. We recall the description of the homology of these groups
according to the results of \cite{cohen, fuks, vain}. 
We'll adopt a notation coherent with \cite{dps} (see also \cite{C2}) for the description of 
the algebraic complex and the generators. Let $\F$ be a field.
The direct sum of the homology of $\Br(n)$ for $n \in \N = \Z_{\geq 0}$ is considered as a bigraded ring $\oplus_{d,n} H_d(\Br(n), \F)$ 
where the product structure $$H_{d_1}(\Br(n_1), \F) \times H_{d_2}(\Br(n_2) \to H_{d_1+d_2}(\Br(n_1+n_2)$$
is induced by the map $\Br(n_1) \times \Br(n_2) \to \Br(n_1 + n_2)$ that juxtapose braids(see \cite{cohen_braids, C2}). 

\subsection{Braid homology over $\Q$}

The homology of the braid group with rational coefficients has a very simple description:
$$
H_d(\Br(n), \Q)= \Q[x_0, x_1]/(x_1^2)_{\deg = n, \dim= d}
$$
where $\deg x_i = i+1 $ and $\dim x_i = i$.
In the Salvetti complex the element $x_0$ is represented by the string 0 and $x_0$ is 
represented by the string 10. In the representation of a monomial $x_0^ax_{1}^b$ we drop the last 0. 

For example the generator of $H_1(\Br(4),\Q)$ is the monomial $x_0^2x_1$ and we can also write it as a string in the form $001$ (instead of $0010$, dropping the last $0$).

We denote by $A(\Q)$ the module $\Q[x_0, x_1]/(x_1^2)[t^\pmu]$.

\subsection{Braid homology over $\F_2$}

With coefficients in $\F_2$ we have:

$$
H_d(\Br(n), \F_2)= \F_2[x_0, x_1, x_2, x_3, \ldots]_{\deg = n, \dim= d}
$$
where the generator $x_i, i \in \N$ has degree $\deg x_i = 2^i$ and homological dimension 
$\dim x_i = 2^i-1$.

In the Salvetti complex the element $x_i$ is represented by a string of $2^i -1$ 1's 
followed by one 0. In the representation of a monomial $x_{i_1}\cdots x_{i_k}$ we drop the last 0.

We denote by $A(\F_2)$ the module $\F_2[x_0, x_1, x_2, x_3, \cdots][t^\pmu]$.

\subsection{Braid homology over $\F_p$, $p>2$}

With coefficients in $\F_p$, with $p$ an odd prime, we have:

$$
H_d(\Br(n), \F_p)= \left( \F_2[h, y_1, y_2, y_3, \ldots] \otimes \Lambda[x_0, x_1, x_2, x_3, \ldots] 
\right)_{\deg = n, \dim= d}
$$
where the second factor in the tensor product is the exterior algebra over the field $\F_p$ 
with generators $x_i, i \in\N$. 
The generator $h$ has degree $\deg h = 1$ and homological dimension $\dim h=0$.
The generator $y_i, i \in \N$ has degree $\deg y_i = 2p^i$ and homological dimension 
$\dim y_i = 2p^i-2$.
The generator $x_i, i \in \N$ has degree $\deg x_i = 2p^i$ and homological dimension 
$\dim x_i = 2p^i-1$.

In the Salvetti complex the element $h$ is represented by the string $0$, the element $x_i$ is 
represented by a string of $2^i -1$ 1's followed by one 0. 
The element $y_i$ is represented by the following term (the differential is computed over 
the integers and then, 
after dividing by $p$, we  consider the result modulo $p$):
$$
\frac{d(x_i)}{p}.
$$
In the representation of a monomial $x_{i_1}\cdots x_{i_k}h^iy_{j_1}\cdots y_{j_h}$ we drop 
the last 0.

We denote by $A(\F_p)$ the module 
$$\F_p[h,y_1, y_2, y_3, \ldots] \otimes \Lambda[x_0, x_1, x_2, x_3, \ldots][t^\pmu].$$
We write simply $A$ instead of $A(\Q)$, $A(\F_2)$ or $A(\F_p)$ when the field we are considering 
is understood.

\section{Homology of $\B(2e,e,r)$} \label{s:b2eer}
\subsection{Preliminary computations} \label{ss:b2eerintro}
Recall from \cite{BMR} that for $d>1$ $\B(de,e,r) = \B(2e,e,r)$. 

In this section and in section \ref{s:iso} we always 
assume $d>1$. The case $d=1$ will be treated in a different part.

We want to understand the homology of $\B(*e,e,r)=\B(2e,e,r)$ 
with coefficient in $\F_p$.


We start computing the homology of the group $\B(2e,e,r)$ with coefficients in the field 
$\F$. In what follows $\F$ will be mainly a prime field $\F_p$, 
but we will also be interested to obtain again the results of Lehrer for rational coefficients in order to have a 
description of the generators.

According to \cite{BMR} we have that for $d>1$ the group $\B(d,1,r)$ is the subgroup of the 
classical braid group $\Br(r+1)=$ 
$$<\xi_1, \ldots, \xi_r | \xi_i \xi_{i+1} \xi_i = \xi_{i+1} \xi_i \xi_{i+1}, [\xi_i, \xi_j] = 1 
\mbox{ if } |i-j| \neq 1>$$
generated by the elements $\xi_1^2, \xi_2, \ldots, \xi_r$. This is isomorphic to the 
Artin group of type $\Bx_r$, $\Art_{\Bx_r}$ with corresponding generators 
$\overline{\sigma}_1, \sigma_2, \ldots, \sigma_r$ and Dynkin diagram as in Table \ref{tab:DynkinB_r}. 

%
\begin{dummyenvironment}
\entrymodifiers={=<4pt>[o][F-]}
\begin{table}[hbtp]
\begin{center}
\begin{tabular}{c}
\xymatrix @R=2pc @C=2pc {
\ar @{-}[r]_(-0.20){\overline{\sigma}_1}^4 & \ar @{-}[r]_(-0.20){\sigma_2} & 
\ar @{.}[r]_(-0.20){\sigma_3} & \ar @{-}[r]_(-0.20){\sigma_{r-1}}_(.80){\sigma_r} &}
\end{tabular}
\end{center}
\caption{Dynkin diagram for the Artin group of type $\Bx_r$}
\label{tab:DynkinB_r}
\end{table}
\end{dummyenvironment}
The group 
$\B(de,e,r)$ is isomorphic to the subgroup of $\B(d,1,r)$ generated 
by $$\xi_1^{2e}, \xi_1^2 \xi_2 \xi_1^{-2},\xi_2, \ldots, \xi_r$$ hence it corresponds 
to the kernel of the map
$$
\phi_{e,r}:\Art_{\Bx_r} \to \Z/e \; \; \; , \; \; \; \left\{
\begin{array}{c}
\overline{\sigma}_1 \mapsto 1 \\
\sigma_i \mapsto 0 \mbox{ for } i > 1
\end{array}
\right.
$$

We can consider the group rings $\F[\Z/e] = \F[t]/(1-(-t)^e)$ and $\F[\Z] = \F[t^\pmu]$ 
as $\Art_{\Bx_r}$-modules through the maps $\phi_{e,r}$ and $\phi_r:\Art_{\Bx_r} \to \Z$ that 
maps $\overline{\sigma}_1 \mapsto 1$ and $\sigma_i \mapsto 0 $ for $i > 1$.

Applying the Shapiro Lemma (see \cite{brown}) we have that
$$
H_*(\B(de,e,r), \F) = H_*(\Art_{\Bx_r}, \F[t]/(1-(-t)^e)).
$$
Notice that this statement is also true when $\F$ is an arbitrary ring.
In order to compute the right term of the equality we begin studying 
the homology $H_*(\Art_{\Bx_r}, \F[t^\pmu])$, where the local system is determined by 
the map $\phi_r$.
To do this we consider the algebraic Salvetti complex for the Artin group of 
type $\Bx_r$, $C_*(r) = C_*(\Art_{\Bx_r})$ (see \cite{SALVETTI}) with coefficients in the group 
ring $\F[t^\pmu]$. We order the generators of $\Art_{\Bx_r}$ as in the diagram of Table \ref{tab:DynkinB_r}. 
We filter the complex $C_*(r)$ as follows:
$$
\Cal{F}_iC_*(r) = <AB>
$$
where $<AB>$ is the $\F[t^\pmu]$-submodule of $C_*(r)$ generated by all the strings of type $AB$, 
with $A$ a string of 0's and 1's of length $i$ with at least one 0.
It follows that we have an isomorphism 
$$
\Cal{F}_{i+1}C_*(r)/\Cal{F}_iC_*(r) \stackrel{\iota}{\simeq} C_*(\Art_{\Ax_{r-i-1}})[i,i+1]
$$
between the quotient of two consecutive filtrated terms and the Salvetti complex for the Artin 
group of type $\Ax_{r-i-1}$, that is the braid group on $r-i$ strands $\Br(r-i)$. The first index 
in square brackets means
a dimension shifting by $i$ and the second index means a degree shifting by $i+1$. 
The complex $\Cal{F}_{i+1}C_*(r)/\Cal{F}_iC_*(r)$ is generated by strings of the form $1^i0B$. 
Moreover the string $1^i0B$ corresponds, through the isomorphism $\iota$, to the string $B$ 
in the complex $C_*(\Art_{\Ax_{r-i-1}})[i,i+1]$.

We consider the direct sum 
$$C_*= \bigoplus_{r=0}^\infty C_*(r) $$
and we study the first quadrant spectral sequence $\{E^k_{i,j}, d^k\}_k$ induced by 
the filtration $\Cal{F}$ on the complex $C_*$. The complex $C_*$ is bigraded 
with $$|S| = \dim S = \mbox{ the number of 1's of the string } S$$ and 
$$\deg S = \mbox{ the length of the string }S.$$

The first observation is that we get a first quadrant spectral sequence and in the $E^0$ term 
we have:

$$
E^0_{i,j} = \Cal{F}_iC_{i+j}/\Cal{F}_{i-1}C_{i+j} = \bigoplus_{r=0}^\infty C_{j}(\Art_{\Ax_{r-i-1}})
$$

We can now study the first differential of the spectral sequence, that is $d^0$. Because of the 
chosen filtration, on each columns of
the spectral sequence the differential $d^0$ corresponds to the boundary map of the complex 
$C_*(\Art_{\Ax_{r-i-1}})$ with trivial local system.
It follows that
\begin{prop} \label{prop:ss}
The $E^1$ term of the first quadrant spectral sequence defined above 
is given as follows ($i,j \geq 0$):
$$
E^1_{i,j} = \bigoplus_{r=0}^\infty H_j(\Art_{\Ax_{r-i-1}}, \F[t^\pmu]) =  
\bigoplus_{r=0}^\infty H_j(\Br (r-i), \F[t^\pmu])=
$$
$$
=\bigoplus_{r=0}^\infty H_j(\Br (r-i), \F)\otimes \F [t^\pmu]
$$
since the $t$-local system is trivial on $\Br(r-i)$.

\end{prop}

\emph{Notation:} We denote by $\Br(0)$ and $\Br(1)$ the trivial group with one element, 
while $\Br(i)$ is empty for $i < 0$. Hence
$ H_*(\Br (1), \F) =  H_*(\Br (0), \F) = \F$ and both modules are concentrated 
in dimension $0$, while $H_j(\Br (i), \F)$ is the trivial summand for $i<0$. 

\begin{remark}
Proposition \ref{prop:ss} actually gives us an infinite family of spectral sequences. Using the 
previous argument we can define, for every index $r \in \N$, a spectral sequence 
$\{E_{i,j}^k(r), d^k(r)\}_k$ with first term
$$
E^1_{i,j}(r) = H_j(\Br (r-i), \F)\otimes \F [t^\pmu]
$$
which converges to the homology group $H_*(\Art_{\Bx_r}, \F[t^\pmu])$.
\end{remark}
 
Notice that each column of the $E^1$ term of the spectral sequence of Proposition \ref{prop:ss} 
is isomorphic to the bigraded module $A(\F)$ defined in section \ref{s:braids}.
The correspondence between an element $x \in A(\F)$ and an element in the $i$-th column of the 
spectral sequence is the following: if $x$ is a monomial, that corresponds to a string of $0$'s 
and $1$'s, we lift it to the same string preceded by a sequence $$\overbrace{1\cdots1}^{i}0.$$ 
For a generic element $x$ we extend the correspondence by linearity. We denote the lifted element 
by $z_i x$.

Our interest now is to study the higher differentials of the spectral sequence. Since they are 
induced by the boundary map of the complex $C_*(r)$, we give a description of this complex 
according to \cite{SALVETTI} and \cite{CMS2}.

We recall the definition of the following $q$-analog and $q,t$-analog polynomials:
$$ 
[0]_q :=1, \hspace{5mm} [m]_q := 1 + q + \cdots + q^{m-1} = \frac{q^m-1}{q-1} \mbox{ for }m\geq 1,
$$
$$
[m]_q! := \prod_{i=1}^m [m]_q,
$$
$$
\qbin{m}{i}_{q}: = \frac{[m]_q!}{[i]_q!
[m-i]_q!},
$$
$$
[2m]_{q,t} := [m]_q (1+tq^{m-1}),
$$
$$
[2m]_{q,t}!! := \prod_{i=1}^m [2i]_{q,t}\ =\ [ m ]_q!
\prod_{i=0}^{m-1} (1+tq^i),
$$
$$
\qbin{m}{i}_{q,t}': = \frac{[2m]_{q,t}!!}{[2i]_{q,t}!! [m-i]_q!} \ =
\qbin{m}{i}_{q}\prod_{j=i}^{m-1}(1+tq^j).
$$

In our computations, since we consider a local system that maps the generator associated to 
the first node of the Dynkin diagram $\Bx_r$ to a non-trivial action (i.e. $(-t)$-multiplication) 
and the other generators to a trivial action, we will specialize our polynomials to $q= -1$ 
(see \cite{SALVETTI}). 

%
%

By an easy computation with cyclotomic polynomials, combined with some result that appears in  
\cite{C2} we can easily prove the following Lemma, that will be useful in further computations. 
We will write $[*]_{-1}$ for the $q$-analog $[*]$ evaluated at $q = -1$.
\begin{lemma} \label{lem:qanalog}
For $q= -1$, the polynomial $\qbin{m}{i}_q$ evaluated over the integers is zero if and only if 
$m$ is even and $i$ is odd. 

Evaluated modulo $2$, it is non-zero if and only if, when we write $i$ and $m$ as sums of 
powers of $2$ without repetitions, all the terms in the sum for $i$ appears in the sum for $m$. 
Let $h$ be the number of integers $k$ such that there is a $1$ in the binary decomposition of $i$ or $m-i$ at the $k$-th position, but not in the binary decomposition for $m$. Then $2^h$ is the highest 
power of $2$ that divides the integer $\qbin{m}{i}_{-1}$.

Evaluated modulo a prime $p$, with $p>2$, the expression $\qbin{m}{i}_{-1}$ is non-zero if and only 
if when we write $i$ and $m-i$ as sums of terms of the form 
$$
i = l_0 + \sum_{k=1}^s l_k 2p^{k-1}
$$
$$
m-i = l'_0 + \sum_{k=1}^s l'_k 2p^{k-1}
$$
with $0\leq l_0, l_0' \leq 1$, $0 \leq l_k, l_k' < p$ for $k=1, \ldots, s$, we have 
$l_0 + l_0' < 2$ and $l_k + l_k' < p$ for all $k=1, \ldots, s$. Moreover, if 
$$
m = l''_0 + \sum_{k=1}^s l''_k 2p^{k-1}
$$
with $0\leq l_0'' \leq 1$, $0 \leq l_k'' < p$ for $k=1, \ldots, s$, then the integer $h$ defined as
$$
h:= \sharp\{k \in \N \mid l_h'' < l_h + l_h'\}
$$
is the greatest exponent such that $p^h$ divides $\qbin{m}{i}_{-1}$. 
\end{lemma}
\begin{proof}
Let us sketch the idea of the proof. The main point is to study the divisibility of the 
polynomial $\qbin{m}{i}_{q}$ by the cyclotomic polynomials $\ph_j(q)$. Moreover we need to 
recall that $\ph_j(-1) \neq 0$ if $j \neq 2$ and  for any prime $p$, $\ph_{2p^j} (-1) = p$ 
for $j>0$, and in all the other cases $\ph_j(-1)=1$. The number $h$ is the number of digits 
that we \emph{carry over} in the sum between $i$ and $m-i$ written respectively in base $2$ and 
in the base associated to an odd prime $p$, corresponding to the sums of the last part of the 
statement. The integer $h$ actually counts the number of times a factors of the form 
$\ph_{2p^j}$, $p>0$ divides the $q$-analog $\qbin{m}{i}_{q}$.
\end{proof}

%
%
%
%

Finally we present the boundary maps for the complex $C_*(\Art_{\Ax_r})$ and $C_*(r)$.
We write $\partial$ for the boundary map in the complex $C_*(\Art_{\Ax_r})$ and 
$\overline{\partial}$ for the boundary in the complex $C_*(r)$.
Recall that the complex $C_*(\Art_{\Ax_r})$ over a module $M$ is the direct sum 
$$
\bigoplus_{|x| = r} M.x
$$
of one copy of $M$ for each string $x$, made of $0$'s and $1$'s, of length $r$. Notice that 
these strings are in 1 to 1 correspondence with 
the parts of a set of $r$ elements (in particular with the set of the nodes of the Dynkin 
diagram of type $\Ax_r$). A $1$ in the $j$-th position of the string means that the $j$-th 
element belongs to the subset, while a 0 means that it doesn't belong to the subset. The complex
is graduated as follows: the dimension of a non-zero element $m \in M.x$ is given by the number 
of $1$'s in the string $x$, that is the cardinality of the corresponding subset.

The complex $C_*(r) = C_*(\Art_{\Bx_r})$ has the same description, as a graded module, as the 
complex $C_*(\Art_{\Ax_r})$ and they differs for the boundary.
Since in the Dynkin diagram of type $\Bx_r$ the first node is \emph{special}, we change slightly 
our notation for the string representing the generators of $C_*(r)$ using $\overline{0}$ or 
$\overline{1}$ in the first position, according to whether or not the first element belongs to 
the subset of the nodes.

We consider the nodes of the Dynkin diagram of type $\Ax_r$ ordered as in the Table 
\ref{tab:DynkinA_r}. 

%
\begin{dummyenvironment}
\entrymodifiers={=<4pt>[o][F-]}
\begin{table}[hbtp]
\begin{center}
\begin{tabular}{c}
\xymatrix @R=2pc @C=2pc {
\ar @{-}[r]_(-0.20){1} & \ar @{-}[r]_(-0.20){2} & \ar @{.}[r]_(-0.20){3} & 
\ar @{-}[r]_(-0.20){{r-1}}_(.80){r} &}
\end{tabular}
\end{center}
\caption{Dynkin diagram for the Artin group of type $\Ax_r$}
\label{tab:DynkinA_r}
\end{table}
\end{dummyenvironment}

Let $x$ be the string
$$
\overbrace{1\cdots 1}^{i_1}0\overbrace{1\cdots 1}^{i_2}0 \cdots 0\overbrace{1\cdots 1}^{i_k}
$$
we write it in a more compact notation as
$$
x= 1^{i_1}01^{i_2}0\cdots01^{i_k}.
$$
The boundary of $x$ in the complex $C_*(\Art_{\Ax_r})$ is given by the following sum:
$$
\partial x = \sum_{j=1}^k \sum_{h=0}^{i_j-1} (-1)^{i_1 + \ldots + i_{j-1} + h}\qbin{i_j+1}{h+1}_{-1}
  1^{i_1}0\cdots01^{i_{j-1}}01^{h}01^{i_j-h-1}01^{i_{j+1}}0\cdots01^{i_k}.
$$

In a simpler way (see \cite{dps, dpss}) we can say that the boundary is null on the string made 
of all $0$'s, moreover:
$$
\partial 1^l = \sum_{h=0}^l-1 (-1)^h \qbin{l+1}{h+1}_{-1} 1^{h}01^{l-h-1}
$$
and if $A$ and $B$ are two strings
$$
\partial A0B = (\partial A)0B + (-1)^{|A|} A0 \partial B.
$$

In the complex $C_*(r)$ the boundary $\DB x$ is given as follows: 
$$
\DB \overline{0}A = \overline{0}\DA A,
$$
$$
\DB \; \overline{1}1^{l-1} = 
\qbin{l}{0}_{-1}' \overline{0}1^{l-1} + \sum_{h=1}^{l-1} (-1)^{h}\qbin{l}{h}_{-1}' 
\overline{1}1^{h-1}01^{l-h-1}
$$
and
$$
\DB A0B = (\DB A)0B + (-1)^{|A|}A0\DA B.
$$

We can use the given description of the algebraic complex to compute explicitly the differential, 
$d^1$ in the spectral sequence. This is a first tutorial step in the computation of the whole 
spectral sequence of Proposition \ref{prop:ss}.
Recall that $d^1$ is an homomorphism with bidegree $(-1, 0)$ and maps
$$
d^1_{i,j}(r): E^1_{i,j}(r) \to E^1_{i-1,j}(r).
$$
A representative of a generator of $E^1_{i,j}(r)$ is of the form $\overline{1}1^{i-1}0x = z_ix$ 
where $x$ is a representative of an homology class in $H_j(\Br (r-i), \F)$. Since $x$ is already 
a cycle, we need to consider only the part of the boundary $\DB 1^i0x$ starting with $1^{i-1}0$, 
that is 
$$
d^1_{i,j}(r): z_ix \mapsto [i]_{-1} (1+t(-1)^{i-1}) z_{i-1}0x.
$$
Since the coefficient $[i]_{-1} $ is $0$ for even $i$'s and $1$ otherwise, we get:
$$
d^1_{i,j}(r)z_ix = \left\{
\begin{array}{lc}
 0 & \mbox{ if $i$ is even} \\
 (1+t)z_{i-1}0x & \mbox{ if $i$ is odd.}\\
\end{array}
\right.
$$
When we work on the prime field $\F_p$, with $p=2$ we write also $z_{i-1}x_0x$ for $z_{i-1}0x$ 
and when $p>2$ we write $z_{i-1}hx$ for $z_{i-1}0x$.
Each odd column inject in the even column on its left. The $E^2$ term of the spectral sequence 
easily follows from the description of the differential $d^1$.  We can briefly state this as:
\begin{prop}\label{prop:E^2}
In the $E^2$ term of the first quadrant spectral sequence of Proposition \ref{prop:ss} each 
column in even position is isomorphic to the quotient ring $A(\F_2)/((1+t)x_0)$ 
(resp. $A(\F_p)/((1+t)h)$ or $A(\Q)((1+t)x_0)$) for $\F=\F_2$ (resp. $\F = \F_p$, $p>2$ or $\F = \Q$). 
All the columns in odd position are zero.
\end{prop}

For a more advanced study of the spectral sequence and of its other terms, we need to split 
our analysis, considering separately the case $\F = \Q$ and the cases $\F = \F_p$, with $p=2$ and $p>2$.
\subsection{Homology of $\B(2e,e,r)$ with rational coefficients}
We start continuing the study of the spectral sequence of Proposition \ref{prop:E^2}. 
We only need to compute the differential of the $E^2$-term of the spectral sequence, since the
spectral sequence is concentrated in the first two rows, hence all 
the other differentials are zero and the spectral sequence collapses at $E^3$.

The differential $$
d^2_{i,j}(r): E^1_{i,j}(r) \to E^1_{i-2,+1}(r)
$$
acts as follows: 
$$
d^2_{i,j}(r): z_ix \mapsto \qbin{i}{2}_{-1} (1-t^2) z_{i-2}x_1x.
$$
The coefficient $\qbin{i}{2}_{-1}$ is always nonzero, hence we can define the
quotient 
$$
A_0(\Q) = A_0 = A(\Q)/((1+t)x_0, (1-t^2)x_1)
$$
in the $E^\infty$ term we have:

\begin{table}[htbp]
\begin{tabular}{|m{2em}|m{2em}|m{2em}|m{2em}|m{2em}|m{2em}|m{2em}}
& & & & & & \\
 \begin{center} $A_0$ \end{center} &\begin{center} $0$ \end{center} & \begin{center} $A_0$ 
\end{center} 
& \begin{center} $0$ \end{center} & \begin{center} $A_0$ \end{center} & \begin{center} $0$
\end{center} & \begin{center} $\cdots$ \end{center} \\
 & & & & & & \\
 \hline
\end{tabular}
\end{table}

The terms of the form $z_{2i}x_0^j$ lift, in $H_*(\Art_{\Bx_r}, \Q[t^\pmu])$ to the cycle
$$ 
\omega_{2i,j,0}= \frac{\DB(z_{2i+1}x_0^{j-1})}{(1+t)}
$$
while the terms of the form $z_{2i}x_0^jx_1$, for $j > 1$ lift to
$$
\omega_{2i,j,1}=\frac{\DB(z_{2i+1}x_0^{j-1}x_1)}{(1+t)}
$$
and, for $j=0$ to
$$
\omega_{2i,0,1}=\frac{\DB(z_{2i+2})}{(1-t^2)}
$$

We can then compute the homology $H_*(\B(2e,e,r), \Q)= H_*(\Art_{\Bx_r}, \Q[t]/(1-(-t)^e))$ by means
of the homology long exact sequence associated to the short exact sequence
\begin{equation} \label{eq:short_rat}
0 \to \Q[t^\pmu] \stackrel{(1-(-t)^e)}{\too} \Q[t^\pmu] 
\stackrel{\pi}{\too} \Q[t^\pmu]/(1-(-t)^e) \to 0
\end{equation}

We consider the following cycles in the complex for $H_*(\Art_{\Bx_r}, \Q[t]/(1-(-t)^e))$:
$$
\overline{\omega}_{2i,j,0}= \frac{(1-(-t)^e)z_{2i+1}x_0^{j-1}}{(1+t)},
$$
$$
\overline{\omega}_{2i,j,1}=\frac{(1-(-t)^e)z_{2i+1}x_0^{j-1}x_1}{(1+t)}
$$
and
$$
\overline{\omega}_{2i,0,1}=\frac{(1-(-t)^e)z_{2i+2}}{(1-t^2)}.
$$

Let $\delta$ be the differential of the long exact sequence of homology associated to the short 
exact sequence of Equation (\ref{eq:short_rat}), it is clear that
$$
\delta(\overline{\omega}_{2i,j,k}) = {\omega}_{2i,j,k}.
$$
Moreover we have that the cycles $\overline{\omega}_{2i,j,k}$ have $(1+t)$-torsion if $(j,k) \neq (0,1)$, 
$(1-t^2)$-torsion otherwise. This proves that the cycles $\overline{\omega}$ and $\pi_*(\omega)$ are the generators of the homology
$H_*(\B(2e,e,r), \Q)$ confirming the Poincar\'e polynomial already given by Lehrer (\cite{lehr04}).

\subsection{$H_*(\Art_{\Bx_r}, \F_2[t^\pmu])$}

We can now compute the differential in the term $E^2$ of the spectral sequence. 
The boundary map tells us that the differential
$$
d^2_{i,j}(r): E^1_{i,j}(r) \to E^1_{i-2,+1}(r)
$$
acts as follows: 
$$
d^2_{i,j}(r): z_ix \mapsto \qbin{i}{2}_{-1} (1+t)^2 z_{i-2}x_1x.
$$
The coefficient $\qbin{i}{2}_{-1}$, that we consider only for even values of $i$, 
is zero if $4 \mid i$, otherwise it is non-zero and the kernel of the differential 
is generated by the element $x_0$. Hence the picture of the spectral sequence, for $E^3 = E^4$ 
(note that the $d^3$ differential must be zero) is as follows: if $i$ is a multiple of $4$, 
then the $i$-th columns is isomorphic to the quotient 
$A(\F_2)/((1+t)x_0, (1+t)^2x_1)$ and 
if $i$ is even, but $4 \nmid i$, then the $i$-th columns is isomorphic to the submodule quotient 
$x_0 A(\F_2)/((1+t)x_0) \simeq A(\F_2)/((1+t))$ (this is an isomorphism, but not a 
bi-graded-isomorphism); all the other columns are zero.

In order to give a description of the general behaviour of the spectral sequence we need 
the following definitions.

For $a \in \N$ we define the following ideals of $A = A(\F_2)$ (also for these definitions we'll 
drop the notation referring to the prime $p=2$ when it is understood):
$$
J_a(\F_2)=J_a=((1+t)x_0, (1-t^2)x_1, \ldots, (1-t^2)^{2^{a-1}}x_{a}).
$$
We define also the quotients:
$$
A_a(\F_2) = A_a= A(\F_2)/J_a(\F_2)
$$
and the ideals of $A_a$:
$$I_a(\F_2)= I_a = (x_0, x_1, \ldots, x_a) \subset A_a(\F_2).$$
Moreover we define $$J_\infty(\F_2)= J_\infty = \cup_{a=0}^\infty J_a(\F_2)$$ and 
$$A_\infty(\F_2) = A_\infty = A/J_\infty.$$
With this notation the page $E^3=E^4$ of the spectral sequence looks as follows:
\begin{table}[htbp]
\begin{tabular}{|m{2em}|m{2em}|m{2em}|m{2em}|m{2em}|m{2em}|m{2em}}
& & & & & & \\
 \begin{center} $A_1$ \end{center} &\begin{center} $0$ \end{center} & \begin{center} $I_0$ 
\end{center} 
& \begin{center} $0$ \end{center} & \begin{center} $A_1$ \end{center} & \begin{center} $0$
\end{center} & \begin{center} $\cdots$ \end{center} \\
 & & & & & & \\
 \hline
\end{tabular}
\end{table}

This result gives a description of the general behaviour of the spectral sequence:
\begin{theo} \label{t:ss}
The $k$-th term of the spectral sequence described in Proposition \ref{prop:ss} computing 
the homology $H_*(\Art_{\Bx_r}, \F_2[t^\pmu])$ is as follows:
\begin{itemize}
\item if $k = 2^a$ the $i$-th column is isomorphic to:
\begin{itemize}
\item $0$ if $i$ is odd;
\item $I_h$ if $2^{h+1} \mid i$ and $2^{h+2} \nmid i$, with $h+1<a$;
\item $A_{a-1}$ if $2^a \mid i$.
\end{itemize}
The differential $d^{2^a}$ is as follows: if $2^{a} \mid i $ and $2^{a+1}\nmid i$ we have the map
$$
d^k_{i,j}: z_ix \mapsto \qbin{i}{2^a}_{-1} (1+t)^{2^a} z_{i-2^a}x_ax
$$
where the $q$-analog coefficient is invertible; all the other differentials are trivial.
\item if $2^a < k < 2^{a+1}$ $E^k = E^{2^{a+1}}$ and the differential $d^k$ is trivial.
\end{itemize}
In $E^\infty$ term of the spectral sequence the $i$-th column is isomorphic to:
\begin{itemize}
\item $0$ if $i$ is odd;
\item $I_h$ if $2^{h+1} \mid i$ and $2^{h+2} \nmid i$;
\item $A_{\infty}$ if $i = 0$.
\end{itemize}
The homology $H_*(\Art_{\Bx_r}, \F_2[t^\pmu])$ is isomorphic to the graduate module associated 
to the $E^\infty$ term.
\end{theo}
\begin{proof}
We prove the first part of the statement by induction on $a$. 
The second part of the Theorem will follow from the first part.

We already have a description of the term $E^4$, so we can use $a=2$ as a starting point for the induction.

In order to prove the inductive step, it is useful to give a more precise statement with an explicit description of the generator of the
generators in the $E^k$ term of the spectral sequence. 

Let $2^{a-1} < k \leq 2^a$ and let $I_h$ be an ideal in the $i$-th column (hence $2^{h+1} \mid i$ and 
$2^{h+2} \nmid i, h+1 < a$). The generators $x_0, \ldots, x_h$ of the ideal $I_h$ are the images of the elements $z_ix_0, \cdots, z_ix_h$ of the $E^0$
term of the spectral sequence. A generic monomial of the ideal $I_h$ is in the form $m= x_s x_{s_1} \cdots x_{s_n}$ with $0 \leq s \leq h$, 
$s \leq s_1 \leq \cdots \leq s_n$. The monomial $m$ is the image of the element $z_i x_s x_{s_1} \cdots x_{s_n}$ in the $E^0$ term of the spectral
sequence. Its lifting in the $E^k$ term of the spectral sequence is given by 
$$
\alpha_{i, 0} = \frac{\DB(z_{i+1}x_{s_1} \cdots x_{s_n})}{(1+t)}
$$
for $s=0$ and
$$
\alpha_{i, s} = \frac{\DB(z_{i+2^s}x_{s_1} \cdots x_{s_n})}{(1-t^2)^{2^{s-1}}}
$$
for $s > 0$. In particular these terms lift to cycles, hence all the further differentials in the spectral sequence map them to zero.

The differential $\DB (z_{2^{l}(2m+1)} x_{s_1} \cdots x_{s_n})$ is given by a sum of the form
$$
\qbin{2^l(2m+1)}{2^l}'_{-1}z_{2^{l+1}m} x_{l} x_{s_1} \cdots x_{s_n} + \ldots
$$
where the remaining terms start with factors $z_r$ with $r < 2^{l+1}m$, hence they belong to an higher degree of the filtration with respect to the first term written above.

We note that the coefficient $\qbin{2^l(2m+1)}{2^l}'_{-1}$ is nonzero. In particular $$\qbin{2^l(2m+1)}{2^l}'_{-1} = \qbin{2^l(2m+1)}{2^l}_{-1} (1-t^2)^{l-1} $$ and the coefficient  $\qbin{2^l(2m+1)}{2^l}_{-1}$ is invertible, as proved in Lemma \ref{lem:qanalog}.

Now let $A_{a-1}$ be the module in the column $i$ with $2^a \mid i$. A monomial $w$ in $A_{a-1}$ is in the form $w= x_{s_1} \cdots x_{s_n}$
with $s_1 \leq \cdots \leq s_n$ (of course it can be $n=0$, that is $m= 1$). The monomial $m$ is the image of the element $z_i x_{s_1} \cdots x_{s_n}$ in the $E^0$ term of the spectral sequence.
For what we have just observed, $z_i x_{s_1} \cdots x_{s_n}$ will survive in the spectral sequence until page $E^{2^a}$. If $2^{a+1} \mid i$ then the differential $d^{2^a}w$ will be zero. Otherwise, if $i = 2^a(2m+1)$, then 
$$
d^{2^a}w = d^{2^a} z_{2^a(2m+1)} x_{s_1} \cdots x_{s_n} = \qbin{2^a(2m+1)}{2^a}'_{-1}z_{2^{a+1}m} x_{a} x_{s_1} \cdots x_{s_n}
$$
that is, up to invertible factors:
$$
d^{2^a}w = d^{2^a} z_{2^a(2m+1)} x_{s_1} \cdots x_{s_n} = (1-t^2)^{a-1} z_{2^{a+1}m} x_{a} x_{s_1} \cdots x_{s_n}.
$$

This means that the differential $d^{2^a}$ is as described in the statement of the Theorem:
$$
d^k_{i,j}: z_ix \mapsto \qbin{i}{2^a}_{-1} (1+t)^{2^a} z_{i-2^a}x_ax.
$$
The kernel of $d^k_{i,j}:A_{a-1} \to A_{a-1} $ is the ideal of $A_{a-1}$ generated by those monomials that are killed by the multiplication by $(1+t)^{2^a} x_a$ that is the ideal $(x_0, \cdots, x_a) = I_a$. 

The cokernel of $d^k_{i,j}:A_{a-1} \to A_{a-1} $ is the quotient of the ring $A_{a-1}$ by the ideal generated by $(1+t)^{2^a} x_a$, that is $A_a$.

\end{proof}

\begin{remark} \label{rem:generators2}
The proof of Theorem \ref{t:ss} gives us a precise description of the generators of 
the $E^\infty$ term:
\begin{itemize}
\item the module $I_h$ in the $[2^{h+1}(2m +1)]$-th columns is generated by the terms 
$$
\alpha_{2^{h+1}(2m +1), 0} = \frac{\DB(z_{2^{h+1}(2m+1)+1})}{(1+t)},
$$
$$\alpha_{2^{h+1}(2m +1), i} = \frac{\DB(z_{2^{h+1}(2m+1)+2^i})}{(1-t^2)^{2^{i-1}}} $$ 
for $i = 1, \ldots, h$, corresponding to the generators $x_0,x_1, \ldots,x_h$ of the ideal $I_h$; 
the generator corresponding to the monomial $x_ix_{i_1}\cdots x_{i_k}$ ($i_j \geq i$ for all $j$) is
$$
\alpha_{2^{h+1}(2m +1), 0} x_{i_1}\cdots x_{i_k}= 
\frac{\DB(z_{2^{h+1}(2m+1)+1}x_{i_1}\cdots x_{i_k})}{(1+t)}
$$
for $i=0$ and
$$\alpha_{2^{h+1}(2m +1), i} x_{i_1}\cdots x_{i_k}= 
\frac{\DB(z_{2^{h+1}(2m+1)+2^i}x_{i_1}\cdots x_{i_k})}{(1-t^2)^{2^{i-1}}} $$
for $i >0$.
\item the module $A_{\infty}$ in the $0$-th column is generated by $z_0$.
\end{itemize}
These generators actually are cycles in the algebraic complex $C_*(r)$ and naturally lift to 
generators of the homology $H_*(\Art_{\Bx_r}, \F_2[t^\pmu])$ which inherit the structure of 
$A(\F_2)[t^\pmu]$-module.
\end{remark}

\emph{Note:} when we use the notation 
$$
\frac{\DB x}{a(t)}
$$
we mean that we consider the boundary of the element $x$ computed in the complex 
$C_*(r) = C_*(\Art_{\Bx_r})$ with coefficients in the ring
of Laurent polynomials over the integers $\Z[t^\pmu]$, then we divide exactly by the polynomial 
$a(t)$ and finally we consider the quotient as a class in the coefficients we are using 
(for example, $\F_2[t^\pmu]$ in the case of Remark \ref{rem:generators2}).

\subsection{Homology of $\B(2e,e,r)$, $p=2$}

The result of Theorem \ref{t:ss} together with the description of the generators of the modules 
in the spectral sequence
allow us to compute the homology $H_*(\B(2e,e,r), \F_2)= H_*(\Art_{\Bx_r}, \F_2[t]/(1+(t)^e))$. 
We only need to study the homology long exact sequence associated to the short exact sequence
\begin{equation} \label{eq:short}
0 \to \F_2[t^\pmu] \stackrel{(1+t^e)}{\too} \F_2[t^\pmu] 
\stackrel{\pi}{\too} \F_2[t^\pmu]/(1+t^e) \to 0
\end{equation}
on the coefficients. We can state the following result:

\begin{prop}\label{prop:splitting}
We have a decomposition of the $\F_2[t^\pmu]$-module 
$$H_i(\Art_{\Bx_r}, \F_2[t]/(1+(t)^e)) = h_i(r,e) \oplus h'_i(r,e)$$ such that 
the homology long exact sequence associated to the short exact sequence given in Equation 
(\ref{eq:short}) splits:
\begin{equation} \label{eq:splitted}
0 \to h'_{i+1}(r,e) \stackrel{\delta}{\to} H_i(\Art_{\Bx_r}, \F_2[t^\pmu]) 
\stackrel{(1+t^e)}{\too} H_i(\Art_{\Bx_r}, \F_2[t^\pmu]) \stackrel{\pi_*}{\to} h_i(r,e) \to 0.
\end{equation}
\end{prop}
\begin{proof}
In order to prove this splitting, for each generator $x$ of the module $H_i(\Art_{\Bx_r}, \F_2[t])$ 
we provide an element $\widetilde{x} \in H_{i+1}(\Art_{\Bx_r}, \F_2[t]/(1+(t)^e))$ that maps to $x$ 
and we prove that $\widetilde{x}$ has the right torsion, with respect to the ring $\F_2[t^\pmu]$, 
in order to generate a submodule of $H_i(\Art_{\Bx_r}, \F_2[t]/(1+(t)^e))$ isomorphic to 
$$
\Ker(\F_2[t^\pmu]x \stackrel{(1+t^e)}{\too} \F_2[t^\pmu]x).
$$
Let $2^a$ be the greatest power of $2$ that divides $e$. We observe that the following equivalence 
holds, modulo $2$: $$1+t^e \simeq (1+t^{2^a}) \simeq (1+t)^{2^a} \mod{2}.$$
It turns out that the kernel and the cokernel of the map 
$$\F_2[t^\pmu]/(1+t)^{2^i} \stackrel{(1+t^e)}{\too} \F_2[t^\pmu]/(1+t)^{2^i}$$ are 
both isomorphic to the quotient $$\F_2[t^\pmu]/(1+t^{2^{\min (i,a)}}).$$
We are going to prove that every direct summand of the $\F_2[t^\pmu]$-module 
$H_i(\Art_{\Bx_r}, \F_2[t])$ 
of the form $\F_2[t^\pmu]/(1+t)^{2^i}$ gives rise to two copies of the module 
$\F_2[t^\pmu]/(1+t^{2^{\min (i,a)}})$, one in the same dimension, one in $1$ dimension higher. 
In particular the generator $\alpha_{c, i}$, where $c=2^{h+1}(2m +1)$, $i\leq h$, 
determines the two generators
$$
\widetilde{\alpha}_{c, i}= (1+t)^{2^a-2^{\min(i,a)}} z_{c+2^i} 
$$
and
$$
\pi_* \left(\frac{1}{(1+t)^{2^{\min(i,a)}}} \delta (\widetilde{\alpha}_{c, i})\right) = \pi_* (\alpha_{c, i}) = 
\frac{\DB(z_{c+2^i})}{(1+t)^{2^{i}}}
$$

Given a generic monomial $x=\alpha_{c, i} x_{i_1} \cdots x_{i_k}$ (again 
$c= (2m+1)2^{h+1}$, $i\leq h$) its projection is given by the cycle
$$
\pi_*(\alpha_{c, i}x_{i_1} \cdots x_{i_k}) = 
\frac{\DB(z_{c+2^i}x_{i_1} \cdots x_{i_k})}{(1+t)^{2^{i}}}.
$$
We remark that, given two elements $x=\alpha_{c, i} x_j x_{i_1} \cdots x_{i_k}$ and 
$x'=\alpha_{c, j} x_i x_{i_1} \cdots x_{i_k}$, since they correspond to the lifting 
of the same element in 
the spectral sequence, they represent the same homology class in $H_*(\Art_{\Bx_r}, \F_2[t^\pmu])$ 
(and, henceforth, their projection $\pi(x)$ and $\pi(x')$ are homologous). 
Hence we can suppose that the monomial $x=\alpha_{c, i} x_j x_{i_1} \cdots x_{i_k}$ 
is written in the form such that $i \leq i_1 \leq \cdots i_k$. 
We define the element $\widetilde{x}$ as
$$
\widetilde{x} = (1+t)^{2^a-2^{\min(i,a)}} z_{c+2^i} x_{i_1} \cdots x_{i_k}.
$$
With this definitions it is straightforward to check that $\delta (\widetilde{x})$ generates 
the submodule $$\Ker(\F_2[t^\pmu]x \stackrel{(1+t^e)}{\too} \F_2[t^\pmu]x)$$ and that 
$(1+t)^{2^{\min(i,a)}} \widetilde x=0$. 
Hence we have proved the splitting in Equation (\ref{eq:splitted}). 
The proof gives also a description of the 
generators of the homology $H_i(\Art_{\Bx_r}, \F_2[t]/(1+t^e))$ as a $\F_2[t^\pmu]$-module.
\end{proof}

As a consequence we can give a description of the homology of $\B(2e,e,r)$. 
Let us define for an integer $n$ the value $h_p(n)$ such that $p^{h_p(n)} \mid n$ and
$p^{h_p(n)+1} \nmid n$. For a bigraded module $M$, with degree $\deg$ and dimension $\dim$, 
we use the notation $M[n,m]$ for the module $M'$ isomorphic to $M$, but with bi-graduation shifted
such that $\deg' = \deg + n$, $\dim' = \dim + m$.
Finally, let $M\{n\} = M \oplus M[0,1] \oplus \cdots M[0,n-1]$

We can state the result as follows:

\begin{theo} \label{prop:b2eerF2}
The sum of homology groups $$\bigoplus_{r \geq 0} H_*(\B(2e,e,r), \F_2) = \bigoplus_{r \geq 0} 
H_*(\Art_{\Bx_r}, \F_2[t]/(1+t^e)) $$ is given by the sum
$$ A_\infty [1,0] \otimes \F_2[t]/(1+t^e)\{2\} \oplus
\bigoplus_{n=1}^\infty I_{h_2(n)}[2n+1  , 2n ] 
\otimes \F_2[t]/(1+t^e)\{2\}.
$$  \qed
\end{theo}


We can write explicitly the Poincar\'e polynomial of the homology $H_*(\B(2e,e,r), \F_2)$. If we call $P_2(\B(2e,e,r),u) = \sum_{i=0}^\infty \dim_{\F_2}H_i(\B(2e,e,r),\F_2) u^i$ such a polynomial, it is more convenient to consider the series in two variables
$$
P_2(\B(2e,e,*),u,v) = \sum_{r=0}^\infty P_2(\B(2e,e,r),u) v^r.
$$

The Poincar\'e series for the bigraded ring $A_\infty/(1+t)^e$ is given by
$$
P_{A_\infty(\F_2),e}(u,v) = e + \sum_{i=0}^\infty \left( 2^{\min(h_2(e),i)} u^{2^i-1}v^{2^i} \prod_{j\geq i} \frac{1}{1-u^{2^j-1}v^{2^j}} \right)
$$
and the Poincar\'e series of the ideal $I_{a} \otimes \F_2[t]/(1+t^e)$ 
is given by
$$
P_{I_{a}(\F_2),e}(u,v) =\sum_{i=0}^a \left( 2^{\min(h_2(e),i)} u^{2^i-1}v^{2^i} \prod_{j\geq i} \frac{1}{1-u^{2^j-1}v^{2^j}} \right).
$$

Hence we obtain the following result:
\begin{cor} The Poincar\'e polynomial of the homology of the groups $\B(2e,e,r)$ with $\F_2$ coefficients is given by:
$$
P_2(\B(2e,e,*),u,v) = v P_{A_\infty(\F_2),e}(u,v) (1+u) + 
\sum_{n=1}^\infty v^{2n+1}u^{2n}P_{I_{h_2(n)}(\F_2),e}(u,v) (1+u).
$$ \qed
\end{cor}

As an example of these computations we give in Table \ref{tab:homolF2} and Table \ref{tab:homolF2bis} the first homology groups of $\B(2e,e,r)$ with 
coefficients in the field $\F_2$ and the stable part up to homological dimension  $5$.

\begin{table}[htb]
\begin{center}
\begin{tabular}{|c|c|c|c|c|c|c|c|c|c|c|}
\hline
$r$&\multicolumn{2}{c|}{$2$}&$3$&\multicolumn{3}{c|}{$4$}&$5$&\multicolumn{2}{c|}{$6$}&$ 7$\\
\hline
$e \mod (m)$ & $0(2)$ & $1(2)$ & any & $0(4)$& $2(4)$ & $1(2)$& any & $0(2)$ & $1(2)$ & any \\
\hline
$\dim H_0$& $1$  &$1 $&$1 $&$1 $&$1 $&$1 $&$1 $&$1 $&$1 $ & $1$\\
\hline
$\dim H_1$ & $3$ & $2$                 & $2$ & $2$ & $2$ & $2$ & $2$ & $2$ & $2$ & $2$ \\
\hline
$\dim H_2$ & $2$& $1$   & $1$ & $4$ & $4$ & $3$ & $3$ & $3$ & $3$ & $3$ \\
\hline
$\dim H_3$ &\multicolumn{2}{c|}{$0$}  & $0$ & $7$ & $5$ & $3$ & $3$ & $6$ & $5$ & $5$ \\
\hline
$\dim H_4$ &\multicolumn{2}{c|}{$0$}  & $0$ & $4$ & $2$ & $1$ & $2$ & $6$ & $4$ & $3$ \\
\hline
$\dim H_5$ &\multicolumn{2}{c|}{$0$} & $0$ & \multicolumn{3}{c|}{$0$}  & $1$ & $5$ & $3$ & $4$ \\
\hline
$\dim H_6$ &\multicolumn{2}{c|}{$0$} & $0$ & \multicolumn{3}{c|}{$0$}  & $0$ & $2$ & $1$ & $3$ \\
\hline
$\dim H_7$ &\multicolumn{2}{c|}{$0$} & $0$ & \multicolumn{3}{c|}{$0$}  & $0$ & \multicolumn{2}{c|}{$0$}  & $1$ \\
\hline
\end{tabular}
\end{center}
\caption{$\dim H_*(\B(2e,e,r); \F_2)$, $r<8$} 
\label{tab:homolF2}
\end{table}

\begin{table}[htb]
\begin{center}
\begin{tabular}{|c|c|c|c|c|c|}
\hline
$r$&\multicolumn{4}{c|}{$8$}&$\geq 9$\\
\hline
$e \mod (m)$ &$0(8)$& $4(8)$ & $2(4)$& $1(2)$ & any\\
\hline
$\dim H_0$&$1$ &$1$ &$1$ &$1$ &$1$\\
\hline
$\dim H_1$&$2$ &$2$ &$2$ &$2$ &$2$\\
\hline
$\dim H_2$&$3$ &$3$ &$3$ &$3$ &$3$\\
\hline
$\dim H_3$&$5$ &$5$ &$5$ &$5$ &$5$\\
\hline
$\dim H_4$&$6$ &$6$ &$6$ &$5$ &$5$\\
\hline
$\dim H_5$&$8$ &$8$ &$8$ & $6$ &$6$\\
\hline
$\dim H_6$&$11$ &$11$ &$9$ & $6$&$ $\\
\hline
$\dim H_7$&$15$ &$11$ &$7$ & $4$&$ $\\
\hline
$\dim H_8$&$8$ &$4$ &$2$ & $1$&$ $\\
\hline
\end{tabular}
\end{center}
\caption{$\dim H_*(\B(2e,e,r); \F_2)$, $r=8$ and stable part up to $H_5$} 
\label{tab:homolF2bis}
\end{table}


\subsection{$H_*(\Art_{\Bx_r}, \F_p[t^\pmu])$}

As for the case $p=2$ we start computing the differential in the $E^2$ term of the spectral sequence. 
Again, the starting point is the result in Proposition \ref{prop:E^2}.
The differential $$
d^2_{i,j}(r): E^1_{i,j}(r) \to E^1_{i-2,+1}(r)
$$
acts as follows: 
$$
d^2_{i,j}(r): z_ix \mapsto \qbin{i}{2}_{-1} (1-t^2) z_{i-2}x_1x.
$$
The coefficient $\qbin{i}{2}_{-1}$ is zero in $p \mid i$. Recall in fact that we are considering only 
even columns, hence even values of $i$. So we have that for $p \mid i$, 
$\ph_{2p}(q) \mid \qbin{i}{2}_q$ and when we evaluate the polynomial for $q=-1$ we get $\ph_{2p}(-1) =p$.

Hence we can give the picture of the $E^3$ page of the spectral sequence:
all the odd columns are zero, if $i$ is a even multiple of $p$, then the $i$-th column is isomorphic 
to the quotient $A(\F_p)/((1+t)h,(1-t^2)x_0)$, if $i+2$ is a multiple of $p$, then the $i$-th column
is isomorphic to the submodule of $A(\F_p)/((1+t)h)$ generated by $h$ and $x_0$. If $i, i+2$ are 
not multiples of $p$, then the $i$-th column is isomorphic to the  submodule of  
$A(\F_p)/((1+t)h,(1-t^2)x_0)$ generated by $h$ and $x_0$.

As in the case of the prime $p=2$, we need to define some ideals of $A=A(\F_p)$:
$$
J_{2a+1}(\F_p) = J_{2a+1} = ((1+t)h, (1-t^2)x_0, (1-t^2)^{p-1}y_1, (1-t^2)^p x_1, \ldots, (1-t^2)^{p^a}x_a)
$$
and
$$
J_{2a}(\F_p) = J_{2a} = ((1+t)h, (1-t^2)x_0, (1-t^2)^{p-1}y_1, (1-t^2)^p x_1, \ldots, 
(1-t^2)^{(p-1)p^{a-1}}y_a).
$$

We define the quotients: 
$$
A_a(\F_p) = A_a = A(\F_p)/J_a(\F_p)
$$
and the ideals:
$$
I_{2a+1}(\F_p) = I_{2a+1} = (h, x_0, y_1, x_1, \ldots, y_a, x_a) \sst A_{2a+1},
$$
$$
I_{2a}(\F_p) = I_{2a} = (h, x_0, y_1, x_1, \ldots, y_a, x_a) \sst A_{2a}
$$
and
$$
K_{2a}(\F_p) = K_{2a} = (h, x_0, y_1, x_1, \ldots, x_{a-1}, y_a) \sst A_{2a}.
$$

Finally, as for $p=2$, we define

$$J_\infty(\F_p)= J_\infty = \cup_{a=0}^\infty J_a(\F_p)$$ and 
$$A_\infty(\F_p) = A_\infty = A/J_\infty.$$

With this notation the page $E^3$ of the spectral sequence looks as in Table \ref{tab:E3_p_odd}
(on the bottom we denote the number of the columns).
\begin{table}[htbp]
\begin{tabular}{|m{2em}|m{2em}|m{2em}|m{2em}|m{2em}|m{2em}|m{2em}|m{2em}|m{2em}|m{2em}}
& & & & & & & & & \\
 \begin{center} $A_1$ \end{center} &\begin{center} $0$ \end{center} & \begin{center} $I_1$ 
\end{center} 
& \begin{center} $\cdots$ \end{center} & \begin{center} $I_1 $ \end{center} 
& \begin{center} $0$ \end{center} &\begin{center} $I_0$ \end{center} 
& \begin{center} $0$ \end{center} & \begin{center} $A_1$ \end{center} 
& \begin{center} $\cdots$ \end{center} \\
${}_0$ & ${}_1$ & ${}_2$ & ${}_{\cdots}$ & ${}_{2p-4}$& ${}_{2p-3}$ &  ${}_{2p-2}$ 
& ${}_{2p-1}$  & ${}_{2p}$ & \\
 \hline 
\end{tabular}
\caption{The $E^3$ page of the spectral sequence for $p$ odd}
\label{tab:E3_p_odd}
\end{table}

It turns out that in the following terms the differential until $2p-2$ is zero, so $E^3 = E^{2p-2}$.

Here we have the general description of the spectral sequence in the analogous 
of Theorem \ref{t:ss} for odd primes:

\begin{theo} \label{t:ssp}
The $k$-th term of the spectral sequence described in Proposition \ref{prop:ss} computing 
the homology $H_*(\Art_{\Bx_r}, \F_p[t^\pmu])$ is as follows:
\begin{itemize}
\item if $k = 2p^a$ the $i$-th column is isomorphic to:
\begin{itemize}
\item $0$ if $i$ is odd;
\item $I_{2h+1}$ if $2p^{h} \mid i$ and $2p^{h+1} \nmid i$, $2p^{h+1} \nmid i+2p^{h}$ with $h < a$;
\item $K_{2h}$ if $2p^{h+1} \mid i+ 2p^{h}$ with $h<a$;
\item $A_{2a}$ if $2p^a \mid i$.
\end{itemize}
The differential $d^{2p^a}$ is as follows: if $2p^{a} \mid i $ and $2p^{a+1}\nmid i$ we have the map
$$
d^k_{i,j}: z_ix \mapsto \qbin{i}{2p^a}_{-1} (1-t^2)^{p^a} z_{i-2p^a}x_ax
$$
where the $q$-analog coefficient is invertible; all the other differentials are trivial.
\item if $k = 2p^a(p-1)$ the $i$-th column is isomorphic to:
\begin{itemize}
\item $0$ if $i$ is odd;
\item $I_{2h+1}$ if $2p^{h} \mid i$ and $2p^{h+1} \nmid i$, $2p^{h+1} \nmid i+2p^{h}$ with $h \leq a$;
\item $I_{2h}$ if $2p^{a+1} \mid i+ 2p^{a}$
\item $K_{2h}$ if $2p^{h+1} \mid i+ 2p^{h}$ with $h<a$;
\item $A_{2a+1}$ if $2p^{a+1} \mid i$.
\end{itemize}
The differential $d^{2p^a(p-1)}$ is as follows: if $2p^{a+1}\mid i+2p^a$ we have the map
$$
d^k_{i,j}: z_i x_a x \mapsto (1-t^2)^{p^a(p-1)} z_{i-2p^a(p-1)}y_{a+1}x
$$
and all the other differentials are trivial.
\item if $2p^a < k < 2p^a(p-1)$ $E^k = E^{2p^a(p-1)}$ and the differential $d^k$ is trivial.
\item if $2p^a(p-1) < k < 2p^{a+1}$ $E^k = E^{2p^{a+1}}$ and the differential $d^k$ is trivial.
\end{itemize}
In $E^\infty$ term of the spectral sequence the $i$-th column is isomorphic to:
\begin{itemize}
\item $0$ if $i$ is odd;
\item $I_{2h+1}$ if $2p^{h} \mid i$ and $2p^{h+1} \nmid i$, $2p^{h+1} \nmid i+2p^{h}$;
\item $K_{2h}$ if $2p^{h+1} \mid i+2p^h$;
\item $A_{\infty}$ if $i = 0$.
\end{itemize}
The homology $H_*(\Art_{\Bx_r}, \F_p[t^\pmu])$ is isomorphic to the graduate module associated 
to the $E^\infty$ term.
\end{theo}
\begin{proof}
As for the case of $p=2$, we prove the first part of the statement by induction on $a$ and the second 
part of the Theorem will follow from the first part.

We start with the description of $E^3$ that we gave in Table \ref{tab:E3_p_odd}. 
In order to work by induction, we'll give an explicit description of the generator of the
generators in the $E^k$ term of the spectral sequence. 


Let $2p^a < k$ and let $I_{2h+1}$ be an ideal in the $i$-th column. Hence $2p^h \mid i$ and $2p^{h+1} \nmid i$, 
$2p^{h+1} \nmid i + 2p^h$. The generators $h, x_0, y_1$, $x_1, \ldots, y_h$, $x_h$ are the images of the elements
$z_ih, z_ix_0, z_iy_1$, $z_ix_1, \ldots, z_iy_h$, $z_ix_h$ in the $E^0$ term. We consider these generators
of the ideal $I_{2h+1}$ ordered
as we wrote them, so $h$ will be the smallest generator and $x_h$ will be the biggest.
Let $m$ be a generic monomial in the ideal $I_{2h+1}$. We write its factors ordered from the smallest 
to the biggest. It is the image of the element $z_im$ in the $E^0$ 
term of the spectral sequence. Its lifting in the $E^k$ term is given as follows:
$$\frac{\DB(z_{i+1}m')}{(1+t)}$$ 
if the smallest factor of $m$ is $h$ and $m = hm'$,
$$\frac{\DB(z_{i+2p^s}m')}{(1-t^2)^{p^s}}$$ 
if the smallest factor of $m$ is $x_s$ and $m = x_sm'$. Note that we must have $s \leq h$ and hence 
$2p^{h+1} \nmid i + 2p^s$. This implies that the coefficient of $z_i x_s m'$ in $\DB(z_{i+2p^s}m')$ is non-zero.

If the smallest factor of $m$ is $y_s$ and let $m = y_sm'$ we need to define the following element.
Let $\DB[p](z_{i+2p^s}m')$ be the sum of all the terms that appears in $\DB(z_{i+2p^s}m')$ with a coefficients that is
divisible by $p$ (when we consider the boundary with integer coefficients). 
Notice that, with respect to the filtration $\Cal{F}$ of the complex,
 the highest term that doesn't appear in the sum is $(1-t^2)^{p^s}z_{i}x_sm'$. We define
$$
\gamma[p](z_{i+2p^s}m') = \frac{\DB[p](z_{i+2p^s})m'}{p(1-t^2)^{p^{s-1}}}
$$
and we have that the lifting of $z_im$ is given by
$$\frac{\DB(\gamma[p](z_{i+2p^s}m'))}{(1-t^2)^{(p-1)p^{s-1}}}.$$ 
In fact it is clear that the difference $$\frac{\gamma[p](z_{i+2p^s}m')}{(1-t^2)^{(p-1)p^{s-1}}} - z_iy_sm'$$ 
projects to a boundary in the quotient $\Cal{F}_{i+1}C_*(r)/\Cal{F}_iC_*(r)$. Notice that the quotient 
$$\frac{\gamma[p](z_{i+2p^s}m')}{(1-t^2)^{(p-1)p^{s-1}}}$$ is not defined in a $\F_p[t^\pmu]$-module
but still we can define it if we consider only the summands of $\gamma[p](z_{i+2p^s}m')$ that are 
not in $\Cal{F}_iC_*(r)$.  Hence we can use the first term of the
difference, instead of the second, to lift the class of $z_im$ to a representative in $E^k$. 

Since all the liftings that we have defined are global cycle, all the following differentials in the spectral sequence
map these terms to zero.


We now give a description of the generators of the ideal $I_{2h}$ appearing in the $i$-th column of the $E^k$ 
term of the spectral sequence. We must have $2p^{a+1} \mid i + 2p^a$. 
As before, the ideal $I_{2h}$ is generated by the terms $h, x_0, y_1$, $x_1, \ldots, y_h$, $x_h$ that are 
the images of the elements $z_ih, z_ix_0, z_iy_1$, $z_ix_1, \ldots, z_iy_h$, $z_ix_h$ in the $E^0$ term.

Given a monomial $m$ with smallest factor different from $x_h$, it is easy to verify that
the lifting is the same as in the previous description for the ideal $I_{2h+1}$.

Let now $m$ have smallest term $x_h$, with $m = x_h m'$. Let $d = h + d'$ be such that 
$2p^d \mid i + 2p^h$ and $2p^{d+1}\nmid i + 2p^h$. One can verify that $(1-t^2)^{p^h}z_ix_hm'$ is
the first non-zero element in $$
\frac{\DB[p^{d'}](z_{i+2p^h})m'}{p^{d'}}$$
since all the previous summands in $\DB[p^{d'}](z_{i+2p^h})m'$ 
(when we consider the boundary with integer coefficients) have a coefficients divisible by $p^{d'+1}$.
The first term missing in the sum $\DB[p^{d'}](z_{i+2p^h})m'$ is the monomial 
$(1-t^2)^{p^{h+1}} z_{i-(p-1)p^h} x_{h+1} m'$. It follows that we can take as a lifting of $x_hm'$ in 
$E^{2p^h(p-1)}$ the element
$$\frac{\DB[p^{d'}](z_{i+2p^h})m'}{p^{d'}(1-t^2)^{p^h}}$$
and hence its differential with respect to the map $d^{2p^h(p-1)}$ is 
$$\DB\frac{\DB[p^{d'}](z_{i+2p^h})m'}{p^{d'}(1-t^2)^{p^h}}$$
that is homologous, for what we have seen before, to the element
$$
(1-t^2)^{p^{h}(p-1)} z_{i-(p-1)p^h} y_{h+1} m'
$$
as stated in the Theorem.

Hence the differential $d^{2p^h(p-1)}$ maps $x_hm' \mapsto (1-t^2)^{p^{h}(p-1)} y_{h+1} m'$
and is zero for all the others elements. Clearly the kernel is given by the ideal $K_{2h}$.


For what concerns the ideal $K_{2h}$, generated by $h, x_0, y_1$, $x_1, \ldots, y_h$, the description of the
generators given before can be repeated and all the generators lift to global cycles as soon as $k > 2p^h(p-1)$.


Finally we consider the modules $A_{2h}$ and $A_{2h+1})$ that appear in the column $i$ of the spectral sequence.
Each monomial $m$ of $A_{2h}$ or $A_{2h+1})$ corresponds to a monomial $z_i m$ in the term $E^0$ of the
spectral sequence. If $2p^a \mid i$ then the monomial $z_i m$ will survive until the term $E^{2p^a}$
of the spectral sequence, since all the first summands of the differential $\DB z_i m$ are zero.
The first summand that can be non-zero is 
$$
d^k_{i,j}: z_ix \mapsto \qbin{i}{2p^a}_{-1} (1-t^2)^{p^a} z_{i-2p^a}x_ax
$$
that is actually non-zero if and only if $2p^{a} \mid i $ and $2p^{a+1}\nmid i$.
The kernel of the map $d^{2p^a}: A_{2a} \to A_{2a}$ is the ideal $I_{2a}$ and the quotient
of the kernel by the image of $d^{2p^a}$, when the image is non-zero, is the module $A_{2a+1}$.
Hence all the other differentials are forced to be zero and the behaviour of the spectral sequence is 
as described in the Theorem.
%
\end{proof}

\begin{remark}
From the proof we can read the description of the generators of 
the $E^\infty$ term:
\begin{itemize}
\item the module $I_{2h+1}$ in the $(2p^{h}n)$-th column, with $n \neq 0, -1 \mod p$ is generated
by the terms
$$
\beta_{(2p^{h}n),h}= \frac{\DB(z_{(2p^{h}n)+1})}{(1+t)}
$$
$$
\beta_{(2p^{h}n),x_i}=\frac{\DB(z_{(2p^{h}n)+2p^i})}{(1-t^2)^{p^i}}
$$
$$
\beta_{(2p^{h}n),y_i}=\frac{\DB(\gamma_{(2p^{h}n),y_i})}{(1-t^2)^{(p-1)p^{i-1}}}
$$
where we set
$$
\gamma_{(2p^{h}n),y_i} = \frac{\DB(z_{2p^{h}n+2p^i})-\sum_{(2p^{h}n),y_i}(p)}{p(1-t^2)^{p^{i-1}}}.
$$
and $\sum_{(2p^{h}n),y_i}(p)$ is the sum of the terms in $\DB(z_{2p^{h}n+2p^i})$ that have 
coefficients not divisible by $p$. Notice that the first of this terms is $(1-t^2)^{p^i}z_{2p^{h}n}x_i$.
\item the module $K_{2h}$ in the $2p^{h}(pn-1)$-th column, is generated by the terms
$$
\beta_{2p^{h}(pn-1),h}= \frac{\DB(z_{2p^{h}(pn-1)+1})}{(1+t)}
$$
$$
\beta_{2p^{h}(pn-1),x_i}=\frac{\DB(z_{2p^{h}(pn-1)+2p^i})}{(1-t^2)^{p^i}}
$$
and
$$
\beta_{2p^{h}(pn-1),y_i}=\frac{\DB(\gamma_{2p^{h}(pn-1),y_i})}{(1-t^2)^{(p-1)p^{i-1}}}
$$
with, for $i<h$
$$
\gamma_{2p^{h}(pn-1),y_i} = \frac{\DB(z_{2p^{h}(pn-1)+2p^i})-\sum_{2p^{h}(pn-1),y_i}(p)}{p(1-t^2)^{p^{i-1}}}.
$$
while, for $i=h$ we set
$$
\gamma_{2p^{h}(pn-1),y_h} = \frac{\DB(z_{2p^{h+1}n})-\sum_{2p^{h}(pn-1),y_h}(p^2)}{p^2(1-t^2)^{p^{h-1}}}
$$
where $\sum_{2p^{h}(pn-1),y_h}(p^2)$ is the sum of the terms in $\DB(z_{2p^{h+1}n})$ that have 
coefficients not divisible by $p^2$. Notice that the first of this terms is 
$p(1-t^2)^{p^h}z_{2p^{h}(pn-1)}x_h$.
\item the module $A_{\infty}$ in the $0$-th column is generated by $z_0$.
\end{itemize}

\end{remark}

\subsection{Homology of $\B(2e,e,r)$, $p > 2$}


Now we compute $H_*(\B(2e,e,r), \F_p)= H_*(\Art_{\Bx_r}, \F_p[t]/(1-(-t)^e))$ by means of the 
 homology long exact sequence associated to the short exact sequence
\begin{equation} \label{eq:shortp}
0 \to \F_p[t^\pmu] \stackrel{(1-(-t)^e)}{\too} \F_p[t^\pmu] 
\stackrel{\pi}{\too} \F_p[t^\pmu]/(1-(-t)^e) \to 0.
\end{equation}

As in the previous section, we have a splitting result:

\begin{prop}\label{prop:splittingp}
We have a decomposition of the $\F_p[t^\pmu]$-module 
$$H_i(\Art_{\Bx_r}, \F_p[t]/(1-(-t)^e)) = h_i(r,e) \oplus h'_i(r,e)$$ such that 
the homology long exact sequence associated to the short exact sequence given in Equation 
(\ref{eq:short}) splits:
\begin{equation} \label{eq:splittedp}
0 \to h'_{i+1}(r,e) \stackrel{\delta}{\to} H_i(\Art_{\Bx_r}, \F_p[t^\pmu]) 
\stackrel{(1+t^e)}{\too} H_i(\Art_{\Bx_r}, \F_p[t^\pmu]) \stackrel{\pi_*}{\to} h_i(r,e) \to 0.
\end{equation}
\end{prop}
\begin{proof}
We start observing that, since $1+t$ and $1-t$ are co-prime in $\F_p[t^\pmu]$ we can spit any module
of the form $\F_p[t^\pmu]/(1-t^2)^l$ as a direct sum
$$
\F_p[t^\pmu]/(1+t)^l \oplus \F_p[t^\pmu]/(1-t)^l
$$
of two modules, generated respectively by $(1-t)^l$ and $(1+t)^l$,

Moreover we consider the following properties for $1-(-t)^e$ in $\F_p[t^\pmu]$:
$$
(1-t, 1-(-t)^e)=\left\{
\begin{array}{lc}
 (1-t) & \mbox{ if $e$ is even} \\
 (1) & \mbox{ if $e$ is odd.}\\
\end{array}
\right.
$$
and
$$
(1+t, 1-(-t)^e)=\left\{
\begin{array}{lc}
 (1+t) & \mbox{ if $e$ is even or if $p\nmid e$} \\
 (1+t)^{p^i} & \mbox{ if $e$ is odd and $p^i \mid e$, but $p^{i+1}\nmid e$.}\\
\end{array}
\right.
$$
The second equality follows from the fact that the polynomial $1+t$ divides the cyclotomic polynomial
$\ph_{2p^i}$ with order exactly $\phi(p^i) = (p-1)p^{i-1}$ (in $\F_p[t^\pmu]$) and is co-prime with all 
the other cyclotomic polynomials.

Now, let us fix the value of $e$.
When we study the exact sequence of Equation (\ref{eq:splittedp}) we need to consider only the highest power of
$p$ that divides $e$ and whether $e$ is even or odd. 

Let us consider the monomial 
$x = z_{2m} h^r y_{i_1}^{s_{i_1}} \cdots y_{i_n}^{s_{i_n}} x_{k_1} \cdots x_{k_m}$ in 
$ H_i(\Art_{\Bx_r}, \F_p[t^\pmu])$. We suppose that we have the indexes ordered such that 
$i_1 < \cdots < i_n$ and $k_1 < \cdots < k_m$. We want to define a lifting $\widetilde{x}$ of $x$
in $H_i(\Art_{\Bx_r}, \F_p[t]/(1-(-t)^e))$. To do this we have to consider different cases:

First consider the case of $e$ even. If $r > 0$ we define
$$
\widetilde{x}= \frac{1-(-t)^e}{1+t}z_{2m+1} h^{r-1} y_{i_1}^{s_{i_1}} 
\cdots y_{i_n}^{s_{i_n}} x_{k_1} \cdots x_{k_m},
$$
if $r=0$ and $k_1< i_1$ then
$$
\widetilde{x}= \frac{1-(-t)^e}{1-t^2}z_{2m+2p^{i_0}} y_{i_1}^{s_{i_1}} 
\cdots y_{i_n}^{s_{i_n}} x_{k_2} \cdots x_{k_m}
$$
and if $i_1 \leq k_1$ then 
$$
\widetilde{x}= \frac{1-(-t)^e}{1-t^2}\gamma_{2m,y_{i_1}} y_{i_1}^{s_{i_1}-1} y_{i_2}^{s_{i_2}} 
\cdots y_{i_n}^{s_{i_n}} x_{k_2} \cdots x_{k_m}.
$$

Now we suppose $e$ odd and let $j$ be an integer be such that $p^j \mid e$ and $p^{j+1} \nmid e$.
If $r > 0$ then set again:
$$
\widetilde{x}= \frac{1-(-t)^e}{1+t}z_{2m+1} h^{r-1} y_{i_1}^{s_{i_1}} 
\cdots y_{i_n}^{s_{i_n}} x_{k_1} \cdots x_{k_m},
$$
if $r=0$ and $k_1< i_1$ then
$$
\widetilde{x}= \frac{1-(-t)^e}{(1+t)^{p^{\min (j,i_0)}}}z_{2m+2p^{i_1}} y_{i_1}^{s_{i_1}} 
\cdots y_{i_n}^{s_{i_n}} x_{k_2} \cdots x_{k_m}
$$
and if $i_1 \leq k_1$ then 
$$
\widetilde{x}= \frac{1-(-t)^e}{(1+t)^{\min(p^j,(p-1)p^{i_1-1})}}\gamma_{2m,y_{i_1}} y_{i_1}^{s_{i_1}-1} y_{i_2}^{s_{i_2}} 
\cdots y_{i_n}^{s_{i_n}} x_{k_2} \cdots x_{k_m}.
$$

It is clear from the definitions that $\delta (\widetilde{x}) = x$ and 
that the $\F_p[t^\pmu]$-module generated by $\widetilde{x}$ is isomorphic to the submodule 
$$\Ker(\F_p[t^\pmu]x \stackrel{(1-(-t)^e)}{\too} \F_p[t^\pmu]x)$$ of $H_i(\Art_{\Bx_r}, \F_p[t^\pmu])$.
Hence the same argument of Proposition \ref{prop:splitting} holds: the map 
$$
\delta: H_i(\Art_{\Bx_r}, \F_p[t]/(1-(-t)^e)) \to \Ker\left(
H_i(\Art_{\Bx_r}, \F_p[t^\pmu]) \stackrel{(1-(-t)^e)}{\too} H_i(\Art_{\Bx_r}, \F_p[t^\pmu])\right)
$$
has a section and the statement of the proposition holds, with $h_*(r,e)$ generated by the elements
of type $\pi_*(x)$ and $h'_*(r,e)$ generated by the elements of type $\widetilde{x}$.
\end{proof}

We can state the result for the $\F_p$-homology of $\B(2e,e,r)$ as follows:
\begin{theo} \label{prop:b2eerFp}
Let $p$ be an odd prime.
$$H_*(\Art_{\Bx_*}, \F_p[t]/(1-(-t)^e)) = $$
$$
= \left(
\begin{array}{c}
A_\infty [1,0]  \oplus  \\
\bigoplus_{n \geq 1}K_{2h_p(n)} [2(pn-p^{h_p(n)})+1, 2(pn-p^{h_p(n)}] \\
\bigoplus_{\hspace{-1,5cm}
\begin{array}{c}
{}_{n \geq 1,}\\{}_{h_p(n) = h_p(n+p^{h_p(n)})}
\end{array}
} \hspace{-1,5cm}I_{2h_p(n)+1} [2n+1, 2n]
\end{array}
\right)
\otimes \F_p[t]/(1-(-t)^e)\{2\} . 
$$\qed
\end{theo}

We can now give the Poincar\'e polynomial of the homology $H_*(\B(2e,e,r), \F_p)$. 
If we call $P_p(\B(2e,e,r),u) = \sum_{i=0}^\infty \dim_{\F_p}H_i(\B(2e,e,r),\F_p) u^i$ such a polynomial, 
we can consider the series in two variables
$$
P_p(\B(2e,e,*),u,v) = \sum_{r=0}^\infty P_p(\B(2e,e,r),u) v^r.
$$

The Poincar\'e series for the bigraded ring $A_\infty/(1-(-t)^e)$ is given, for $e$ odd, by
$$
P_{A_\infty(\F_p),e}(u,v) = 
%
%
\frac{1}{1-v} \prod_{i \geq 1}\frac{1}{1-u^{2p^i-2}v^{2p^i}} \prod_{j\geq0} (1 + u^{2p^{j}-1}v^{2p^j}) 
$$
end for $e$ even by
$$
P_{A_\infty(\F_p),e}(u,v) = e + 
%
%
\prod\!{}_0^p + \sum_{0}^\infty\!{}_1^p(e) + \sum_{0}^\infty\!{}_2^p(e)
$$
where we define the following terms:
$$
\prod\!{}_0^p = v \frac{1}{1-v} \prod_{i \geq 1}\frac{1}{1-u^{2p^i-2}v^{2p^i}} \prod_{j\geq0} (1 + u^{2p^{j}-1}v^{2p^j}),$$
$$
\sum_{k_1}^{k_2}\!{}_1^p(e) = \sum_{r = k_1}^{k_2} 2p^{\min(h_p(e),r)} u^{2p^r-1}v^{2p^r} \prod_{i \geq r+1}\frac{1}{1-u^{2p^i-2}v^{2p^i}} \prod_{j\geq r} (1 + u^{2p^{j}-1}v^{2p^j})
$$
and
$$
\sum_{k_1}^{k_2}\!{}_2^p(e) = \sum_{r = k_1}^{k_2} 2{\min(p^{h_p(e)},(p-1)p^{r-1})}u^{2p^r-2}v^{2p^r} \prod_{i \geq r}\frac{1}{1-u^{2p^i-2}v^{2p^i}} \prod_{j\geq r} (1 + u^{2p^{j}-1}v^{2p^j}).
$$

The Poincar\'e series of the ideal $I_{2a+1} \otimes \F_p[t]/(1+t^e)$ 
is given, for $e$ odd, by
$$
P_{I_{2a+1}(\F_p),e}(u,v) = \prod{}_0^p + \sum_0^a{}_1^p(1) + \sum_1^a{}_2^p(1)
$$
and, for $e$ even, by
$$
P_{I_{2a+1}(\F_p),e}(u,v) = \prod{}_0^p + \sum_0^a{}_1^p(e) + \sum_1^a{}_2^p(e).
$$
The Poincar\'e series of the ideal $K_{2a} \otimes \F_p[t]/(1+t^e)$ 
is given, for $e$ odd, by
$$
P_{K_{2a}(\F_p),e}(u,v) = \prod{}_0^p + \sum_0^{a-1}{}_1^p(1) + \sum_1^a{}_2^p(1)
$$
and, for $e$ even, by
$$
P_{K_{2a}(\F_p),e}(u,v) = \prod{}_0^p + \sum_0^{a-1}{}_1^p(e) + \sum_1^a{}_2^p(e).
$$

Hence we obtain:
\begin{cor}Let $p$ be an odd prime. The Poincar\'e polynomial of the homology of the groups $\B(2e,e,r)$ with $\F_p$ coefficients is given by:
$$
P_p(\B(2e,e,*),u,v) = $$ $$ (1+v) \left( v P_{A_\infty(\F_p),e}(u,v) + \sum_{n\geq 1} v^{2(pn-p^{h_p(n)})+1} u^{2(pn-p^{h_p(n)}}P_{K_{2h_p(n)}(\F_p),e}(u,v) + \right. $$ 
$$ \left. + \hspace{-1cm} \sum_{
\begin{array}{c}
{}_{n \geq 1,}\\{}_{h_p(n) = h_p(n+p^{h_p(n)})} 
\end{array}}  \hspace{-1cm} v^{2n+1} u^{2n}P_{I_{2h_p(n)+1}(\F_p),e}(u,v)  \right).
$$ \qed
\end{cor}

As an example of this computations we give in Table \ref{tab:homolF3}  the first homology groups of $\B(2e,e,r)$ with 
coefficients in the field $\F_3$ and the stable part up to homological dimension  $5$. For a prime $p \geq 5$ there's no $p$-torsion in the integral homology $H_*(\B(2e,e,r),\Z)$ for $r \leq 8$. Actually, for an odd prime, the first $p$-torsion in $H_*(\B(2e,e,r),\Z)$ appears for $r=2p$, as it comes from the classes associated to the generators $x_1, y_1$ in $H_*(\B(2e,e,r),\F_p)$.


\begin{table}[htb]
\begin{center}
\begin{tabular}{|c|c|c|c|c|c|c|c|c|c|c|c|c|c|}
\hline
$r$&\multicolumn{2}{c|}{$2$}&$3$& \multicolumn{2}{c|}{$4$} &$5$&\multicolumn{3}{c|}{$6$}&$7$& \multicolumn{2}{c|}{$8$} & $\geq 9$\\
\hline
$e \! \mod \! (m)$ & \!$0(2)$\!&\!$1(2)$\!&\!any\!&\!$0(2)$\!&\!$1(2)$\!&\!any\!&\!$0(6)$\!&\!$2,4(6)$\!&\!$1(2)$\!&\!any\!&\!$0(2)$\!&\!$1(2)$\!&\!any\!\\
\hline
$\dim H_0$& $1$  &$1 $&$1 $&$1 $&$1 $&$1 $&$1 $&$1 $&$1 $ & $1$&$1 $&$1 $ & $1$\\
\hline
$\dim H_1$ & $3$ & $2$                 & $2$ & $2$ & $2$ & $2$ & $2$ & $2$ & $2$ & $2$ & $2$ & $2$ & $2$\\
\hline
$\dim H_2$ & $2$& $1$   & $2$ & $2$ & $2$ & $2$ & $2$ & $2$ & $2$ & $2$& $2$ & $2$ & $2$ \\
\hline
$\dim H_3$ &\multicolumn{2}{c|}{$0$}  & $1$ & $3$ & $2$ & $2$ & $2$ & $2$ & $2$ & $2$ & $2$ & $2$ & $2$\\
\hline
$\dim H_4$ &\multicolumn{2}{c|}{$0$}  & $0$ & $2$ & $1$ & $2$ & $6$ & $4$ & $3$ & $3$ & $3$ & $3$ & $3$ \\
\hline
$\dim H_5$ &\multicolumn{2}{c|}{$0$} & $0$ & \multicolumn{2}{c|}{$0$} & $1$ & $11$ & $7$ & $4$ & $4$ &$6$ & $5$ &$5$ \\
\hline
$\dim H_6$ &\multicolumn{2}{c|}{$0$} & $0$ & \multicolumn{2}{c|}{$0$} & $0$ & $6$ & $4$ & $2$ & $2$ & $7$ & $5$ & \\
\hline
$\dim H_7$ &\multicolumn{2}{c|}{$0$} & $0$ & \multicolumn{2}{c|}{$0$} & $0$ & \multicolumn{3}{c|}{$0$}  & $1$ &$5$ & $3$ &\\
\hline
\end{tabular}
\end{center}
\caption{$\dim H_*(\B(2e,e,r); \F_3)$, $r<8$} 
\label{tab:homolF3}
\end{table}

\subsection{Stabilization}

There is a natural inclusion $j_r: \B(2e,e,r) \into \B(2e,e,r+1)$. The map $j_r$ is induced 
by the embedding of diagrams. Moreover it is induced by the analogous natural stabilization map for the Artin groups of type $\Bx_r$ as we have the commuting diagram


\begin{dummyenvironment}
\begin{center}
\begin{tabular}{c}
\xymatrix @R=2pc @C=2pc {
\B(2e,e,r) \ar@{^{(}->}[r]^{j_r} \ar@{^{(}->}[d]& \B(2e,e,r+1) \ar@{^{(}->}[d] \\
\Art_{\Bx_r} \ar@{^{(}->}[r] \ar[d]^{\phi_{e,r}} & \Art_{\Bx_{r+1}} \ar[d]^{\phi_{e,r+1}} \\
\Z/e \ar[r]^{\simeq}& \Z/e.
}
\end{tabular}
\end{center}
\end{dummyenvironment}

Hence the direct limit $\B(2e,e,\infty) :=\displaystyle \lim_{\longrightarrow}{}_r \B(2e,e,r)$ is a natural union of groups. The inclusion map $j_r$ correspond to the inclusion map for the algebraic complexes, hence we can compute the homology of the group $\B(2e,e,\infty)$ using the direct limit of the complexes for $\B(2e,e,r)$. It is easy to describe, as a corollary of Theorem \ref{prop:b2eerF2} and \ref{prop:b2eerFp}, the homology of the group $\B(2e,e,\infty)$. It turns out that the stable homology does not depend on the parameter $e$.

Let we define the graded modules
$$
sA_\infty(\F_2) = \F_2[t]/(1+t)[\overline{x}_1, \overline{x}_2, \overline{x}_3, \ldots]
$$
with $\dim \overline{x}_i = 2^i-1$ and for an odd prime $p$
$$
sA_\infty(\F_p) = \F_p[t]/(1+t)[\overline{y}_1, \overline{y}_2, \overline{y}_3, \ldots] \otimes \Lambda[\overline{x}_0, \overline{x}_1, \overline{x}_2, \ldots]
$$
with $\dim \overline{x}_i = 2p^i-1$, $\dim \overline{y}_i = 2p^i-2$.

The description of the stable homology is the following:
\begin{cor}\label{c:stabBeerF2}
The homology $H_*(\B(2e,e,\infty), \F_2)$ of the group $\B(2e,e,\infty)$ is isomorphic to the direct limit 
$\displaystyle \lim_{\longrightarrow}{}_r H_*(\B(2e,e,r),\F_2)$ and is given by the tensor product
$$
sA_\infty(\F_2) \otimes \F_2[w_1]
$$
where $w_1$ is an element of dimension $1$. 
Moreover the projection to the direct limit
$$H_i((\B(2e,e,r),\F_2) \to H_i(\B(2e,e,\infty), \F_2)$$ is an isomorphism for $r > 2 i$. \qed
\end{cor}

\begin{cor}\label{c:stabBeerFp}
The homology $H_*(\B(2e,e,\infty), \F_p)$ of the group $\B(2e,e,\infty)$ is isomorphic to the direct limit 
$\displaystyle \lim_{\longrightarrow}{}_r H_*(\B(2e,e,r),\F_p)$ and is given by the tensor product
$$
sA_\infty(\F_p) \otimes \F_p[w_1]
$$
where $w_1$ is an element of dimension $1$. 
Moreover the projection to the direct limit
$$H_i((\B(2e,e,r),\F_p) \to H_i(\B(2e,e,\infty), \F_p)$$ is an isomorphism for $r > (i-1)\frac{p}{p-1} + 2$. \qed
\end{cor}

We can then write the Poincar\'e polynomial for the stable homology.

We define
$$
P_{sA_\infty(\F_2)} (u):= \prod_{j \geq 1} \frac{1}{1-u^{2^j-1}} 
$$
and
$$
P_{sA_\infty(\F_p)} (u):= 
\prod_{i \geq 1} \frac{1}{1-u^{2p^i-2}} \prod{j\geq 0} (1+u^{2p^j-1}).
$$

\begin{cor} \label{c:stabPoinc}
For any prime $p$ the Poincar\'e polynomial for $H_*(\B(2e,e,\infty), \F_p)$ is $$P_{sA_\infty(\F_p)} (u)\frac{1}{1-u}. $$\qed
\end{cor}

\subsection{Some computations for torsion in integral homology} \label{ss:higher}

We are not able to compute the whole integral homology of the groups $\B(2e,e,r)$, but we provide a simple Bockstein computation in order to complete the proof of Theorem \ref{H2B2EER}.

According to the notation in the proof of Proposition \ref{prop:splitting}, the module
$H_2(\B(2e,e,r+4),\F_2)$ is generated by the cycles
$$
\pi_*(z_2x_0^{2+r}), \widetilde{x_0^{2+r}x_1}, \pi_*(x_0^rx_1^2).
$$
It is straightforward to check that the classes in $H_2(\B(2e,e,r+4),\Z)$ corresponding to $\pi_*(z_2x_0^{2+r})$ and $ \widetilde{x_0^{2+r}x_1}$ project to the generators of the rational homology $H_2(\B(2e,e,r+4),\Q)$, hence they generate torsion-free $\Z$-modules.

The generator $\pi_*(x_0^rx_1^2)$ is the image of an element $\rho \in H_2(\B(2e,e,r+4),\Z)$.
We claim that $\rho$ has $2$-torsion. In order to prove this we will use the Bockstein exact
sequence associated to the extension
$$
0 \to \Z_2 \stackrel{2}{\to} \Z_4 \to \Z_2 \to 0.
$$
In particular, we shows that the Bockstein differential $\beta_2$ maps
$\pi_*(x_0^r x_2) \mapsto \pi_*(x_0^2 x_1^2)$. It follows by standard argument that 
$\pi_*(x_0^2 x_1^2)$ generates a $\Z_2$-torsion class in $H_2(\B(2e,e,r+4),\Z_4)$ and hence, by the Universal Coefficients Theorem, in $H_2(\B(2e,e,r+4),\Z)$.

In order to compute the Bockstein $\beta_2(\pi_*(x_0^r x_2))$ recall that the class of
$\pi_*(x_0^r x_1^2)$ in $H_2(\B(2e,e,r+4),\F_2)$ is given by 
$$
\frac{\DB(z_2x_1)}{(1-t^2)}
$$
for $r=0$ and by
$$
\frac{\DB(z_1x_0^{r-1}x_1^2)}{(1+t)}
$$
for $r>0$ and that the class of
$\pi_*(x_0^r x_2)$ in $H_3(\B(2e,e,r+4),\F_2)$ is given by 
$$
\frac{\DB(z_4)}{(1-t^2)^2}
$$
for $r=0$
and by
$$
\frac{\DB(z_1x_0^{r-1}x_2)}{(1+t)}
$$
for $r >0$.
Their lifting to the complex with $\Z_4$ coefficients are given by
$$
\frac{\DB(z_4)-2z_2x_1}{(1-t^2)^2}
$$
and respectively
$$
\frac{\DB(z_1x_0^{r-1}x_2)-2z_1 x_0^{r-1}x_1^2}{(1+t)}.
$$
Now the claim about the Bockstein map follows since it is clear that the differential calculated for the chosen liftings give exactly the double of cycle $\pi_*(x_0^rx_1^2)$.

\vspace{5mm}

Our next purpose is to show that in general there can be $p^2$-torsion in the integral homology of $\B(2e,e,r)$. We will actually prove that there is a class of $4$-torsion in
$H_7(\B(16,8,8), \Z)$.

We consider the homology class $\widetilde{x}_3 \in H_8(\B(16,8,8), \F_2)$, that is represented by $\pi_*(z_8)$. It generates a $\F_2[t^\pmu]$-module isomorphic to $\F_2[t^\pmu]/(1-t^2)^4$. We want to compute the Bockstein $\beta_2$ of $\widetilde{x}_3$.

With the description given in section \ref{ss:b2eerintro} we can compute, with coefficients in $\Z[t^\pmu]$ 
$$
\DB z_8 = \DB (\overline{1}1^7) = \qbin{8}{2}_{-1} (1-t^2) \overline{1}1^501 + \qbin{8}{4}_{-1} (1-t^2)^2 \overline{1}1^301^3 +
\qbin{8}{6}_{-1} (1-t^2)^3 \overline{1}101^5 + \qbin{8}{8}_{-1}(1-t^2)^4 \overline{0}1^7.
$$
Then, considering this chain with coefficients in $\Z_4[t^\pmu]/(1-t^2)^4$ we get
$$
\DB z_8 = 2 (1-t^2)^2 z_4x_2 + 2t^4(1+t^4) x_3
$$
and dividing by two we get the following 
cycle in $H_7(\B(16,8,8), \F_2)$:
$$
\beta_2 (\widetilde{x}_3) = \widetilde{x_2^2} + t^4 (1-t^2)^2 \pi_*(x_3). 
$$
Notice that both $\widetilde{x_2^2}$ and $t^4(1-t^2)^2 \pi_*(x_3)$ generates a submodule of
$H_7(\B(16,8,8), \F_2)$ that is isomorphic to $\F_2[t^\pmu]/(1-t^2)^2$. It follows that 
 the kernel of the map 
$$\beta_2: H_8(\B(16,8,8), \F_2) \to H_7(\B(16,8,8), \F_2)
$$
is generated by the cycle $(1-t^2)^2 \widetilde{x}_3$ 
and hence $\ker {\beta_2}_{\mid H_8} \simeq \F_2[t^\pmu]/(1-t^2)^2$ is a $\F_2$ vector space of dimension $4$. Now recall that, according to Theorem \ref{t:lehr04}, we have that $\dim H_8(\B(16,8,8),\Q) = 2$.
The Bockstein spectral sequence implies then that there should be an element in $H_7(\B(16,8,8), \Z)$ that has at least $4$-torsion.

\section{Isomorphism and non-isomorphism results for $\B(2e,e,r)$} \label{s:iso}
We already recalled from \cite{BMR} that for $d > 1$ $\B(de,e,r) \simeq \B(2e,e,r)$.

In this section we want to study the groups of type $\B(de,e,r)$ from a more 
elementary point of view, 
in order to get some isomorphism and non-isomorphism result.

Let us start recalling, as in section \ref{s:b2eer}, the isomorphisms
$$
\B(d,1,r) = \Art_{\Bx_r}
$$
and 
$$
\B(de,e,r) = \Ker \phi_{e,r}
$$
where $\phi_{e,r}$ maps $\B(d,1,r) \to \Z/e$.
We can give a presentation for $ \Art_{\Bx_r}$ that is different from the one provided before 
(see \cite{kp02}).
We define $\tau = \overline{\sigma}_1 \sigma_2 \cdots \sigma_r$ and 
$\sigma_1 =\tau^{-1} \sigma_2 \tau$.
It is easy to check that 
$$
\tau \sigma_i \tau^{-1} = \sigma_{i+1}
$$
where the indexes are considered in $\Z/r$.

We have that the group $\Art_{\Bx_r}$ has a presentation with generators 
$\Cal{G} = \{ \tau, \sigma_i, i \in \Z/r \}$ and relations
$$
\Cal{R} = \{ \sigma_i \sigma_j = \sigma_j \sigma_i \mbox{ for } i \neq j \pmu, 
\sigma_i \sigma_{i+1} \sigma_i = \sigma_{i+1} \sigma_i \sigma_{i+1}, \tau \sigma_i \tau^{-1} = 
\sigma_{i+1} \}
$$

With this presentation the map $\phi_{e,r}$ maps $\tau \mapsto 1$, $\sigma_i \mapsto 0$ for all $i$.

We notice that the subgroup of $\Art_{\Bx_r}$ generated by the elements 
$\sigma_1, \ldots, \sigma_r$ is the Artin group of type 
$\widetilde{A}_{r-1}$, $\Art_{\widetilde{A}_{r-1}}$. 
Hence, if we write $\Z_{\tau^e}$ for the infinite cyclic group generated by $\tau^e$, 
where $\tau$ acts on $\Art_{\widetilde{A}_{r-1}}$
as before, we can write the following semidirect product decompositions:
$$
\B(d,1,r)= \Z_\tau \ltimes \Art_{\widetilde{A}_{r-1}}
$$
and
\begin{equation} \label{eq:Bdeer}
\B(de,e,r)= \Z_{\tau^e} \ltimes \Art_{\widetilde{A}_{r-1}}.
\end{equation}

According to \cite{kp02} and \cite{BMR} the center of $\B(de,e,r)$ is generated by 
$\beta(de,e,r) = (\tau^e)^{(r/r \wedge e)}$.
Hence it follows that in the quotient $\B(de,e,r)/Z(\B(de,e,r))$ there is an element, namely $(\tau^e)$, 
that has order at most $(r/r \wedge e)$ and is the image of a root of the generator of the center
of $\B(de,e,r)$. Now let us consider the map
$$
\lambda: \B(de,e,r) \to \Z/(r/r \wedge e)
$$
given by $(\tau^e) \mapsto 1$ and $\sigma_i \to 1$ for all $i$. This map passes to the quotient
$$
\overline{\lambda}: \B(de,e,r)/Z(\B(de,e,r)) \to \Z/(r/r \wedge e)
$$
and hence the order of $(\tau^e)$ in the quotient 
$\B(de,e,r)/Z(\B(de,e,r))$ is exactly $(r/r \wedge e)$.
The length function in $\B(de,e,r)$ tells us also that the generator of the center $\beta(de,e,r)$ can't 
have roots of order higher than $(r/r \wedge e)$. We have proved the following:
\begin{prop}
The groups $\B(de,e,r)$ and $\B(de',e',r)$ are not isomorphic if $$r \wedge e \neq r \wedge e'.$$ \qed
\end{prop}

From Equation (\ref{eq:Bdeer}) it is possible to deduce the following elementary result:
\begin{prop}
The group $\B(2e,e,r)$ is isomorphic to $\B(2e',e',r)$ if $e \simeq \pm e' \mod r$.
\end{prop}

\begin{proof}
 This is straightforward since the Dynkin diagram of ${\widetilde{A}_{r-1}}$ is an $r$-gon and we can suppose
without loss of generality that the vertices are numbered counterclockwise. Hence the element $\tau$ acts by conjugation 
rotating the $r$-gon by $\frac{2 \pi}{r}$ and the subgroup of $\mathrm{Inn}(\Art_{\widetilde{A}_{r-1}})$ generated by
$\tau$ is cyclic of order $r$. It follows that a conjugation by $\tau^e$ is equivalent to a conjugation by $\tau^{e'}$ 
if $e \simeq e' \mod r$. Moreover we can consider the automorphism $\varsigma$ of $\Art_{\widetilde{A}_{r-1}}$
given by $\varsigma(\sigma_i) = \sigma_{r+1-i}$. The map $\varsigma$ induces an isomorphism 
$$\overline{\varsigma}: \Z_{\tau^e} \ltimes \Art_{\widetilde{A}_{r-1}} \to \Z_{\tau^{-e}} \ltimes \Art_{\widetilde{A}_{r-1}}$$
given by $\overline{\varsigma}(n,w) = (-n, \varsigma(w))$.
\end{proof}












%




\section{Complexes from Garside theory}

We recall a few homological constructions from the
theory of Garside monoids and groups. Recall that
a Garside group $G$ is the group of fractions of a Garside monoid $M$,
where Garside means that $M$ satisfy several conditions for
which we refer to \cite{dehpar}. In particular, $M$ admits (left)
lcm's, and contains a special element, called the Garside element.
We denote $\mathcal{X}$ the set of atoms in $M$, assumed to be finite. The homology of $G$ coincides with the homology
of $M$. Garside theory provides two useful resolutions
of $\Z$ by free $\Z M$-modules.

The first one was defined in \cite{CMW}. Another one, with
more complicated differential but a smaller number of cells,
has been defined in \cite{DL}.

\subsection{The Dehornoy-Lafont complex}

Let $M$ be a Garside monoid with a finite set of atoms $\mathcal{X}$. We
choose an arbitrary linear order $<$ on $\mathcal{X}$. For $m \in M$,
denote $\md(m)$ denote the smaller element in $\mathcal{X}$
which divides $m$ on the right ($m = a\md(m)$ for some
$a \in M$). Recall that $\lcm(x,y)$ for $x,y \in M$
denotes the least common multiple on the left, that is $v = gx = h y$
implies $v = j \lcm(x,y)$ for some $j \in M$. If $A = (x,B)$
is a list of elements in $M$ we define inductively $\lcm(A) = \lcm(x, \lcm(B))$.

A $n$-cell is a $n$-tuple $[x_1,\dots,x_n]$ of elements
in $\mathcal{X}$ such that $x_1 < \dots < x_n$ and $x_i
= \md (\lcm(x_i,x_{i+1},\dots,x_n))$. Let $\mathcal{X}_n$
denote the set of all such $n$-cells. By convention $\mathcal{X}_0
= \{ [\emptyset] \}$. The set $C_n$ of $n$-chains is the free $\Z M$-module
with basis $\mathcal{X}_n$. A differential $\partial_n : C_n \to C_{n-1}$
is defined recursively through two auxiliary $\Z$-module homomorphisms
$s_n : C_n \to C_{n+1}$ and $r_n : C_n \to C_n$. Let
$[\alpha,A]$  be a $(n+1)$-cell,
with $\alpha \in \mathcal{X}$ and $A$ a $n$-cell. We let $\alpha_{/A}$
denote the unique element in $M$ such that
$(\alpha_{/A}) lcm(A) = lcm(\alpha,A)$. The defining equations for $\partial$
and $r$ are the following ones.
$$
\partial_{n+1}[\alpha,A] = \alpha_{/A} [A] - r_n(\alpha_{/A} [A]), \ \ r_{n+1} = s_n \circ
\partial_{n+1}, \ \ r_0(m[\emptyset]) = [\emptyset].
$$
In order to define $s_n$, we say that $x [A]$ for $x \in M$
and $A$ a $n$-cell is \emph{irreducible} if $x = 1$ and $A = \emptyset$,
or if $\alpha = \md(x \lcm(A))$ coincides with the first coefficient
in $A$. In that case, we let $s_n(x [A]) = 0$, and otherwise
$$
s_n(x[A]) = y[\alpha,A] + s_n(y r_n(\alpha_{/A} [A]))
$$
with $x = y \alpha_{/A}$.

\subsection{The Charney-Meyer-Wittlesey complex}

Let again $G$ denote the group of fractions of a Garside monoid $M$,
with Garside element $\Delta$. Let $\mathcal{D}$ denote
the set of simple elements in $M$, namely the (finite) set of proper
divisors of $\Delta$. We let $\mathcal{D}_n$ denote
the set of $n$-tuples $[\mu_1|\dots|\mu_n]$ such that each $\mu_i$
as well as the product $\mu_1 \dots \mu_n$ lie in $\mathcal{D}$.
The differential from the free $\Z M$-modules $\Z M \mathcal{D}_n$
to $\Z M \mathcal{D}_{n-1}$ is given by
$$
\partial_n [\mu_1|\dots|\mu_n] = \mu_1 [\mu_2|\dots|\mu_n]
+ \sum_{i=1}^{n-1} (-1)^i [\mu_1,\dots,\mu_i \mu_{i+1},\dots,\mu_n]
+ (-1)^n [\mu_1|\dots|\mu_{n-1}] 
$$
This complex in general has larger cells than the previous one.
Its main advantage for us is that the definition of the
differential is simpler, and does not involve many recursion
levels anymore.

\subsection{Application to the exceptional groups}

When $W$ is well-generated, meaning that it can be generated
by $n$ reflections, where $n$ denotes the rank of $W$, then $B$ is
the group fractions of (usually) several Garside monoids that
generalize the Birman-Ko-Lee monoid of the usual braid groups.
These monoids have been introduced by D. Bessis in \cite{BESSISKPI1} 
and call there dual braid monoids. They are determined by the choice
of a so-called Coxeter element $c$. Such an element is regular, meaning
that it admits only one eigenvalue different from 1 with the
corresponding eigenvector outside the reflection hyperplanes. A
Coxeter element is a regular element with eigenvalue
$\exp(2 \ii \pi/h)$, where $h$ denotes the (generalized)
Coxeter number for $W$, namely its highest degree as a reflection group.

The corresponding Garside monoid $M_c$ is then generated by some
set $R_c$ of braided reflections with relations
of the form $r r' = r' r''$ (see \cite{BESSISKPI1} for more details).
The above complexes for these monoids have been implemented
by Jean Michel and the second author, using the (development version of)
the CHEVIE package for GAP3. The chosen Coxeter element
are indicated in Table \ref{coxeterdualmonoids}, in terms
of the usual presentations of these groups (see \cite{BMR} for
an explanations of the diagrams).

\font\hugecirc = lcircle10 scaled 1950
\font\largecirc = lcircle10 scaled  \magstep 3
\font\mediumcirc = lcircle10 scaled \magstep 2
\font\smallcirc = lcircle10 
\def\BBcirc{$\hbox{\hugecirc\char"6E}$}
\def\Bcirc{$\hbox{\largecirc\char"6E}$}
\def\sBcirc{$\hbox{\mediumcirc\char"6E}$}
\def\sbcirc{$\hbox{\smallcirc\char"6E}$}
\def\ucirc{$\hbox{\mediumcirc\char"13}$}
\def\lcirc{$\hbox{\mediumcirc\char"12}$}
\def\ncnode#1#2{{\kern -0.4pt\mathop\bigcirc\limits_{#2}\kern-8.6pt{\scriptstyle#1}\kern 2.3pt}}
\def\dbar#1pt{{\rlap{\vrule width#1pt height2pt depth-1pt} 
                 \vrule width#1pt height4pt depth-3pt}}
\def\nnode#1{{\kern -0.6pt\mathop\bigcirc\limits_{#1}\kern -1pt}}
\def\bar#1pt{{\vrule width#1pt height3pt depth-2pt}}
\def\cnode#1{{\kern -0.4pt\bigcirc\kern-7pt{\scriptstyle#1}\kern 2.6pt}}
\def\node{{\kern -0.4pt\bigcirc\kern -1pt}}
\def\snode{{\kern -0.4pt{\scriptstyle\bigcirc}\kern -1pt}}
\def\tbar#1pt{{\rlap{\vrule width#1pt height1pt depth0pt}
           \rlap{\vrule width#1pt height3pt depth-2pt}
\vrule width#1pt height5pt depth-4pt}}
\def\overmark#1#2{\kern -1.5pt\mathop{#2}\limits^{#1}\kern -2pt}
\def\trianglerel#1#2#3{
 \nnode#1\bar14pt\kern-13pt\raise7.5pt\hbox{$\displaystyle \underleftarrow 6$}
  \kern 2pt\nnode#2
 \kern-29pt\raise9.5pt\hbox{$\diagup$}
 \kern -2pt
 \raise17.5pt\hbox{$\node$\rlap{\raise 2pt\hbox{$\kern 1pt\scriptstyle #3$}}}
 \kern -1pt\raise9.5pt\hbox{$\diagdown$}
 \kern 3pt
}
\def\vertbar#1#2{\rlap{$\nnode{#1}$}
                 \rlap{\kern4pt\vrule width1pt height17.3pt depth-7.3pt}
\raise19.4pt\hbox{$\node$\rlap{$\kern 1pt\scriptstyle#2$}}}
\def\stacktwo#1#2{\genfrac{}{}{0pt}{0}{#1}{#2}}
\def\stackthree#1#2#3{\genfrac{}{}{0pt}{0}{#1}{\genfrac{}{}{0pt}{0}{#2}{#3}}}

\begin{table}
$$
\begin{array}{c||c|c|c|c|c}
\mbox{Group} & G_{24} & G_{27} & G_{29} & G_{33} & G_{34} \\
\hline
\mbox{Diagram} &  \nnode s\rlap{\kern 2.5pt{\raise 6pt\hbox{$\triangle$}}}
 \dbar14pt\nnode t
 \kern-26.2pt\raise8.5pt\hbox{$\diagup$}
 \kern -3.1pt
 \raise16.7pt\hbox{$\node$\rlap{\raise 2pt\hbox{$\kern 1pt\scriptstyle u$}}}
 \kern -3pt\raise8.3pt\hbox{$\diagdown$}
 \kern -7.6pt\raise9pt\hbox{$\diagdown$}
 \kern 4pt &  \nnode s\rlap{\kern 2.5pt{\raise 6pt\hbox{$\triangle$}}}
 \dbar14pt\nnode t
 \rlap{\kern-6pt\raise11pt\hbox{$\scriptstyle 5$}}
 \kern-27pt\raise9pt\hbox{$\diagup$}
 \kern -2pt
 \raise16.5pt\hbox{$\node$\rlap{\raise 2pt\hbox{$\kern 1pt\scriptstyle u$}}}
 \kern -1pt\raise9pt\hbox{$\diagdown$}
 \kern 4pt &  \nnode s\bar10pt
 \nnode t\rlap{\kern 1pt\raise9pt\hbox{$\underleftarrow{}$}}
        \rlap{\kern 6.5pt\vrule width1pt height16pt depth-11pt
\kern 1pt \vrule width1pt height16pt depth-11pt}
\dbar14pt\nnode u
 \kern-27pt\raise9.5pt\hbox{$\diagup$}
 \kern -3pt\raise18.5pt
\hbox{$\node$\rlap{\raise 2pt\hbox{$\kern 1pt\scriptstyle v$}}}
 \kern -2pt\raise9.5pt\hbox{$\diagdown$} & \nnode s\bar10pt\trianglerel tvu\bar10pt\nnode w & 
\nnode s\bar10pt\trianglerel tvu\bar10pt\nnode w\bar10pt\nnode x \\
\hline
\mbox{Coxeter\ element} & stu & uts & stvu & wvtsu & xwuvts \\
\hline
\end{array}
$$

\caption{Coxeter elements for dual monoids.}
\label{coxeterdualmonoids}
\end{table}

Using the HAP package for GAP4 we then obtained the homologies
described in Table \ref{tableexc} (we recall in Table \ref{tableexcCox} the
ones obtained earlier by Salvetti for the Coxeter groups)
except for the groups $G_{12}, G_{13}, G_{22}, G_{31}$, which are not well-generated,
as well as the $H_3(B,\Z)$ of type $G_{33}$. When $W$ has type $G_{13}$,
the group $B$ is the same as when $W$ has Coxeter type $I_2(6)$,
and the result is known. For $G_{12}$ and $G_{22}$ one can
use Garside monoids introduced by M. Picantin in \cite{THESEPIC}.

A complex for $G_{31}$ can be obtained from the theory of Garside categories
by considering it as the the centralizer of some regular element in the Coxeter group
$E_8$. This viewpoint was used in \cite{BESSISKPI1} in order
to prove that the corresponding spaces $X$ and $X/W$ are $K(\pi,1)$.
More precisely, a simplicial complex (reminiscent from
the Charney-Meyer-Wittlesey complex) is constructed in \cite{BESSISKPI1},
which is homotopically equivalent to $X/W$. From this construction,
we got a complex from an implementation by Jean Michel in CHEVIE.

However, for $G_{31}$, $G_{33}$ and $G_{34}$, the complexes obtained
are too large to be dealt with completely through usual computers and software.
The one missing for $G_{31}$ and $G_{33}$ are the middle homology
$H_2(B,\Z)$ for $G_{31}$ and $H_3(B,\Z)$ for $G_{33}$.
For $G_{33}$
the Dehornoy-Lafont complex for $G_{33}$ is however computable
in reasonable time, and its small size enables to compute the whole homology
by standard methods. For $G_{31}$, for which there is so far no
construction analogous to the Dehornoy-Lafont complex, we used the
following method for computing $H_2(B,\Z)$. 

We first get $H_2(B,\Q) = 0$
by computing the second Betti number from the lattice. Indeed, recall from \cite{OT} (cor. 6.17, p.223) that the Betti
numbers of $X/W$ can be in principle computed from the lattice
of the arrangement. Precisely, the second Betti
number of $X/W$ is given by $\sum_{Z \in T_2} |\mathcal{H}_Z/W_Z| -1$
where $T_i$ is a system of representatives
modulo $W$ of codimension $i$ subspaces in the arrangement lattice ; for
$Z$ such a subspace, $\mathcal{H}_Z = \{ H \in \mathcal{A} \ | \ 
H \supset Z \}$, $W_Z = \{ w \in W \ | \ w(Z) = Z \}$.
More generally, the $i$-th Betti number is given by
$$ (-1)^i \sum_{Z \in T_i} \sum_{\sigma \in U_Z} (-1)^{d(\sigma)}
$$
where $U_Z$ is the set of classes modulo $W$ of the
set of simplices of the augmented Folkman complex of the lattice $\mathcal{A}_z$,
and $d(\sigma)$ denotes the dimension of a cell. The Folkman complex of a lattice
is defined (see \cite{OT}) as the complex of poset
obtained by removing the minimal and maximal elements of the
lattice ; when the maximal codimension of the lattice is 1, then
the Folkman complex is empty. The augmented Folkman complex
is defined by adding to the Folkman complex one $G$-invariant
simplex of dimension $-1$. In the case of $G_{31}$ the computation of
this formula is doable and we get 0 for the second Betti number.

We then reduce our original complex mod $p^r$, for $p^r$ small enough so that
we can encode each matrix entry inside one byte. Then we wrote a
C program to compute $H_2(B,\Z_4) = H_2(B,\Z_2) = \Z_2$,
$H_2(B,\Z_9) = H_2(B,\Z_3) = \Z_3$
and $H_2(B,\Z_5) = 0$ (the matrix of $d_3$ has size $11065 \times 15300$).
Since $G_{31}$ has order $2^{10} .3^2. 5$ and $H_*(P)$ is torsion-free,
for $p \not\in \{ 2, 3, 5 \}$ we have $H_2(B,\Z_p) = H_2(P,\Z_p)^W = (H_2(P,\Z)^W) \otimes \Z_p$.
But $0=H_2(B,\Q) = H_2(P,\Q)^W = H_2(P,\Z)^W \otimes \Q$
whence $H_2(P,\Z)^W = 0$ and $H_2(B,\Z_p) = 0$.
Now $H_1(B,\Z) = \Z$ is torsion-free, hence $H_2(B,\Z_n) \simeq
H_2(B,\Z) \otimes \Z_n$ for any $n$ by the universal coefficients
theorem. Since $H_2(B,\Z)$ is a $\Z$-module of finite type
this yields $H_2(B,\Z) = \Z_6$ and completes the computation for $G_{31}$.


\subsection{Embeddings between Artin-like monoids}
\label{sectembed}
We end this section by proving a few lemmas concerning
submonoids, which will be helpful in computing differentials
in concrete cases.

We consider Garside monoids with set of generators $S$
and endowed with a length function, namely a monoid morphism
$\ell : M \to \N = \Z_{\geq 0}$ such that $\ell(x) = 0 \Leftrightarrow
x = 1$ and $\ell(s) = 1$ for all $s \in S$. We consider the divisibility
relation on the left (that is $U | V$ means $\exists m \ V = U m$) and recall that such a monoid admit lcm's
(on the left).

Let $M,N$ be two such monoids, and $\varphi : M \to N$
a monoid morphism such that

\begin{enumerate}
\item $\forall s \in S \ | \ \varphi(s) \neq 1$
\item $\forall s,t \in S \ \lcm(\varphi(s),\varphi(t)) = \varphi(\lcm(s,t))$
\end{enumerate}

The following results on such morphisms are basically due to J. Crisp,
who proved them in \cite{crisp} in the case of finite-type Artin groups.

\begin{lemma} Let $U,V \in M$. If $\varphi(U) | \varphi(V)$ then $U | V$.
\end{lemma}
\begin{proof} By induction on $\ell(V)$. Since $\forall s \in S \ \ell(\varphi(s)) \geq 1 = \ell(s)$,
we have $\ell(\varphi(U)) \geq \ell(U)$. Since $\varphi(U) | \varphi(V)$, we
have $\ell(\varphi(U)) \leq \ell(\varphi(V))$ hence $\ell(U) \leq \ell(\varphi(V))$.
Hence $\varphi(V) = 1$ implies $\ell(U) = 0$ and $U = 1$, which settles
the case $\ell(V) = 0$.

We thus assume $\ell(V) \geq 1$. The case $U=1$ being clear,
we can assume $U \neq 1$. Then there exists $s,t \in S$
with $s | U$ and $t|V$. It follows that $\varphi(t) | \varphi(V)$
and $\varphi(s) | \varphi(U) | \varphi(V)$, hence
$\lcm(\varphi(s),\varphi(t)) | \varphi(V)$.

Now $\lcm(s,t) = t m$ for some $m \in M$ and $V = t V'$
for some $V' \in M$, hence
$\varphi(t) \varphi(m) | \varphi(V) = \varphi(t) \varphi(V')$
and this implies $\varphi(m) | \varphi(V')$ by cancellability
in $M$. Since $\ell(V') < \ell(V)$, from the induction assumption
follows that $m | V'$ hence $tm | V$ that is $\lcm(s,t) | V$.
In particular we get $s | V$. Writing $V = s V''$ and $U = s U'$
for some $V'',U' \in M$,
the assumption $\varphi(U) | \varphi(V)$ implies
$\varphi(U') | \varphi(V'')$ by cancellability, and
then $U' | V''$ by the induction assumption. It follows
that $U|V$ which proves the claim.
\end{proof}

The lemma has the following consequence.

\begin{lemma} The morphism $\varphi : M \to N$ is injective. If $G_M$, $G_N$
denotes the group of fractions of $M,N$, then $\varphi$ can be extended
to $\widetilde{\varphi} : G_M \into G_N$.
\end{lemma}
\begin{proof}
Let $U,V \in M$ with $\varphi(U) = \varphi(V)$. By the lemma
we get $U|V$ and $V | U$. This implies $\ell(U) = \ell(V)$
hence $U = V$. Composing $\varphi : M \to N$
with the natural morphism $N \into G_N$ yields
a monoid morphism
$M \to G_N$. Since $G_N$ is a group this morphism factors through
the morphism $M \to G_M$ and this provides $\widetilde{\varphi} : G_M \to G_N$.
Let $g \in \Ker \widetilde{\varphi}$. Since $g \in G_M$ there exists
$a,b \in M$ with $g = ab^{-1}$ hence $\varphi(a) = \varphi(b)$,
$a=b$ and $g = 1$.
\end{proof}

We consider the following extra assumption on $\varphi$. We assume that,
for all $m \in M$ and $n \in N$, $n | m$ implies

We can now identify in this $M,N,G_M$
to subsets of $G_N$. We consider the following extra
assumption. We assume that, for all $m \in M, n \in N$,
if $n$ divides $m$ in $N$ then $n \in M$.

\begin{lemma} Under this assumption, $U,V$ in $M$
have the same lcm in $M$ and in $N$. Moreover,
$M = N \cap G_M$.
\end{lemma}
\begin{proof}
Since $\lcm_M(U,V)$ divides $U,V$ in $N$, we have 
that $\lcm_N(U,V))$ divides  $\lcm_M(U,V)$ in $N$.
Conversely, since $\lcm_N(U,V)$ divides $U$ in $N$ 
and $U \in M$, by the assumption we get $\lcm_N(U,V) \in M$.
From the lemma we thus get that $\lcm_N(U,V)$ divides
$U$ and $V$ in $M$ hence $\lcm_M(U,V)$ divides $\lcm_N(U,V)$
in $N$. It follows that $\lcm_M(U,V) = \lcm_N(U,V)$.

We have $M \subset N \cap G_M$. Let $n \in N \cap G_M$.
Since $n \in G_M$ there exists $a,b \in M$
with $n = ab^{-1}$, hence $nb = a \in M$. Hence
$n \in N$ divides $a \in M$ in $M$. By the assumption
we get $n \in M$ and the conclusion.
\end{proof}

\section{The groups $\B(e,e,r)$} \label{s:beer}

\subsection{The Corran-Picantin monoid}

We denote $\B(e,e,r)$ for $e \geq 1$ and $r \geq 2$ the braid group associated
to the complex reflection group $G(e,e,r)$.
The $\B(e,e,r)$ are the group of fractions
of a Garside monoid introduced by R. Corran
and M. Picantin (see \cite{corpic}). This monoid,
that we denote $M(e,e,r)$, has generators (atoms)
$t_0,t_1,\dots,t_{e-1}, s_3, s_4, \dots, s_r$ and relations
\begin{enumerate}\label{rels}
\item $t_{i+1} t_i = t_{j+1} t_j$, with the convention $t_e = t_0$,
\item $s_3 t_i s_3 = t_i s_3 t_i$
\item $s_k t_i  = t_i s_k$ for $k \geq 4$
\item $s_k s_{k+1} s_k = s_{k+1} s_k s_{k+1} $ for $k \geq 3$
\item $s_k s_l = s_l s_k$ when $| l-k | \geq 2$.
\end{enumerate}

\subsection{Link with the topological definition}

The connection between this monoid and the group $\B(e,e,r)$
defined as a fundamental group is quite indirect. In \cite{BMR}
a first presentation is obtained by combining embeddings into usual
braid groups, fibrations and coverings. The presentation used here
is deduced from this one in a purely algebraic matter,
by adding generators in order to get a Garside presentation.
Although it is folklore the description of all generators as
braided reflection does not appear in the literature (see
however \cite{BESSISCORRAN} for a statement without proof in a related context).

In order to provide this connection,
we need to recall the way these generators
are constructed. For clarity, we stick to the notations of \cite{BMR} ;
in this paper, the authors introduce 4 different spaces,
$\mathcal{M}(r+1) = \{ (z_0,\dots,z_r) \in \C^{r+1} \ | \ z_i \neq z_j \}$,
$\mathcal{M}^{\#}(m,r) = \{ (z_1,\dots,z_r) \in \C^r \ | \ z_i \neq 0,
z_i /z_j \not\in \mu_m \}$, $\mathcal{M}(e,r) = \{ (z_1,\dots,z_r) \in \C^r \
| \ z_i \not\in \mu_e z_j \}$, and $\mathcal{M}^{\#}(r) = \{ (z_1,\dots,z_r) \in \C^r \ | \ z_i \neq 0 \}$,
where $\mu_n$ denotes the set of $n$-th roots of 1 in $\C$.
We have a Galois covering $r : \mathcal{M}^{\#}(m,r)
\to \mathcal{M}^{\#}(r) = \mathcal{M}^{\#}(m,r) / (\mu_m)^r$,
a locally trivial fibration $p : \mathcal{M}(r+1) \to \mathcal{M}^{\#}(r)$
with fiber $\C$ given by $(z_0,\dots,z_r) \mapsto (z_0 - z_1,\dots,
z_0-z_r)$, and a natural action of $\mathfrak{S}_r$ on $\mathcal{M}(r+1)$
that leaves the $(r+1)$-st coordinate fixed. We choose
a fixed point $x \in \mathcal{M}(r+1)/\mathfrak{S}_r$,
and a lift $\widetilde{p(x)}$ of $p(x) \in \mathcal{M}^{\#}(r)/\mathfrak{S}_r$
in $\mathcal{M}^{\#}(d,r)/G(d,1,r) = (\mathcal{M}(d,r)/(\mu_m)^r)/\mathfrak{S}_r$.
We get an isomorphism $\psi : \pi_1(\mathcal{M}^{\#}(d,r)/G(d,1,r),
\widetilde{p(x)}) \to \pi_1(\mathcal{M}(r+1)/\mathfrak{S}_r,x)$
by composing the isomorphisms induced by $r$ and $p$.
$$
\xymatrix{
 \pi_1(\mathcal{M}^{\#}(d,r)/G(d,1,r), \widetilde{p(x)}) \ar[dr]_{r}^{\simeq} \ar[rr]^{\psi} & &
\ar[dl]^{p}_{\simeq} \pi_1(\mathcal{M}(r+1)/\mathfrak{S}_r, x) \\
 & \pi_1(\mathcal{M}^{\#}(r)/\mathfrak{S}_r, p(x)) & \\
}
$$
Since $\pi_1(\mathcal{M}^{\#}(d,r)/G(d,1,r)) = \B(d,1,r)$,
$\psi$ identifies the latter group with $\pi_1(\mathcal{M}(r+1)
/\mathfrak{S}_r)$. The generators of $\B(d,1,r)$ are then obtained
in \cite{BMR}
by taking the preimages under $\psi$ and
the covering of $\mathcal{M}(r+1)/\mathfrak{S}_r \to \mathcal{M}(r+1)/\mathfrak{S}_{r+1}$.
Note that this covering provides an injection between fundamental
groups, hence an embedding $\tilde{\psi} : \B(d,1,r) \into \Br(r+1)$,
where $\Br(r+1)$ denotes the usual braid group on $r+1$ strands.
We choose for base point in $\mathcal{M}(r+1)$ the point
$x = (0,x_1,\dots,x_r)$ with the $x_i \in \R$ and $x_{i+1} \ll x_i$,
and for generators of the usual braid group $\mathcal{M}(r+1)/\mathfrak{S}_{r+1}$
the elements $\xi_0, \xi_1,\dots,\xi_{r-1}$ as described below :
\begin{center}
\includegraphics{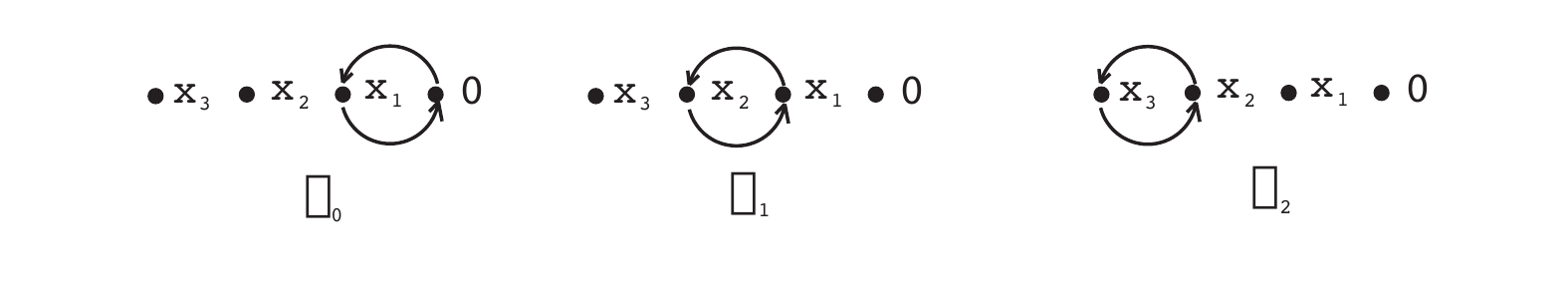}
\end{center}
Then (see \cite{BMR}), the group $\pi_1(\mathcal{M}(r+1)/\mathfrak{S}_r)$
is generated by $\xi_0^2,\xi_1,\dots,\xi_{r-1}$. The element
$\xi_0^2$ is the class in $\mathcal{M}(r+1)$ of the loop
$(\frac{x_1}{2}(1-e^{2 \ii \pi t}), \frac{x_1}{2}( e^{2 \ii \pi t} +1),x_2,
\dots,x_r )$. Taking its image by $p$ provides a loop
based at $(-x_1,-x_2,\dots,-x_r)$ described by $(-x_1 e^{2 \ii \pi t},
\frac{x_1}{2} (1 - e^{2 \ii \pi t}) -x_2,\dots,\frac{x_1}{2} (1 - e^{2 \ii \pi t}) - x_r)$.
Since $|x_i | \ll |x_{i+1}|$, this path is homotopic to
$(-x_1 e^{2 \ii \pi t}, -x_2, \dots, -x_r)$, both in $\mathcal{M}^{\#}(r)$
and in $\mathcal{M}^{\#}(r)/\mathfrak{S}_r$.
Letting $a_i = -x_i$, we have $0 < a_1 < a_2 < \dots < a_r$, and we choose
$y = \widetilde{p(x)}$ to be $y = (a_1^{\frac{1}{d}},\dots,
a_r^{\frac{1}{d}})$. The above loop thus lifts under $r$ to
the path $(a_1 e^{2 \ii \pi t/d},a_2,\dots,a_r)$ in $\mathcal{M}^{\#}(d,r)$.
By definition of $\psi$,
the class of this path $\sigma = \psi^{-1}(\xi_0^2)$.
Similarly, we can determine $\psi^{-1}(\xi_i)$ when $i \geq 1$ : the image
of $\xi_i$ under $p$ is a path in $\mathcal{M}^{\#}(r)$ homotopic to
\begin{center}
\resizebox{!}{3cm}{\includegraphics{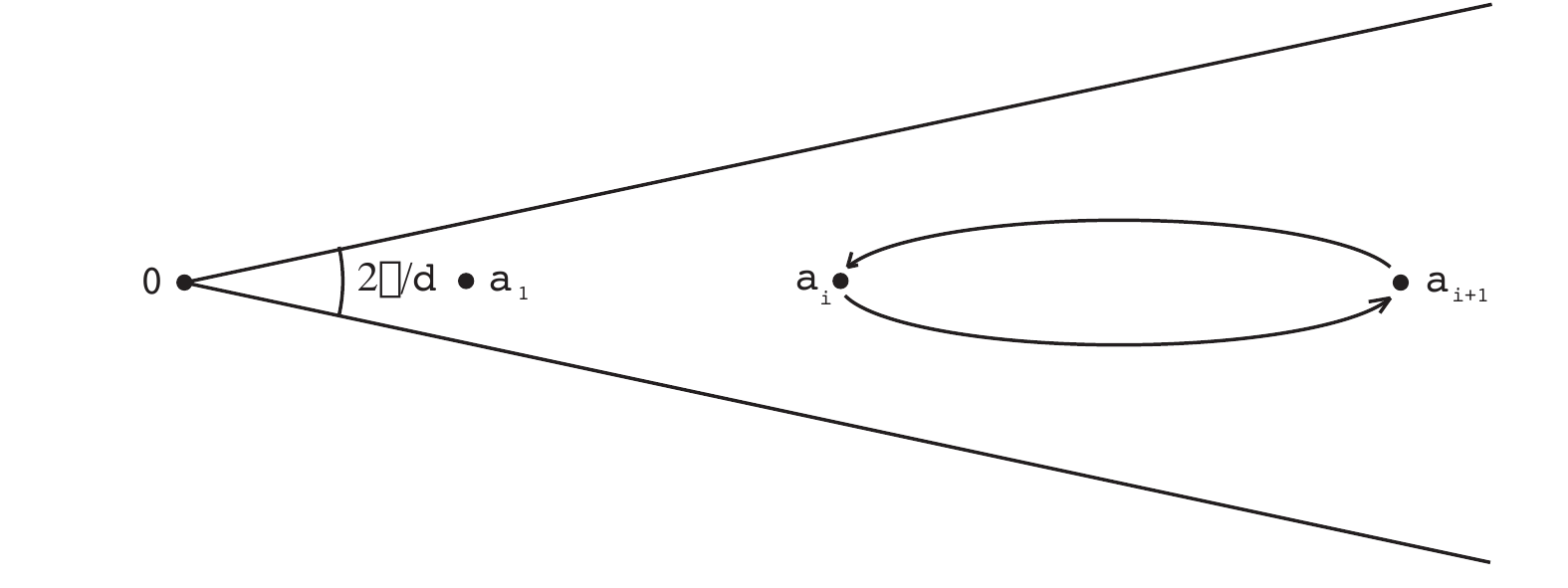}}
\end{center}
On the open cone described in the picture, the map $z \mapsto z^d$
is a positive homeomorphism, and this enables one to lift this path
to
\begin{center}
\resizebox{!}{2cm}{\includegraphics{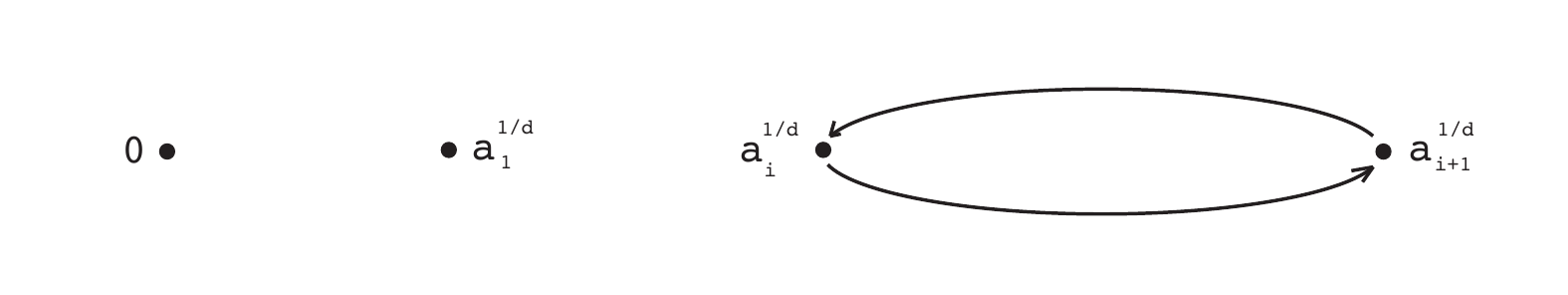}}
\end{center}
meaning that $\tau_i = \psi^{-1}(\xi_i)$ is the class of this path,
from $(a_1^{1/d}, \dots, a_i^{1/d}, a_{i+1}^{1/d}, \dots, a_r^{1/d})$
to $(a_1^{1/d}, \dots, a_{i+1}^{1/d}, a_{i}^{1/d}, \dots, a_r^{1/d})$.
We recall that $\B(de,e,r)$ is defined by $\pi_1( \mathcal{M}^{\#}(de,r)/G(de,e,r))$
when $d > 1$. If moreover $e=1$, then  $\B(d,1,r)$ is generated by $\sigma,
\tau_1,\dots,\tau_{r-1}$ ; in general, it is generated by $\sigma^e,
\tau_1,\dots,\tau_{r-1}$. Now, the morphism $\tilde{\psi} : \B(d,1,r) \into
\Br(r+1) = \pi_1(\mathcal{M}(r+1)/\mathfrak{S}_{r+1})$
commutes
with the natural morphisms
$$
\xymatrix{ \B(d,1,r) \ar[rr] \ar[d] & & \Br(r+1) \ar[d] \\
G(d,1,r) \ar[rr] \ar[dr] & & \mathfrak{S}_{r+1} \\ 
 & \mathfrak{S}_r \ar[ur] & } 
$$
Letting as in \cite{BMR} $\xi'_1 = \xi_0^2 \xi_1 \xi_0^{-2} \in \pi_1(
\mathcal{M}(r+1)/\mathfrak{S}_r, x)$, we have $\tau'_1 = \psi^{-1}(\xi'_1)
\in \pi_1(\mathcal{M}^{\#}(d,r) / G(d,1,r),y)$. As before we let $\zeta = \exp(2 \ii \pi/d)$
and $g_i \in G(d,1,r)$ being defined by $g_i.(z_1,\dots,z_r) = (z_1,
z_2,\dots,\zeta z_i, \dots, z_r)$. We let $b_i = a_i^{1/d}$.
$$
\xymatrix{
(b_1,\dots,b_r) \ar@/^1pc/[d]^{\xi_0^{-2}} \ar@/_5pc/@{-->}[dd]_{g_1^{-1} s_1}
\ar@/_10pc/@{-->}[ddd]_{g_1^{-1}s_1g_1}\ar@/^10pc/[ddd]_{\xi'_1}\\
(\zeta^{-1} b_1,b_2,\dots,b_r) \ar@/^1pc/[d]^{g_1^{-1}.\xi_1}\\
(\zeta^{-1} b_2, b_1,\dots,b_r) \ar@/^1pc/[d]^{g_1^{-1} s_1 . \xi_0^2}\\
(\zeta^{-1} b_2, \zeta b_1,\dots,b_r) }
$$
In order to generate $\B(e,e,r) = \pi_1(\mathcal{M}(e,r)/\mathfrak{S}_r, x)$,
and letting $e=d$,
we only need to take the image of $\tau'_1,\tau_1,\dots,\tau_{r-1}$
under $i^*$ where $i : \mathcal{M}^{\#}(e,r) \to \mathcal{M}(e,r)$
is the natural inclusion.
We will use the following definition.

\begin{defi}  Let $X$ be the complement of an hyperplane arrangement
$\mathcal{A}$ in $\C^l$, and $v,v' \in X$. A line segment from
$v$ to $v'$ is $t \mapsto (1-t) v + t v'$ for $t \in [0,1]$,
If this line segment crosses exactly one hyperplane of $\mathcal{A}$
at one point, a positive \emph{detour} from $v$ to $v'$ is a path of
the form $\gamma(t) = (1-t) v + t v' + \ii t (1-t) (v-v') \eps$
for $\eps > 0$ small enough so that it and the similar paths $\gamma'$
for $0 < \eps' < \eps$ do not cross any hyperplane in $\mathcal{A}$. All
such detours are clearly homotopic to each other. A negative detour
is defined similarly with $\ii$ replaced by $- \ii$.
\end{defi} 
Note that, for $v \in \mathcal{M}(e,r)$ and $s$ a reflection in $G(e,e,r)$,
if there exists a positive detour from the base point $\underline{b} = (b_1,\dots,b_r)$
 to $w.\underline{b}$, then it provides
a braided reflection around the hyperplane attached to $s$.

The elements $i^*(\tau_k)$ are now easy-to-describe
braided reflections, as the positive detours from $\underline{b}$
to their images by the corresponding reflections. In case $e = 2$,
the given monoid is then clearly the classical Artin monoid of type
$D_r$, so we can assume $e \geq 3$.
The paths corresponding to $\xi_0^{2}$
and to its translates are homotopic to a line segment in
$\mathcal{M}(e,r)$. The fact that $\tau'_1$ is a braided reflection
essentially amounts to the fact that
$i^*(g_0^{-1} .\tau_1)$ is a braided reflection in
$\pi_1(\mathcal{M}(e,r)/G(e,e,r), g_0^{-1} . y)$,
and this holds true because $\tau_1$ is a braided reflection in
$\mathcal{M}^{\#}(e,r)/G(e,e,r)$.

We consider the plane $P$ defined by the equations
$z_i = b_i$ for $i = 3,\dots,r$, and identify
it to $\C^2$ through $(z_1,z_2)$. We let
$P^0 = \C^2 \setminus \bigcup \{ z_2 = z_1 \eta \ | \ \eta \in \mu_e \}
= P \cap \mathcal{M}(e,r)$. 
Then $\tau_1,\tau'_1$ lie in the plane $P$, and
$\tau'_1$ is homotopic in $P^0$ to
\begin{center}
\resizebox{!}{3cm}{\includegraphics{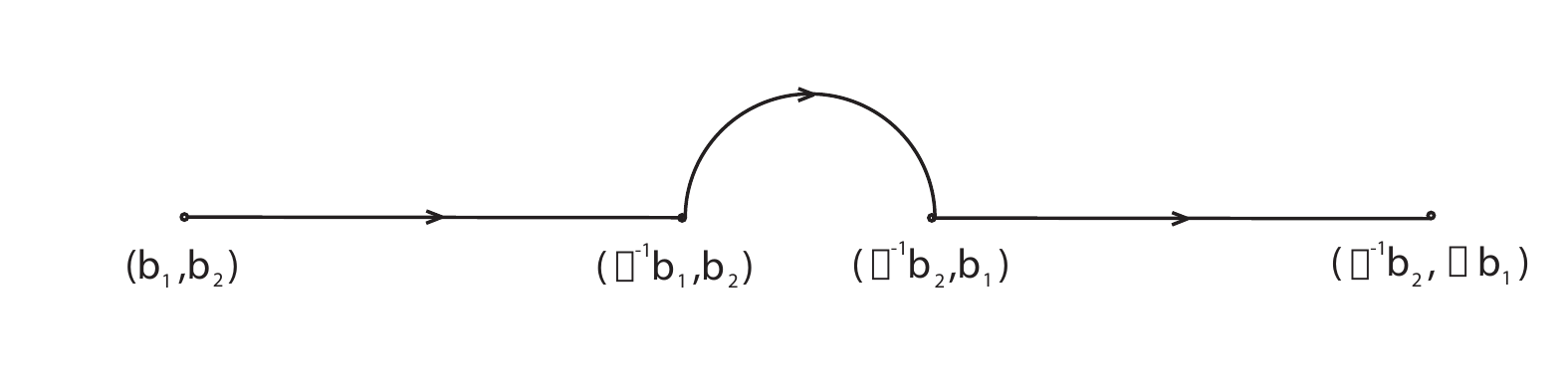}}
\end{center}
where the half-circle represents the positive detour from
$(\zeta^{-1} b_1,b_2)$ to $(\zeta^{-1} b_2,b_1)$. We let
now $t_0 = \tau_1$, $t_1 = \tau'_1$, 
$t_{i+1} = t_i^{-1} t_{i-1}t_i$ 
for $1 \leq i \leq e-2$.
A way to understand paths in $P \simeq \C^2$ is to use the
projection $\C^2 \to \mathbbm{P}^1(\C)$ given by $(z_1,z_2) \mapsto
z_2/z_1$. Note for example that two paths $\gamma_1,\gamma_2$
in $P$ with the same endpoints whose images are homotopic in $\mathbbm{P}^1(\C) \setminus
\mu_e$ are homotopic in $P^0$ as soon as, writing $\gamma_i(t)
=(x_i(t),y_i(t))$, the set
$x_1([0,1]) \cup x_2([0,1])$ is contained in some simply
connected subspace of $\C \setminus \{ 0 \}$.
We let $\alpha = b_2/b_1 \gg 1$. Then the positive detour $t_0$ is mapped
to a path from $\alpha$ to $\alpha^{-1}$ close to the line segment,
with image in the positive half-plane. The line segments of the
form $\gamma(t) = (z_1,z_2(t)$ are mapped to line segments,
and lines form $\gamma(t) = (z_1(t),z_2)$
are mapped to images of a line under $z \mapsto 1/z$,
which is the composite of the complex conjugation with
the geometric inversion  with respect to the unit circle ; they are thus mapped to a line
if the original line passes through 0, and otherwise
to a circle passing through the origin. The induced action of $G(e,e,r)$
is given by $s_1 : z \mapsto \frac{1}{z}$, $g_1 : z \mapsto \zeta^{-1} z$,
$g_2 : z \mapsto \zeta z$. The images of $t_1$ and $t_2$ are depicted
in figure \ref{figt1t2}. 
The images of $t_2$ and of the positive detour from $(b_1,b_2)$ to 
$(\zeta^{-2} b_2,\zeta^2 b_1)$ are then clearly homotopic (see figure \ref{figt2detour}),
and the first coordinate of both paths is easily checked to remain in
a simply connected region of $\C \setminus \{ 0 \}$. With
the same argument, using the relation $t_{i+1} = t_i^{-1} t_0 t_1$
and possibly using $(z_1,z_2) \mapsto z_1/z_2$ instead of $(z_1,z_2)
\mapsto z_2/z_1$, we get that each $t_i$ is
(homotopic to) the positive detour from $(b_1,b_2)$ to
$(\zeta^{-i}b_2,\zeta^i b_1)$. We thus got the following

\begin{prop} Let $\underline{b} = (b_1,\dots,b_r) \in \mathcal{M}(e,r)$
with $0 < b_1 \ll b_2 \ll \dots \ll b_r$. Then
$\B(e,e,r) = \pi_1( \mathcal{M}(e,r)/G(e,e,r),\underline{b})$
is generated by braided reflections $t_0,\dots,t_{e-1},s_3,\dots,s_r$
which are positive detours from $\underline{b}$ to their
images under the corresponding reflection. Under $\B(e,e,r) \onto
G(e,e,r)$, $t_i$ is mapped to $(z_1,z_2,\dots,z_r) \mapsto
(\zeta^{-i} z_2,\zeta^i z_1,\dots,z_r)$, and $t_0,s_3,\dots,
s_r$ are mapped to the successive transpositions of $\mathfrak{S}_r$
in that order. These generators provide a presentation
of $\B(e,e,r)$ with the relations (2)-(5) of page \pageref{rels}, and
with (1) replaced by $t_i t_{i+1} = t_j t_{j+1}$.
\end{prop}

We notice that the slight change in the presentation is
meaningless in monoid-theoretic terms, as both monoids
are isomorphic under $t_i \mapsto t_{- i}$, but it is not
in topological terms, as $t_1 t_0 t_1^{-1}$ is \emph{not}
homotopic to a detour from $\underline{b}$ to its image (see figure
\ref{figt2plus}).

\begin{figure}
\begin{center}
\resizebox{!}{5cm}{\includegraphics{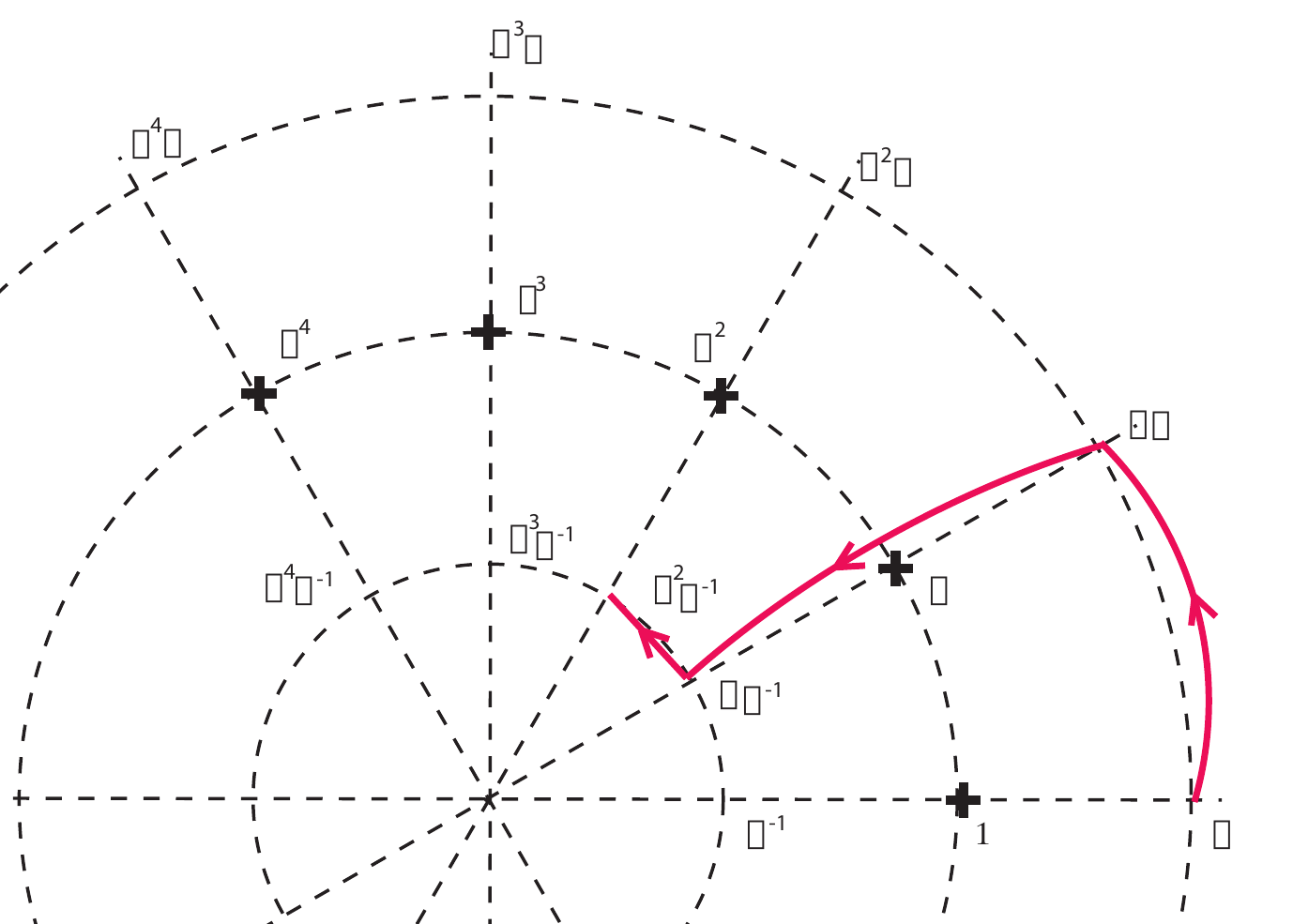}}
\resizebox{!}{5cm}{\includegraphics{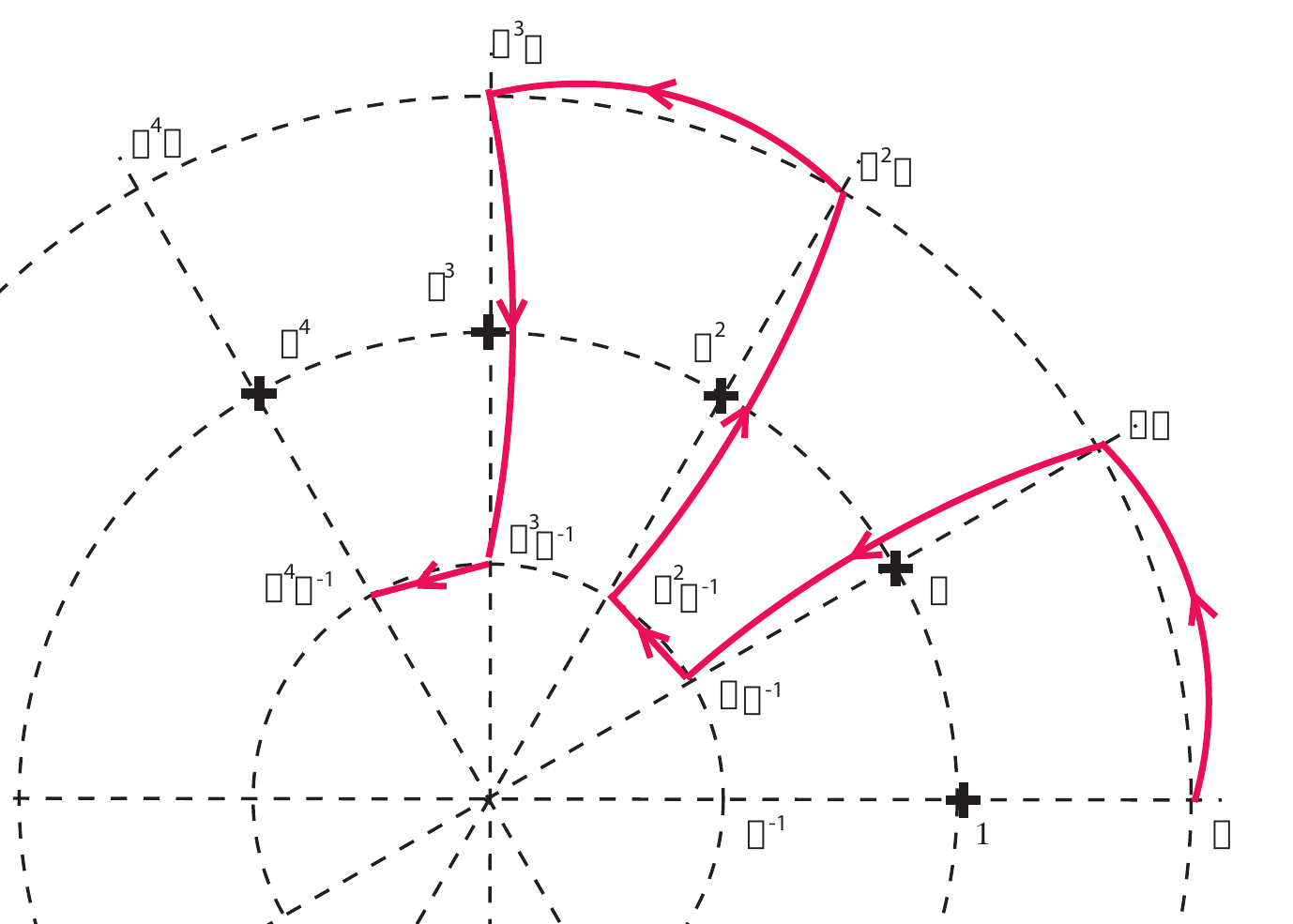}}
\caption{Images of $t_1$ and $t_2$ in $\mathbbm{P}_1(\C)$}
\label{figt1t2}
\end{center}
\end{figure}

\begin{figure}
\begin{center}
\resizebox{!}{8cm}{\includegraphics{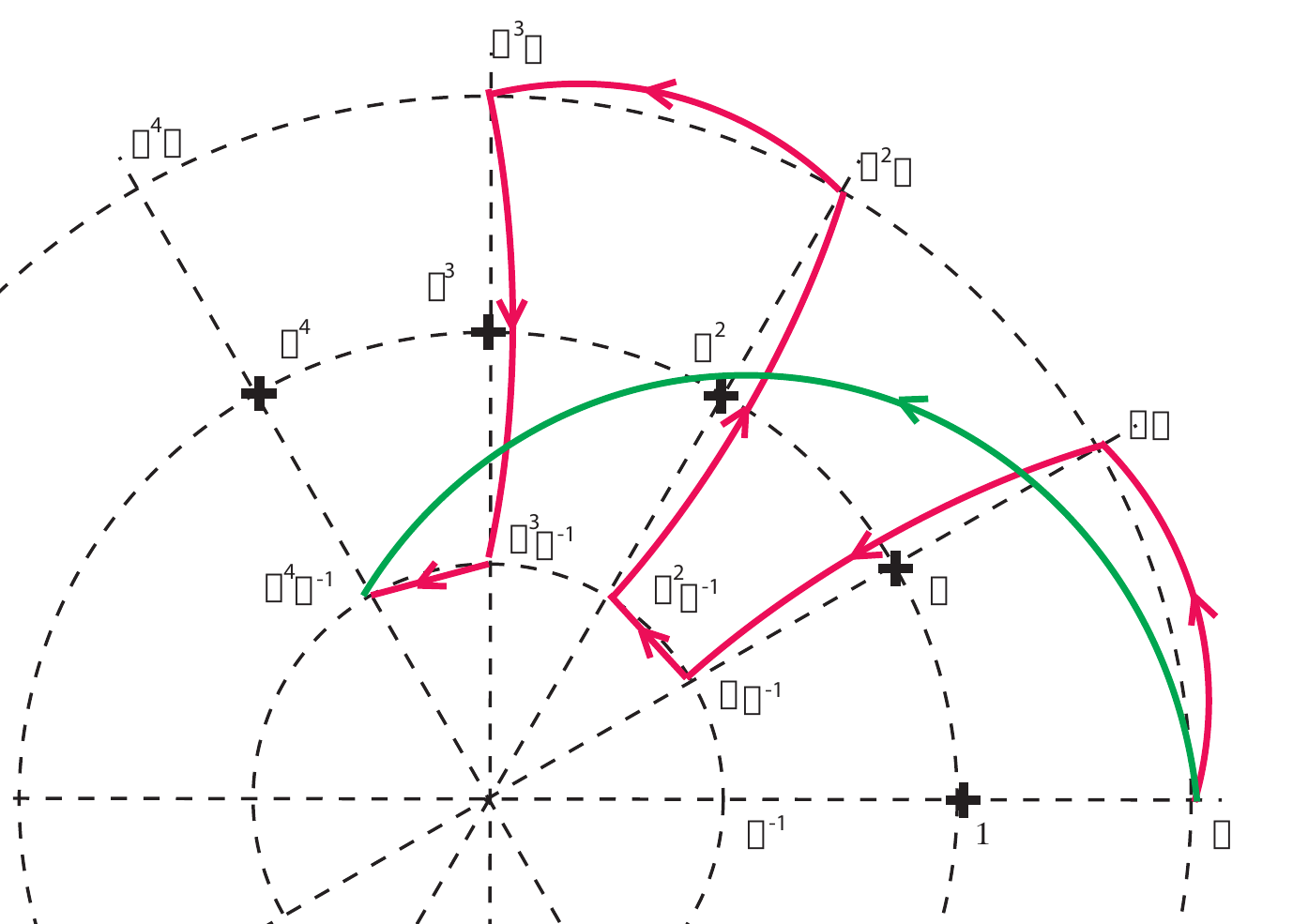}}
\caption{$t_2$ and the positive detour in $\mathbbm{P}_1(\C)$}
\label{figt2detour}
\end{center}
\end{figure}

\begin{figure}
\begin{center}
\resizebox{!}{5cm}{\includegraphics{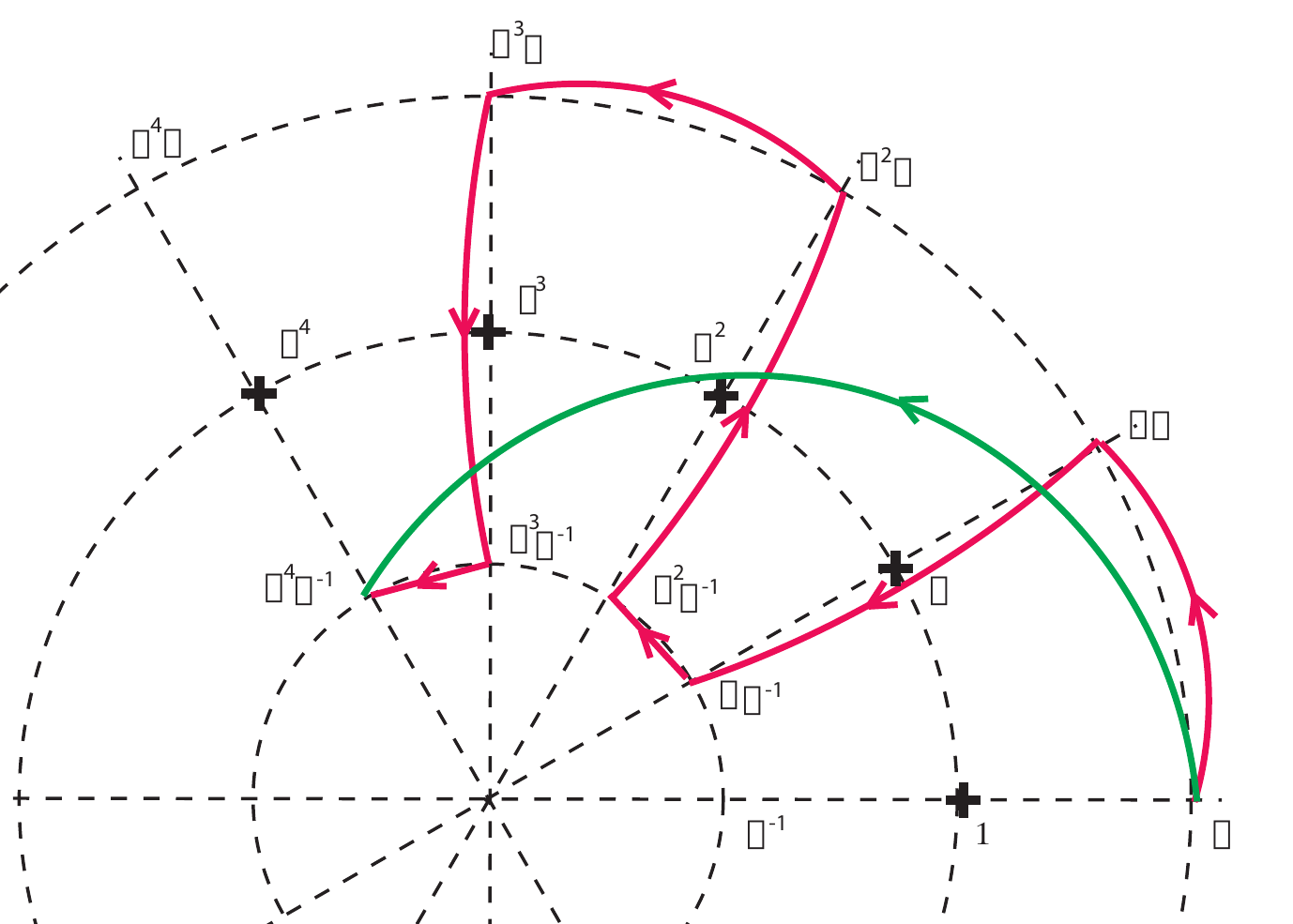}}
\end{center}
\caption{Comparison between $t_1 t_0 t_1^{-1}$ and the positive detour, in $\mathbbm{P}_1(\C)$}
\label{figt2plus}
\end{figure}

\begin{prop} 
Let $S_0 = \{ t_i,s_3,\dots,s_r \}$. The subgroup
of $\B(e,e,r) = \pi_1(\mathcal{M}(e,r)/G(e,e,r)$ generated by $S_0$
is a parabolic subgroup in the sense of \cite{BMR}, and can be
naturally identified with the braid group on $r$ strands
as the fundamental group of $\{ (z_1,\dots,z_r) \ | \ z_i \neq z_j, z_1+\dots+z_r = 0 \}/\mathfrak{S}_r$,
with base point $(-\zeta^{-i}(b_1+b_2+\dots + b_r),b_2,\dots,b_r)$,
in such a way that the elements of $S_0$ are identified with positive detours.
\end{prop}
\begin{proof}
The parabolic subgroup of $G(e,e,r)$ defined as the fixer
of $(\zeta^i, 1,1 , \dots, 1)$ is obviously conjugated to the one
fixing $(1,1,\dots,1)$, the latter being the natural
$\mathfrak{S}_r  \subset G(e,e,r)$. We thus need only consider
the case $i = 0$. Let $\alpha = -(b_1+\dots+b_r) \ll 0$, $\underline{b}_0
= (\alpha,b_2,\dots,b_r)$, $X = \mathcal{M}(e,r)$
and $X_0 = \{  (z_1,\dots,z_r) \ | \ z_i \neq z_j, z_1+\dots+z_r = 0 \}$.
By \cite{BMR} we get an embedding $\pi_1(X_0/\mathfrak{S}_r, 
\underline{b}_0) \into \pi_1(X/G(e,e,r),\underline{b})$, natural only
up to the choice of a path from $\underline{b}$ to $\underline{b}_0$ in
$Y$. The line segment $\gamma$
from $\underline{b}$ to $\underline{b}_0$ provides such a natural
choice.

We now need to prove that composing the positive detours
from $\underline{b}$ with this path provides the positive detours
from $\underline{b}_0$, up to homotopy in
$Y = \{  (z_1,\dots,z_r) \ | \ z_i \neq z_j \}$.
For $s_3,\dots,s_r$ this is true because
the first component of the first path can be homotoped to the second
one in $(\R_{\leq b_1}, b_1)$. For $t_0$
we let $\sigma_0$ and $\sigma$ denote the positive detours in $P^0$
from $(b_1,b_2)$ to $(b_2,b_1)$ and from $(\alpha,b_2)$ to
$(b_2,\alpha)$, respectively. Let $\gamma, \gamma'$
denote the line segments $(b_1,b_2) \to (\alpha,b_2)$
and $(b_2,\alpha) \to (b_2,b_1)$. We need to prove that
$\sigma_0$ is homotopic to $\gamma' \sigma \gamma$ in $\{ (z_1,z_2) \in \C^2 \ | \ 
z_1 \neq z_2 \}$,
the other coordinates $z_3,\dots,z_r$ being the same for both paths.
Since $b_1 - b_2$ and $\alpha -b_2$ have the same (negative) sign,
we can homotope $\gamma' \sigma \gamma$
to a path with the same real part (for both coordinates),
and with imaginary part the same as $\sigma$, up to possibly
diminishing the chosen factor $\eps$ in the definition of the
detours. Choosing then an homotopy in 
$\R_{\leq b_2}^2$ between the real parts of these two paths,
provides an homotopy between them in $Y$.
\end{proof}

\subsection{Parabolic submonoids}

We apply the results of section \ref{sectembed} on submonoids
to the monoid $N = M(e,e,r)$,
with generators $S = \{ t_0,\dots,t_{e-1}, s_3,\dots,s_r \}$.
Let $\mathcal{C} = \{ t_0 , \dots ,t_{e-1} \}$. For $S_0 \subset
S$, let $M(S_0)$ be the monoid generated by $S_0$ with the
defining relations of $M(e,e,r)$ which involve only elements of $S_0$.
We get a natural morphism $\varphi : M(S_0) \to M(S) = M(e,e,r)$. We ask
for the following extra assumption on $S_0$ :
$$
S_0 \cap \mathcal{C} \in \{ \emptyset, \mathcal{C}, \{ t_i \} \} \mbox{ for some } i \in \{0,1,\dots,e-1 \} .
$$
In other terms, $S_0$ contains none, all or exactly one of the $t_i$'s. Note that all
the corresponding monoids are known to be Garside and are endowed with a suitable length function.

This condition implies the extra condition on $\varphi$ in section \ref{sectembed},
namely that, if $n \in M(S)$ divides $\varphi(m)$ for some $m \in M(S_0)$,
then $n \in \varphi(M(S_0))$.
Indeed, if we have such $n,m$,
then $n \in \varphi(M(S_0))$ unless $n$ can be written as a word containing
some $x \in S \setminus S_0$.  But in that case $\varphi(m)$
can also be written as a word in $S$ containing $x$. Now note that the defining
relations involving such a $x$ cannot make it disappear, except when $x \in \mathcal{C}$.
By contradiction this settles the cases $S_0 \cap \mathcal{C} = \emptyset$ and $S_0
\supset \mathcal{C}$. In case $S_0 \cap \mathcal{C} = \{ t_{i} \}$, we can assume
$x = t_j$ for $j \neq i$, and would get
equality in $M(S)$ of two words on $S$, one involving $t_i$ and no other element
of $\mathcal{C}$, and the other involving $t_j$. But we check on the defining
relations that all relations involving $t_i$ either involve only $t_i$ and
no other elements of $\mathcal{C}$ in which case they preserve that property
and do not make the $t_i$'s disappear, or they involve several elements of $\mathcal{C}$
and cannot be applied to the first word. This leads to a contradiction,
which proves this property.

This condition also implies the property (2) for $\varphi$. For this we need to compute
the lcm's between two elements $x,y$ of $S$. We need to prove the following in $M(S_0)$,
for any $S_0 \subset S$ satisfying the above condition that contains $x$ and $y$.
\begin{itemize}
\item $\lcm(s_i,s_j) = s_i s_j = s_j s_i$ if $|j-i| \geq 2$
\item $\lcm(s_i,s_{i+1}) = s_i s_{i+1} s_i = s_{i+1} s_i s_{i+1}$
\item $\lcm(t_i,t_j) = t_1 t_0 = t_i t_{i-1} = t_jt_{j-1}$
\item $\lcm(t_i ,s_3) = t_i s_3 t_i = s_3 t_i s_3 $
\item $\lcm(t_i, s_j) = t_is_j = s_j t_i $ if $j \geq 4$.
\end{itemize}
The identities with length two are clear, as the lcm exist and cannot have length 1.
For the ones of length 3, namely $\{ x, y \} = \{ s_i,s_{i+1} \}$ and $\{ x, y \}  =  \{ t_i ,s_3 \}$, we
use that $\{ x , y \} \subset S_0 \subset S$ satisfies our condition. Since the lcm of $x,y$
in $M(S_0)$ should divide $xyx=yxy$, it should then come from $M(\{ x, y \})$, meaning that
is should be a word in $x$ and $y$, of length at most $2$. Thus only few possibilities
remain, all of them easily excluded.

Using the previous section, we thus get injective monoid morphisms $M(S_0) \to M(S) = M(e,e,r)$.
Let $B(S_0)$ the group of fractions of $M(S_0)$. It is proved in \cite{corpic} that
$B(S) = \B(e,e,r)$. We call the $B(S_0)$ the parabolic submonoids of $\B(e,e,r)$.
Crucial examples of such submonoids are described below.


\subsubsection{Second homology group}

We choose on the atoms the ordering $s_r < s_{r-1} < \dots< s_3 < t_0 < t_1 < \dots < t_{e-1}$.
By the above construction, the parabolic submonoid $M(e,e,r-1) = M(\{ s_{r-1},\dots,s_3,t_0,\dots,t_{e-1} \})$ is indeed a submonoid of
$M(e,e,r)$, 
and the lcm of a family of elements in $M(e,e,r-1)$ is also its lcm in $M(e,e,r)$. The same holds true
for the following submonoids :
\begin{itemize}
\item the ones generated by $s_3,t_i$, which is an Artin monoid of type $A_2$ ;
\item the ones generated by $s_k,t_i$, $k \geq 4$, which is an Artin monoid of type $A_1 \times A_1$ ;
\item the ones generated by $s_4,s_3,t_i$, which is an Artin monoid of type $A_3$ ;
\item the ones generated by $s_k,s_3,t_i$, $k \geq 5$ which is an Artin monoid of type $A_1 \times A_2$ ;
\item the ones generated by $s_k,s_l,t_i$, $k \geq l+2$, $l \geq 4$, which is an Artin monoid of type $A_1 \times A_1 \times A_1$ ;
\item the ones generated by $s_k,s_l,s_r$, which is an Artin monoid of type given by the obvious subdiagram
(of type $A_1 \times A_1 \times A_1$, $A_2 \times A_1$, $A_1 \times A_2$ or $A_3$).
\end{itemize}

We first compute the differentials of the top cell for the corresponding Artin monoids
(see Table \ref{typesA}), and then use this remark for computing the differentials of the
2-cells and 3-cells. We let $d_n = \partial_n \otimes_{\Z M} \Z :  C_n \otimes_{\Z M} \Z \to C_{n-1} \otimes_{\Z M} \Z$
denote the differential with trivial coefficients.

The 2-cells are the following : $[t_0,t_{i}]$ for $1<i < e$, $[s_3,t_i]$ , $[s_k,t_i]$ for $k \geq 4$
and $[s_k,s_l]$ for $k<l$.
From Table \ref{typesA} we get
$$
\begin{array}{lcll}
d_2 [t_0,t_i] &=& [t_i] + [t_{i+1}] - [t_0] - [t_1] \\
d_2 [s_3, t_i] &=& [t_i] - [s] \\
d_2 [s_k, t_i] &=& 0 \mbox{ if } k \geq 4 \\
d_2 [s_k,s_l] &=& 0 \mbox{ if } l> k+2 \geq 4 \\
d_2 [s_k,s_{k+1}] &=& [s_{k+1}] - [s_k] 
\end{array}
$$
We let $\delta_2 = t_1 t_0$ denote the Garside element
of $M(e,e,2)$ and we assume $e > 1$.
For the 3-cells, we also need to compute

$$
\begin{array}{lcl}
\partial_3 [s,t_0,t_j] &=& (s\delta_2 s - t_{j+2} t_{j+1} s +t_{j+2} s)[t_0,t_j]
-t_{j+2}st_{j+1}[s,t_j] \\ & & +(t_{j+2}-st_{j+2})[s,t_{j+1}] + (s-t_{j+2}s-1)[t_0,t_{j+1}]
+(st_2-t_2)[s,t_1] \\ & & + (t_2 s +1 -s)[t_0,t_1]
  +[s,t_{j+2}] +t_2 s t_1[s,t_0]
-[s,t_2]
\end{array}
$$
when $j \not\equiv -1 \mod e$, and 
$$
\begin{array}{lcl}
\partial_3 [s,t_0,t_{-1}] &=& (s\delta_2 s -t_1t_0s +t_1s)[t_0,t_{-1}]
 -t_{1}st_{0}[s,t_{-1}]  + (1-t_2+st_2)[s,t_1] \\ & & + (1+t_2 s - s)[t_0,t_1]
+(t_1 - st_1)[s,t_0]
 +t_2 s t_1[s,t_0]
-[s,t_2]
\end{array}
$$
This means $d_3 [s,t_0,t_{-1}] = [t_0,t_{-1}]
 -[s,t_{-1}]  + [s,t_1] + [t_0,t_1]
 +[s,t_0]
-[s,t_2]
$ and $d_3 [s,t_0,t_j] = [t_0,t_j]
-[s,t_j] -[t_0,t_{j+1}]
 + [t_0,t_1]
  +[s,t_{j+2}] +[s,t_0]
-[s,t_2]
$ for $j \not\equiv 1 \mod e$.


\begin{table}
$$
\begin{array}{|c|c|c|lcl|}
\hline
\mathrm{Type} & \mathrm{Atoms} & \mathrm{Relations} & & & \mathrm{Differential\ of\ top\ cell} \\
\hline
M(e,e,2) & t_0 < \dots < t_r & t_i t_{i+1} = t_j t_{j+1} & \partial_2 [t_0,t_i] & = & t_{i+1} [t_i] + [t_{i+1}] - t_1[t_0] - [t_1] \\
\hline
M(e,e,2) \times A_1 & s < t_0 < \dots < t_r & t_i t_{i+1} = t_j t_{j+1} & \partial_2 [s,t_0,t_i,] & = &(s-1)[t_0,t_i] - t_{i+1} [s,t_i]  \\
 & & t_i s = st_i &  & +  &  t_1 [s,t_0] - [s,t_{i+1}] + [s,t_1] \\
\hline
A_2 & s < t & sts = tst & \partial_2 [s,t] & =& (ts+1-s)[t] + (t-st-1)[s] \\
\hline
A_1 \times A_1 & s<u & su=us & \partial_2 [s,u] &=& (s-1)[u] - (u-1)[s] \\
\hline
A_3 & s<t<u & sts = tst & \partial_3 [s,t,u] &=& (u+stu-tu-1)[s,t] - [s,u] \\
 &  & su=us & & + & (su-u-s+1-tsu)t[s,u] \\
 &  & tut = utu & & + & (s-1-ts+uts)[t,u] \\
\hline
A_2 \times A_1 & s<t<u & tu = ut & \partial_3 [s,t,u] &= & (1-s+ts)[t,u] \\
 &  & su = us & & + & (t-1-st)[s,u] \\
 &  & sts = tst & & + &  (u-1)[s,t]\\
\hline
A_1 \times A_2 & s < t < u & st=ts & \partial_3 [s,t,u] &=& (1+tu-u)[s,t]\\ 
 & & su = us & & + &(t-1-ut)[s,u] \\
 & & tut=utu & & + & (s-1)[t,u] \\
\hline
A_1 \times A_1 \times A_1 & s<t<u & su = us & \partial_3[s,t,u] &=&  (1-t)[s,u] \\
& & st = ts & & + & (u-1) [s,t]\\
& & tu = ut & & + & (s-1)[t,u] \\
\hline
\end{array}
$$
\caption{Top cells for monoids of small type}\label{typesA}
\end{table}

We now compute the second homology group, starting with
$\Ker d_2$. Let
$$
v_i = [t_0,t_i] + [s,t_0] + [s,t_1] - [s,t_i] - [s,t_{i+1}] \in \Ker d_2
$$
for $1 \leq i \leq e-1$. Let $K_1$ denote the submodule of $\Ker d_2$
spanned by the $v_i$. It is easy to show that $K_1$ is free
on the $v_i$, and $K_1 = \Ker d_2$ for $r = 3$ ; if $r > 3$
we have $\Ker d_2 = K_1 \oplus K_2$ where $K_2$ is
the free $\Z$-module with basis the $[s_k,t_i]$ for
$k \geq 4$ and, if $r \geq 5$, the $[s_l,s_k]$ for $l \geq k+2$, $k \geq 3$.

Now decompose $\Z \mathcal{X}_3 = C_1 \oplus C_2$
where $C_1$ has for basis the $[s_3,t_0,t_i]$ and
$C_2$ has for basis the other 3-cells. By the above computations
we get $d_3(C_1) \subset K_1$,
and  $d_3(C_2) \subset K_2$. Thus
$H_2(B,\Z) = (K_1/d_3(C_1)) \oplus (K_2 / d_3(C_2))$.
We first compute $K_1/d_3(C_1)$. We have
$$
\begin{array}{lclr}
d_3[s_3,t_0,t_j] & = & v_j - v_{j+1} + v_1 & \mbox{ if } 0<j<e-1 \\
d_3[s_3,t_0,t_{e-1}] & = & v_{e-1} + v_{1} + v_1 
\end{array}
$$
We denote $u_{i} = [s_3,t_0,t_i]$  for $1 \leq i \leq e-1$,
and let $w_i = u_i + u_{i+1} + \dots + u_{e-1}$. Then
$d_3 w_i = v_i + (e-i) v_1$. Written on the $\Z$-basis $(w_i)$ and $(v_i)$,
$d_3$ is in triangular form, and the only diagonal coefficient
that differs from $1$ is 
$e$, since $d_3 w_1 = e v_1$. It follows that $K_1/d_3(C_1) \simeq \Z_e$.
Since $H_2(B,\Z) = K_1/d_3(C_1)$ for $r = 3$, we can now assume $r \geq 4$.

First assume $r=4$. In $K_2/d_3(C_2)$ we have $2[s_4,t_i] \equiv 0$,
because $d_3[s_4,s_3,t_i] = -2[s_4,t_i]$. Since
$d_3[s_4,t_0,t_i] = -[s_4,t_i]+[s_4,t_0]-[s_4,t_{i+1}]+[s_4,t_1]$.
we get $[s_4,t_i] + [s_4,t_{i+1}] \equiv [s_4,t_0]+[s_4,t_1]$ when $i>0$.
In particular, $[s_4,t_i] + [s_4,t_{i+1}] \equiv [s_4,t_{i+1}] + [s_4,t_{i+2}]$
that is $[s_4,t_i] \equiv [s_4,t_{i+2}]$, at least if $0<i<e-1$. From $d_3[s_4,t_0,t_1]\equiv 0$
we deduce $[s_4,t_2] \equiv [s_4,t_0]$, and from $d_3[s_4,t_0,t_{e-1}] \equiv 0$
we deduce $[s_4,t_{e-1}] \equiv [s_4,t_1]$. Thus $[s_4,t_i] \equiv [s_4,t_{i+2}]$
for every $i$. When $e$ is odd, $K_2/d_3(C_3)$ is then spanned by the
class of $[s_4,t_0]$. From the other relations one easily gets
that this class is nonzero, and since $2[s_4,t_i] \equiv 0$ we get
$K_2/d_3(C_2) \simeq \Z_2$. When $e$ is even, this quotient is spanned by the classes
of  $[s_4,t_0]$ and $[s_4,t_1]$, and we get similarly $K_2/d_3(C_2) \simeq \Z_2^2$.

We now assume $r \geq 5$. Then $d_3([s_5,s_3,t_i]) = [s_5,s_3]-[s_5,t_i]$
whence $a := [s_5,t_0] \equiv [s_5,t_i]$ for all $i$, regardless
whether $e$ is even or odd.
From $d_3[s_5,s_4,t_i] = [s_4,t_i] - [s_5,t_i]$ we get $[s_4,t_i] \equiv a$
and from $[s_{k+1},s_k,t_i] = [s_k,t_i] - [s_{k+1},t_i]$ we deduce
by induction $[s_k,t_i] \equiv a$. The only remaining relation involving $a$
is then as before $2a \equiv 0$.

On the other hand, we have $[s_5,s_3] \equiv a$. Assume we have $[s_l,s_k] \equiv a$
for some $l,k$ with $l \geq k+2$. From $d_3[s_l,s_{k+1},s_k] = [s_l,s_{k+1}] - [s_l,s_k]$ for $l \geq k+3$
we get $[s_l,s_{k'}] \equiv a$ for all $l-2 \geq k' \geq 3$, and
then that $[s_{l'},s_{k'}] \equiv a$ for all $l'-2 \geq k' \geq 3$.
We thus get $K_2/d_3 C_2 \simeq \Z/2\Z$.

As a consequence, we get the following result.

\begin{theo} \label{H2BEER} Let $B = \B(e,e,r)$ with $r \geq 3$ and $e \geq 2$.
\begin{itemize}
\item When $r = 3$, $H_2(B,\Z) \simeq \Z_e$
\item When $r = 4$ and $e$ is odd, $H_2(B,\Z) \simeq \Z_e \times \Z_2 \simeq \Z_{2e}$
\item When $r = 4$ and $e$ is even, $H_2(B,\Z) \simeq \Z_e \times \Z_2^2$
\item When $r \geq 5$, $H_2(B,\Z) \simeq \Z_e \times \Z_2$
\end{itemize}
\end{theo}

The case $r=2$ is when $W$ is a dihedral group,
and this case is known by \cite{SALVETTI} : we have $H_2(B,\Z) = 0$
if $e$ is odd, $H_2(B,\Z) = \Z$ if $e$ is even.

\section{Low-dimensional homology} \label{s:low}
\subsection{The second homology group} \label{ss:H2}

The computations above provide the second integral homology
group $H_2(B,\Z)$. In the case of the finite group
$W$, the group $H_2(W,\Z)$ can be identified with
the Schur multiplier $H^2(W,\C^{\times})$, which is relevant
for dealing with projective representations. We use the determination
of the $H_2(B,\Z)$ to show a direct connection between the
two groups $H^2(B,\C^{\times})$ and $H^2(W,\C^{\times})$.
We first start with a lemma.

\begin{lemma} Let $W$ be an irreducible finite
complex 2-reflection group, and $B$ the associated braid group.
The inflation morphism $H^2(W,\C^{\times}) \to H^2(B,\C^{\times})$
is into.
\end{lemma}
\begin{proof}
The Hochschild-Serre exact sequence associated to $1 \to P \to B \to W \to 1$
is
$$
0 \to H^1(W,\C^{\times}) \to H^1(B,\C^{\times}) \to
H^1(P,\C^{\times})^W \to H^2(W,\C^{\times}) \to H^2(B,\C^{\times}).
$$
Now $H^1(P,\C^{\times}) = \Hom(P^{ab},\C^{\times})^W
= \Hom((P^{ab})^W ,\C^{\times})$ and $H^1(B,\C^{\times})
= \Hom(B^{ab},\C^{\times})$. Now $P^{ab} = H_1(P,\Z)$ and $B^{ab}$
(see \cite{BMR} thm. 2.17) are torsion-free, with $B^{ab}
\simeq \Z^r$ where $r$ denotes the number of hyperplane orbits,
and $(P^{ab})^W$ can be identified with $(2 \Z)^r$. The
induced map $\Hom(\Z^r,\C^{\times})\to \Hom((2 \Z)^r,\C^{\times}) $
is then onto, since $\C$ is algebraically closed. By the
Hochschild-Serre exact sequence above the conclusion follows.
\end{proof}

{\bf Remark.} Another proof of the lemma can be given
using projective representations instead of the Hochschild-Serre
exact sequence. Let $\alpha \in Z^2(W,\C^{\times})$
with zero image in $H^2(B,\C^{\times})$, choose some
projective representation $R$ of $W$ with 2-cocycle $\alpha$,
and consider its lift $\tilde{R}$ to $B$. By assumption,
it is linearizable into some linear representation $\tilde{S}$. Choosing one
generator of the monodromy $\sigma_i$ in $X/W$ for each hyperplane orbit
(see \cite{BMR} appendix A) we find that $\tilde{S}(\sigma_i^2) = \la_i \in
\C^{\times}$. By \cite{BMR} Theorem 2.17 there exists
a morphism $\varphi : B \to \C^{\times}$ with $\varphi(\sigma_i) = 1/\la_i$,
and then $\tilde{T} = \tilde{S} \circ \varphi$ is a linear representation
of $B$ that factors through $W$ and linearizes $R$, thus proving
that $\alpha$ has zero image in $H^2(W,\C^{\times})$.

\bigskip

It is known by work of Read \cite{READ} and van der Hout \cite{HOUT}
that $H^2(W,\C^{\times}) \simeq H_2(W,\Z)$
is a free $\Z_2$-module in all cases. A nice property
that follows from our computation is that
the part of $H^2(B,\C^{\times})$ that comes from $H^2(W,\C^{\times})$
is exactly the 2-torsion (except for 2 exceptional cases).
Indeed, since $H_1(B,\Z)$ is torsion-free
and $\C^{\times}$ is divisible, by the Universal Coefficients
Theorem we get $H^2(B,\C^{\times}) \simeq \Hom(H_2 B, \C^{\times})$
and the proposition below is a consequence of our computation
of $H_2(B,\Z)$ (see Table \ref{tableH2} for the exceptional groups,
Theorems \ref{H2B2EER} and \ref{H2BEER} for the $G(2e,e,r)$ and the $G(e,e,r)$)
and of the works of Read and van der Hout on $W$. We recall their computation
of $H_2(W,\Z)$ in Table \ref{tableH2} for the exceptional groups, and
the rank over $\Z_2$ for the other ones in Table \ref{schurgen}.
\begin{table}

$$
\begin{array}{|c|l|c|c|}
\hline
r & e& G(e,e,r) & G(2e,e,r) \\
\hline
2 & \mbox{odd} & 0 & 1 \\
 & \mbox{even} & 1 & 2 \\
\hline
3 & \mbox{odd} & 0  & 2 \\
 & \mbox{even} & 1 & 2 \\
\hline
4 & \mbox{odd} & 1 & 3 \\
 & \mbox{even} & 3 & 4 \\
\hline
5 & \mbox{odd} & 1 & 3 \\
 & \mbox{even} & 2 & 3 \\
\hline
\end{array}
$$
\caption{Rank of $H_2(W,\Z)$ as a $\Z_2$-module (after Reid)}  
\label{schurgen}
\end{table}

\begin{prop} Except for $W = G_{33}$ or $W = G_{34}$,
$H^2(W,\C^{\times})$ coincides with the 2-torsion of $H^2(B,\C^{\times})$.
\end{prop}

\begin{table}
$$\begin{array}{|c|c|c||c|c|c|}
\hline
W & H_2 W & H_2 B & W & H_2 W & H_2 B \\
\hline
G_{12} & 0 & 0 & G_{30} & \Z_2 & \Z_2 \\
G_{13} & \Z_2 & \Z & G_{31} & \Z_2 & \Z_6 \\
G_{22} & 0 & 0 & G_{33} & 0 & \Z_6 \\
G_{23} & \Z_2 & \Z & G_{34} & 0 & \Z_6 \\
G_{24} & \Z_2 & \Z & G_{35} & \Z_2 & \Z_2 \\
G_{27} & \Z_2 & \Z_3 \times \Z & G_{36} & \Z_2 & \Z_2 \\
G_{28} & (\Z_2)^2 & \Z^2 & G_{37} & \Z_2 & \Z_2 \\
G_{29} & (\Z_2)^2 & \Z_2 \times \Z_4 & & & \\
\hline
\end{array}
$$
\caption{The second integral homology groups}\label{tableH2}
\end{table}

\subsection{First homology in the sign representation}\label{ss:sign}

If $r = |\mathcal{A}/W|$ denotes the number of hyperplane
classes, the abelianization $B_{ab}$ is isomorphic to $\Z^r$.
There are thus $2^r-1$ nonzero morphisms $B \onto \Z_2$,
which define $2^r-1$ subgroups of even braids. When $r = 1$,
there is only one such morphism $\eps : B \to \Z_2$ and group $B^{(2)} = \Ker \eps$.
We investigate here two abelian invariants of $B$ which are
naturally attached to this group : the abelianization $B^{(2)}_{ab}$
of $B^{(2)}$ and $H_1(B,\Z_{\eps})$.

\begin{lemma} Let $u \in B \setminus B^{(2)}$.
The group $H_1(B,\Z_{\eps})$ is isomorphic to the quotient of $B^{(2)}_{ab}$
by the relations $[u^2] \equiv 0$ and $[h^u] \equiv
-[h]$ for $h \in B^{(2)}_{ab}$, where $h^u = u^{-1} h u$.
\end{lemma}

\begin{proof}
We start from the bar resolution $C_2 \to C_1 \to C_0$,
where $C_i$ is a free $\Z B$-module with basis the $[g_1,\dots,g_i]$
for $g_i \in B$, we have $d_1([g]) = (g-1)[\emptyset]$,
$d_2([g_1,g_2]) = g_1 [g_2] - [g_1 g_2] + [g_1]$. Denoting
$d_i^{\eps}$ the differential with coefficients in $\Z_{\eps}$ and
$C_i^{\eps} = C_i \otimes_{\Z B} \Z_{\eps}$ with $\Z$-basis
the $[g_1,\dots,g_i]$, we get that $\Ker d_1^{\eps}$ is the
direct sum $\Z B^{(2)} \oplus I$ where
$I = \{ \sum_{g \not\in B^{(2)}} x_g [g] \ | \ x_g \in \Z, \sum x_g = 0 \}$.
Choose some $u \in B \setminus B^{(2)}$.
The image of $d_2^{\eps}$ is spanned by the $[g_1g_2] - \eps(g_1)
[g_2] - [g_1]$. Among them we find
\begin{enumerate}
\item $[u^2] + [u] - [u] = [u^2]$
\item $[h_1 h_2] - [h_1] - [h_2]$, for $h_1,h_2$ in $B^{(2)}$
\item $[uh] - [u] + [h]$ for $h \in B^{(2)}$
\item $[h^u] + [h]$ for $h \in B^{(2)}$.
\end{enumerate}
Indeed, the element (4) is the difference of two elements clearly
in $\mathrm{Im} d_2^{\eps}$, $[hu] - [u] - [h]$
and $[u h^u] + [h^u] - [u]$, where $h^u = u^{-1} h u$,
since $u h^u = hu$. By (3), and since $I$ is spanned
by the $[hu] - [u]$ for $h \in B^{(2)}$, we see that
$H_1(B,\Z_{\eps})$ is generated by the images
of the $[h]$ for $h \in B^{(2)}$.
It is easy to check that the relations of the form
$d_2^{\eps}([g_1,g_2]) \equiv 0$ are consequences of (1-4),
hence $H_1(B,\Z_{\eps})$ is the quotient of $ B^{(2)}_{ab}$ by the
relations $(1)$ and $(4)$.
\end{proof}

The computation of $B^{(2)}$ can be done for exceptional groups by
using the Reidemeister-Schreier method (see \cite{MKS})
and the presentations of \cite{BMR} and \cite{bessismichel}. Note
that they are known to provide presentations of $B$ for all groups
but $G_{31}$, for which our results as well will be conjectural. 
We start from one of these standard presentation of $B$ by
braided reflections $\sigma_1,\dots,\sigma_{n}$
and use $\{ 1, \sigma_1\}$ for Schreier transversal. Then
generators for $B^{(2)}$ are given by $\sigma_1^2,\sigma_1 \sigma_2,
\sigma_1 \sigma_3, \dots,\sigma_1 \sigma_n$ and $\sigma_2 \sigma_1^{-1}$,
$\sigma_3 \sigma_1^{-1},\dots ,\sigma_n \sigma_1^{-1}$.
We then apply the Reidemeister-Schreier process and find
a presentation of $B^{(2)}$ from the relations $R$, $\sigma_1 R \sigma_1^{-1}$
where $R$ runs among the relations for $B$. The presentations obtained
for exceptional groups are tabulated in figure \ref{figpreseven} (the column `ST' refers
to the Shephard-Todd number of the group). It is then easy
to abelianize these relations. We choose $u = \sigma_1$.

In order to get $H_1(B,\Z_{\eps})$ from $B^{(2)}_{ab}$
we start by adding the relation $[\sigma_1^2] \equiv 0$.
Note that $\sigma_1 (\sigma_i \sigma_1^{-1})  \sigma_1^{-1}
= (\sigma_1 \sigma_i) (\sigma_1^{-2})$ hence $-[\sigma_i \sigma_1^{-1}]
 \equiv [\sigma_1 \sigma_i] - [\sigma_1^2] \equiv [\sigma_1 \sigma_i]$,
and that $\sigma_1 (\sigma_1 \sigma_i)\sigma_1^{-1} = \sigma_1^2.
\sigma_i \sigma_1^{-1}$ hence $-[\sigma_1 \sigma_i] \equiv [\sigma_i \sigma_1^{-1}]$.
The relations defining $H_1(B,\Z_{\eps})$ from $B^{(2)}_{ab}$
thus boil down to $-[\sigma_1 \sigma_i] \equiv [\sigma_i \sigma_1^{-1}]$
and $[\sigma_1^2] \equiv 0$.

\def\ss{\sigma}

In order to get $H_1(B,\Z_{\eps})$ for the groups $G(*e,e,r)$,
instead of using the complicated presentations of $B$ afforded
by \cite{BMR}, we use the semidirect product decomposition
described in section \ref{s:iso}. Recall that $B = \Z \ltimes \tilde{A}$
where we denote by $A$ the affine Artin group of type $\tilde{A}_{r-1}$.
Then $A$ has Artin generators $\ss_1,\dots,\ss_r$ and the semidirect
product is defined by $\tau \ss_i \tau^{-1} = \ss_{i+e}$ where
addition is considered modulo $r$. From the split
exact sequence $1 \to A \to B \to \Z \to 1$ we get the
Hochschild-Serre short exact sequence
$$
0 = H_2(\Z,H_0(A,\Z_{\eps})) \to H_0(\Z,H_1(A,\Z_{\eps})) \to H_1(B,\Z_{\eps}) \to H_1(\Z,H_0(A,\Z_{\eps})) \to 0 
$$
with $H_2(\Z,H_0(A,\Z_{\eps})) = 0$ since $\Z$ has homological dimension 1. Since $A$ acts on $\Z_{\eps}$
through $\ss_i \mapsto -1$ we have $H_0(A,\Z_{\eps}) = \Z/2\Z = \Z_2$ ; since
$\tau$ acts trivially on $H_0(A,\Z_{\eps})$ we thus get
$H_1(\Z,H_0(A,\Z_{\eps})) \simeq H_1(\Z,\Z_2) \simeq \Z_2$. The short
exact sequence thus boils down to $0 \to H_0(\Z,H_1(A,\Z_{\eps})) \to
H_1(B,\Z_{\eps}) \to \Z_2 \to 0$ and our task is reduced to computing
$H_1(A,\Z_{\eps})$ while keeping track of the action of $\tau$.

In order to compute $H_1(A,\Z_{\eps})$ we apply the above process. Generators
for $A^{(2)}$ are given by $u = \ss_1^2$, $x_i = \ss_1 \ss_i$ and
$y_i = \ss_i \ss_1^{-1}$ for $2 \leq i \leq r$, and relations
are given by rewriting $R$ and $\ss_1 R \ss_1^{-1}$ with $R$ running
along the braid relations for $A$. These braid relations are the
following (where $|j-i| \geq 2$ actually mean that $j,i$ are not
connected in the braid diagram)
$$
\begin{array}{lllcl}
(R) & 1 \not\in \{ i, i+1\} & \ss_i \ss_{i+1} \ss_i \ss_{i+1}^{-1} \ss_i^{-1} \ss_{i+1}^{-1} & \rightsquigarrow & 
y_i  x_{i+1}  y_i  y_{i+1}^{-1}  x_i^{-1} y_{i+1}^{-1}
\\
& |j-i| \geq 2, 1 \not\in \{ i,j \}  & \ss_i \ss_j \ss_i^{-1} \ss_j^{-1} & \rightsquigarrow & 
y_i x_j x_i^{-1}y_j^{-1}\\
& & \ss_1 \ss_2 \ss_1 \ss_2^{-1} \ss_1^{-1} \ss_2^{-1} & \rightsquigarrow &
x_2 y_2^{-1}u^{-1} y_2^{-1} \\
& & \ss_1 \ss_r \ss_1 \ss_r^{-1} \ss_1^{-1} \ss_r^{-1} & \rightsquigarrow & 
x_r  y_r^{-1} u^{-1}y_r^{-1}\\
& i \not\in \{2,r \} & \ss_1 \ss_i \ss_1^{-1} \ss_i^{-1} & \rightsquigarrow &  x_i  u^{-1} y_i^{-1}\\
(\ss_1 R \ss_1^{-1} ) & 1 \not\in \{ i, i+1\} & \ss_1\ss_i \ss_{i+1} \ss_i \ss_{i+1}^{-1} \ss_i^{-1} \ss_{i+1}^{-1}\ss_1^{-1} & \rightsquigarrow & 
x_i y_{i+1}x_i  x_{i+1}^{-1} y_i^{-1} x_{i+1}^{-1} \\
& |j-i| \geq 2, 1 \not\in \{ i,j \}  & \ss_1\ss_i \ss_j \ss_i^{-1} \ss_j^{-1} \ss_1^{-1}& \rightsquigarrow & 
x_i y_j y_i^{-1}x_j^{-1} 
\\
& & \ss_1\ss_1 \ss_2 \ss_1 \ss_2^{-1} \ss_1^{-1} \ss_2^{-1}\ss_1^{-1} & \rightsquigarrow &
u y_2 u x_2^{-1} x_2^{-1}
 \\
& & \ss_1\ss_1 \ss_r \ss_1 \ss_r^{-1} \ss_1^{-1} \ss_r^{-1}\ss_1^{-1} & \rightsquigarrow & 
u y_ru  x_r^{-1} x_r^{-1}
\\
& i \not\in \{2,r \} &\ss_1 \ss_1 \ss_i \ss_1^{-1} \ss_i^{-1}\ss_1^{-1} & \rightsquigarrow & 
u  y_i   x_i^{-1}
\\ 
\end{array}
$$
Abelianizing and dividing out by the relations $y_i = -x_i$ yields an
abelian presentation for $H_1(A,\Z_{\eps})$ by generators $u, x_i$ for
$2 \leq i \leq r$ and relations
$$
\begin{array}{lll}
 1 \not\in \{ i, i+1\} &  3 x_{i+1} =3 x_i
\\
|j-i| \geq 2, 1 \not\in \{ i,j \}  &  
 2 x_j =2 x_i \\
 & 
3x_2 =0 \\
 & 
3x_r =0\\
 i \not\in \{2,r \} & 2 x_i = 0\\
\end{array}
$$
Thus, for $r = 3$, $H_1(A,\Z_{\eps}) = <x_2,x_3 | 3x_2 = 3x_3 = 0> = \Z_3 x_2 \oplus \Z_3 x_3 \simeq \Z_3^2$,
for $r = 4$,
$$
H_1(A,\Z_{\eps}) =  < x_2,x_3,x_4 | 3x_2 = 3x_4 = 0, 2x_3 = 0 , 2x_2 =2x_4, 3x_3 = 3x_2=3x_4 >
$$
hence $H_1(A,\Z_{\eps} ) =  < x_2, x_4 | 3x_2 = 3 x_4 = 0 , x_2 = x_4 >  = \Z_3 x_2 \simeq \Z_3$.
When $r \geq 5$ , $H_1(A,\Z_{\eps})$ is generated by $x_2,\dots,x_r$, and we have $3 x_2 = 3 x_r= 0$.
We have $2 < 3 < r-1 < r$. Then $2x_3 = 2x_{r-1} = 0$ but $0= 3x_2 = 3x_3$ and $0 = 3x_r = 3x_{r-1}$.
It follows that $x_3 = 0$ and $x_{r-1} = 0$. Since $2x_3 = 2x_r$ and $2x_2 = 2x_{r-1}$ we get
$x_2 = x_{r-1}$ and $x_3 = x_{r}$ hence $x_i = 0$ for all $i$ and $H_1(A,\Z_{\eps}) = 0$.

For $r \in \{ 3 , 4 \}$ it remains to compute the action of $\tau$ on  $H_1(A,\Z_{\eps})$.
We have $\tau.\ss_i = \ss_{i+e}$ hence $\tau.(\ss_1 \ss_i) = \ss_{1+e} \ss_{i+e}
= \ss_{1+e} \ss_1^{-1} \ss_1 \ss_{i+e}$. For $e \equiv  0 \mod r$ we have
$\tau.x_i = x_i$ and $H_0(\Z,H_1(A,\Z_{\eps})) \simeq H_1(A,\Z_{\eps})$.
For $r = 3$, $e \equiv 1 \mod 3$, $\tau.x_2 = \ss_2 \ss_3 
\equiv  y_2 + x_3 \equiv -x_2 + x_3  $ and $\tau.x_3 = \ss_2\ss_1 
= \ss_2 \ss_1^{-1} \ss_1^2 \equiv -x_2$. It follows that $H_0(\Z,H_1(A,\Z_{\eps})) \simeq \Z_3$.
For $r=4$, $e \equiv 1 \mod 4$, $\tau.x_2 = \ss_2 \ss_3 = \ss_2 \ss_1^{-1} \ss_1 \ss_3
\equiv y_2 + x_3 \equiv -x_2 + x_3 \equiv -x_2$ hence $H_0(\Z,H_1(A,\Z_{\eps}))=0$.
For $r=4$, $e \equiv 2 \mod 4$, $\tau.x_2 = \ss_3 \ss_4 = \ss_3 \ss_1^{-1} \ss_1 \ss_4 \equiv  x_2$
hence $H_0(\Z,H_1(A,\Z_{\eps}))= H_1(A,\Z_{\eps})$.
Altogether, this yields
\begin{prop} For $B = \B(*e,e,r)$, and $r \geq 3$,
$$
\begin{array}{ll}
H_1(B,\Z_{\eps}) \simeq \Z_2 & \mbox{ for } r \geq 5 \\
H_1(B,\Z_{\eps}) \simeq  \Z_6 & \mbox{ for } r=4, e \equiv 0,2 \mod 4 \\
H_1(B,\Z_{\eps}) \simeq  \Z_2 &  \mbox{ for } r=4, e \equiv 1 \mod 4 \\
H_1(B,\Z_{\eps}) \simeq  \Z_3 \oplus \Z_3 \oplus \Z_2 & \mbox{ for } r=3, e \equiv 0 \mod 4 \\
H_1(B,\Z_{\eps}) \simeq \Z_6 & \mbox{ for } r=3, e \equiv 1 \mod 4  \\
\end{array}
$$
\end{prop}

Finally, for groups of type $G(e,e,r)$, we use the Dehornoy-Lafont complex
associated to the Corran-Picantin monoid. The 1-cells $[s]$ are mapped to $(\eps(s) - 1)[\emptyset]
= -2 [\emptyset]$, hence the 1-cycles are spanned by the $[s]-[t]$ for $s,t$
two atoms. We have $d_{\eps}[s_j,s_i] = 2(s_j-s_i)$ when $|j-i| \geq 2$,
$d_{\eps}[t_0,t_i] = - t_i + t_{i+1} + t_0 - t_1$,
$d_{\eps}[s_3,t_i] = 3t_i - 3 s_3$, $d_{\eps}[s_i,t_0] = 2(s_i-t_0)$ for $i \geq 4$, and $d_{\eps} [s_{i+1},s_i] = 3 (s_{i+1}-s_i)$.
Since a basis of the 1-cycles is given by the $t_i - t_0$, $t_0-s_3$, $s_3-s_4$,
\dots, $s_{r-1}-s_r$, $H_1(B,\Z_{\eps})$ is spanned by $t_1-t_0, t_0-s_3,
\dots,s_{r-1}-s_r$, each of these elements being annihilated by 3. From
$d_{\eps}[s_i,t_0] = 2(s_i-t_0)$ for $i \geq 4$ we get that $s_4 - s_3 \equiv
s_3 - t_0$, from $d_{\eps}[s_5,s_3] = 2(s_5-s_3)$ we get
$s_5 - s_4 \equiv s_4 - s_3$, and so on. Finally, from
$$d_{\eps}[t_1,s_4] = 2(t_1-s_4) = 2(t_1 - t_0) + 2(t_0-s_3) + 2(s_3-s_4)
\equiv 2(t-1-t_0) +(t_0-s_3)$$
we get that $t_1 - t_0 \equiv t_0-s_3$.
It follows that, for $r \geq 4$, $H_1(B,\Z_{\eps})$
is generated by $t_1-t_0$ hence $H_1(B,\Z_{\eps}) \simeq \Z_3$ ; for $r = 3$, it is generated
by $t_1 - t_0$ and $t_0 - s_3$ and $H_1(B,\Z_{\eps}) \simeq \Z_3\oplus \Z_3$ ; it is generated
by $t_1-t_0$ for $r =2$.

The case $e=1$ (that is, of the usual braid group) follows the same
pattern. On the whole, we get the following.

\begin{prop} For the groups $\B(e,e,r)$ with $e \geq 2$, $H_1(B,\Z_{\eps}) \simeq \Z_3$ if $r \geq 4$.
If $r = 3$ then $H_1(B,\Z_{\eps}) \simeq \Z_3 \oplus \Z_3$.
If $r=2$ then $H_1(B,\Z_{\eps}) \simeq \Z$. When $e=1$, we have
$H_1(B,\Z_{\eps}) = 0$ for $r = 2$ or $r \geq 5$, and $H_1(B,\Z_{\eps})=
\Z_3$ if $r=3$ or $r = 4$.
\end{prop}

\bigskip

\begin{figure}
$$
\begin{array}{|c|l|}
ST & \mbox{Presentations for the group of even braids} \\
\hline
12 & vbu=awv,uaw=vbua,vbu=buaw,uaw=wvb\\
13 & buaw=awv,wvb=vbua,vbua=buaw,uawv=wvbu\\
22 & vbua=awvb,uawv=vbuaw,vbua=buaw,uawv=wvbu\\
23 & auaua=vv,vvv=uauau,bu=w,w=ub,bvb=awa,waw=vbv\\ 
24 & vv=auau,uaua=vv,awaw=bvbv,vbvb=wawa,w=bub,ubu=ww, \\
 & auawvb=vbuaw,vvbuaw=uawvbu\\
27 & w=bub,ubu=ww,vv=auau,uaua=vv,awawa=bvbvb,vbvbv=wawaw, \\ & 
bvwaua=waubv,waubvv=ubvwau\\ 
28 & aua=v,vv=uau,bu=w,w=ub,bvbv=awaw,wawa=vbvb,cu=x,x=uc,cv=ax, \\ & 
cv=ax, xa=vc,cwc=bxb,xbx=wcw\\ 
29 & v=aua,uau=vv,axa=cvc,vcv=xax,bxb=cwc,wcw=xbx,awaw=bvbv, \\ & 
vbvb=wawa,w=bu,ub=w,x=cu,uc=x,cwaxbv=bvcwax,xbvcwa=waxbvc\\ 
30 & auaua=vv,vvv=uauau,bu=w,w=ub,bvb=awa,waw=vbv,cu=x,x=uc, \\ & 
cv=ax, xa=vc,cwc=bxb,xbx=wcw\\ 
31^* & x=cuc,ucu=xx,axa=cvc,vcv=xax,dwd=byb,yby=wdw,aya=dvd,vdv=yay, \\ & 
vb=aw,uaw=vbu,aw=bua,vbu=wv,y=du,ud=y,bx=cw,wc=xb,dx=cy,\\
 & yc=xd \\
33 & v=aua,uau=vv,bvb=awa,waw=vbv,cvc=axa,xax=vcv,cwc=bxb,xbx=wcw,\\ &
cyc=dxd,xdx=ycy,w=bu,ub=w,x=cu,uc=x,y=du,ud=y,ay=dv,vd=ya, \\ & 
by=dw,wd=yb,cvbxaw=bxawcv,xawcvb=wcvbxa\\ 
34 & v=aua,uau=vv,bvb=awa,waw=vbv,cvc=axa,xax=vcv,cwc=bxb,xbx=wcw, \\ & 
cyc=dxd,xdx=ycy,w=bu,ub=w,x=cu,uc=x,y=du,ud=y,ay=dv,vd=ya, \\ & 
by=dw,wd=yb,dzd=eye,yey=zdz,z=eu,ue=z,az=ev,ve=za,bz=ew,\\ & 
we=zb,cz=ex,xe=zc,cvbxaw=bxawcv,xawcvb=wcvbxa\\
 35 & au=v,v=ua,bub=w,ww=ubu,bv=aw,wa=vb,cu=x,x=uc,cvc=axa, \\ & 
xax=vcv,cwc=bxb,xbx=wcw,du=y,y=ud,dv=ay,ya=vd,dw=by,yb=wd, \\ & 
dxd=cyc,ycy=xdx,eu=z,z=ue,ev=az,za=ve,ew=bz,zb=we,\\ 
 & ex=cz,zc=xe,eye=dzd,zdz=yey \\
36 & au=v,v=ua,bub=w,ww=ubu,bv=aw,wa=vb,cu=x,x=uc,cvc=axa,\\ & 
xax=vcv,cwc=bxb,xbx=wcw,du=y,y=ud,dv=ay,ya=vd,dw=by, \\ & 
yb=wd,dxd=cyc,ycy=xdx,eu=z,z=ue,ev=az,za=ve,ew=bz,zb=we, \\ & 
ex=cz,zc=xe,eye=dzd,zdz=yey,fu=x_2,x_2=uf,fv=ax_2,x_2a=vf,\\ &
fw=bx_2,x_2b=wf,fx=cx_2,x_2c=xf,fy=dx_2,x_2d=yf,fzf=ex_2e,x_2ex_2=zfz\\ 
37 & au=v,v=ua,bub=w,ww=ubu,bv=aw,wa=vb,cu=x,x=uc,cvc=axa, \\ & 
xax=vcv,cwc=bxb,xbx=wcw,du=y,y=ud,dv=ay,ya=vd,dw=by,\\ & 
yb=wd,dxd=cyc,ycy=xdx,eu=z,z=ue,ev=az,za=ve,ew=bz,zb=we,\\ & 
ex=cz,zc=xe,eye=dzd,zdz=yey,fu=x_2, x_2=uf,fv=ax_2,x_2a=vf, \\ &
fw=bx_2,x_2b=wf,fx=cx_2,x_2c=xf,fy=dx_2,x_2d=yf,fzf=ex_2e,\\ & 
x_2ex_2=zfz,gu=y_2,y_2=ug,gv=ay_2,y_2a=vg,gw=by_2,y_2b=wg,gx=cy_2,\\ 
 & y_2c=xg,gy=dy_2,y_2d=yg,gz=ey_2,y_2e=zg,gx_2g=fy_2f,y_2fy_2=x_2gx_2 \\
\end{array}
$$
\noindent $^*$ Provided that the presentation of $B$ suggested in \cite{BMR} for $G_{31}$ is correct.
\caption{Presentations for even braid groups of exceptional types}
\label{figpreseven}
\end{figure}

\begin{table}
$$
\begin{array}{|c|ccccccc|}
 & H_0 & H_1 & H_2 & H_3 & H_4 & H_5 & H_6 \\
\hline
G_{12} & \Z & \Z & 0 & & & & \\ 
G_{13} & \Z & \Z^2 & \Z & & & & \\ 
G_{22} & \Z & \Z & 0 & & & & \\
G_{24} & \Z & \Z & \Z & \Z & & & \\
G_{27} & \Z & \Z & \Z_3 \times \Z & \Z & & & \\
G_{29} & \Z & \Z & \Z_2 \times \Z_4 & \Z_2 \times \Z & \Z & & \\
G_{31} & \Z & \Z & \Z_6 & \Z & \Z & & \\
G_{33} & \Z & \Z & \Z_6 & \Z_6 & \Z & \Z & \\
G_{34} & \Z & \Z & \Z_6 & ? & ? & ? & ? \\
\end{array}
$$
\caption{Homology of exceptional groups}\label{tableexc}
\end{table}

\begin{table}
$$
\begin{array}{|l|ccccccccc|}
 & H_0 & H_1 & H_2 & H_3 & H_4 & H_5 & H_6 & H_7 & H_8\\
\hline
I_2(2m) & \Z & \Z^2 & \Z & & & & & & \\ 
I_2(2m+1) & \Z & \Z & 0  & & & & & & \\
H_3 = G_{23} & \Z & \Z & \Z & \Z & & & & & \\
H_4 = G_{30} & \Z & \Z & \Z_2   & \Z & \Z & & & & \\
F_4 = G_{28} & \Z & \Z^2 & \Z^2 & \Z^2 & \Z & & & & \\
E_6=G_{35} & \Z & \Z & \Z_2 & \Z_2 & \Z_6 & \Z_3 & 0 & & \\
E_7 = G_{36} & \Z & \Z & \Z_2 & \Z_2^2 & \Z_6^2 & \Z_3 \times \Z_6 & \Z & \Z & \\
E_8 = G_{37} & \Z & \Z & \Z_2 & \Z_2 & \Z_2 \times \Z_6 & \Z_3 \times \Z_6 & \Z_2 \times \Z_6 & \Z & \Z \\
\end{array}
$$
\caption{Homology of exceptional Artin groups (after Salvetti)}\label{tableexcCox}
\end{table}

\begin{table}
$$
\begin{array}{|c|c|c|}
ST & B^{(2)}_{ab} & H_1(B,\Z_{\eps}) \\
\hline
 12 & \Z_3\times \Z & \Z_3 \\ 13 & \Z\times \Z & \Z_2 \\ 22 & \Z & 0 \\  
  23 & \Z & 0 \\ 24 & \Z & 0 \\ 27 & \Z & 0 \\ 28 & \Z\times \Z & \Z_2 \\ 
  29 & \Z & 0 \\ 30 & \Z & 0 \\ 31^* & \Z & 0 \\ 33 & \Z & 0 \\ 34 & \Z & 0 \\ 
  35 & \Z & 0 \\ 36 & \Z & 0 \\ 37 & \Z & 0 \\ 
\end{array}
$$
\caption{Abelianization of even braids and $H_1(B,\Z_{\eps})$}
\end{table}

\cleardoublepage
\bibliographystyle{amsalpha}
\bibliography{biblio}






\end{document}